\documentclass[smallextended]{svjour3}
\smartqed  
\usepackage{amsfonts}
\usepackage{latexsym}
\usepackage{amsmath, amssymb}

\usepackage{color}
\usepackage{graphicx}
\usepackage{epstopdf}
\usepackage{subfigure}

\usepackage{dsfont}
\usepackage{bbm}

\newtheorem{assumption}{Assumption}

\newcommand{\mx}{\mathbf{x}}
\newcommand{\mW}{\mathbf{W}}
\newcommand{\mI}{\mathbf{I}}
\newcommand{\my}{\mathbf{y}}

\newcommand{\ox}{\overline{x}}
\newcommand{\oy}{\overline{y}}

\newcommand{\bE}{\mathbb{E}}
\newcommand{\mA}{\mathbf{A}}
\newcommand{\mB}{\mathbf{B}}
\newcommand{\mM}{\mathbf{M}}
\newcommand{\mS}{\mathbf{S}}
\newcommand{\mPi}{\mathbf{\Pi}}

\newcommand{\trace}{\mathbf{tr}}
\newcommand{\T}{\intercal}
\newcommand{\degree}{\mathrm{deg}}

\newcommand{\ik}{\mathds{1}_k}

\title{
Distributed Stochastic Gradient Tracking Methods\footnote{Parts of the results appeared in the 57th IEEE Conference on Decision and Control \cite{pu2018distributed}.} 
\thanks{This work was supported in parts by the NSF grant CCF-1717391, the ONR grant no.\ N00014-16-1-2245, and the Shenzhen Research Institute of Big Data Startup Fund JCYJ-SP2019090001.}}

\author{
    Shi Pu
    \and
    Angelia Nedi{\'c}
}

\institute{S. Pu \at
	 Shenzhen Research Institute of Big Data and Institute for Data and Decision Analytics, The Chinese University of Hong Kong, Shenzhen, China.\\\
	\email{pushi@cuhk.edu.cn}\\
	ORCID ID: 0000-0002-5813-527X
	\and
	A. Nedi{\'c} \at
	School of Electrical, Computer, and Energy Engineering, Arizona
	State University, Tempe, AZ 85287.\\
	\email{Angelia.Nedich@asu.edu}
}

\begin{document}

\maketitle

\begin{abstract}
	In this paper, we study the problem of distributed multi-agent optimization over a network, where each agent possesses a local cost function that is smooth and strongly convex. The global objective is to find a common solution that minimizes the average of all cost functions. Assuming agents only have access to unbiased estimates of the gradients of their local cost functions, we consider a distributed stochastic gradient tracking method (DSGT) and a gossip-like stochastic gradient tracking method (GSGT). We show that, in expectation, the iterates generated by each agent are attracted to a neighborhood of the optimal solution, where they accumulate exponentially fast (under a constant stepsize choice). 
	Under DSGT, the limiting (expected) error bounds on the distance of the iterates from the optimal solution decrease with the network size $n$, which is a comparable performance to a centralized stochastic gradient algorithm. Moreover, we show that when the network is well-connected, GSGT incurs lower communication cost than DSGT while maintaining a similar computational cost.
	Numerical example further demonstrates the effectiveness of the proposed methods.
	
	\vspace{.1in}
	

	\vspace{.07in}
	
\end{abstract}

\section{Introduction}
Consider a set of agents $\mathcal{N}=\{1,2,\ldots,n\}$ connected over a network. Each agent has a local smooth and strongly convex cost function $f_i:\mathbb{R}^{p}\rightarrow \mathbb{R}$. The global objective is to locate $x\in\mathbb{R}^p$ that minimizes the average of all cost functions:
\begin{equation}
\min_{x\in \mathbb{R}^{p}}f(x)\left(=\frac{1}{n}\sum_{i=1}^{n}f_i(x)\right).  \label{opt Problem_def}
\end{equation}%
Scenarios in which problem (\ref{opt Problem_def}) is considered include distributed machine learning \cite{forrester2007multi,nedic2017fast,cohen2017projected}, multi-agent target seeking \cite{pu2016noise,chen2012diffusion}, and wireless networks \cite{cohen2017distributed,mateos2012distributed,baingana2014proximal}, among many others.

To solve problem (\ref{opt Problem_def}), we assume each agent $i$ queries a stochastic oracle ($\mathcal{SO}$) to obtain noisy gradient samples of the form $g_i(x,\xi_i)$ that satisfies the following condition:
\begin{assumption}
	\label{asp: gradient samples}
	For all $i\in\mathcal{N}$ and all $x\in\mathbb{R}^p$, 
	each random vector $\xi_i\in\mathbb{R}^m$ is independent, and
	\begin{equation}
	\begin{split}
	& \mathbb{E}_{\xi_i}[g_i(x,\xi_i)\mid x] =  \nabla f_i(x),\\
	& \mathbb{E}_{\xi_i}[\|g_i(x,\xi_i)-\nabla f_i(x)\|^2\mid x]  \le  \sigma^2\quad\hbox{\ for some $\sigma>0$}.
	\end{split}
	\end{equation}
\end{assumption}
The above assumption of stochastic gradients holds true for many on-line distributed learning problems, where $f_i(x)=\bE_{\xi_i}[F_i(x,\xi_i)]$ denotes the expected loss function agent $i$ wishes to minimize, while independent samples $\xi_i$ are gathered continuously over time.
For another example, in simulation-based optimization, the gradient estimation often incurs noise 
that can be due to various sources, such as modeling and discretization errors, 
incomplete convergence, and finite sample size for Monte-Carlo methods~\cite{kleijnen2008design}. 

Distributed algorithms dealing with problem (\ref{opt Problem_def}) have been studied extensively in the literature \cite{tsitsiklis1986distributed,nedic2009distributed,nedic2010constrained,lobel2011distributed,jakovetic2014fast,kia2015distributed,shi2015extra,di2016next,qu2017harnessing,nedic2017achieving,pu2020push}.
Recently, there has been considerable interest in distributed implementation of stochastic gradient algorithms \cite{ram2010distributed,srivastava2011distributed,duchi2012dual,bianchi2013convergence,cavalcante2013distributed,towfic2014adaptive,chatzipanagiotis2015augmented,chen2015learning,chen2015learning2,chatzipanagiotis2016distributed,nedic2016stochastic,lan2017communication,lian2017can,pu2017flocking,sayin2017stochastic,pu2018swarming,jakovetic2018convergence}. The literature has shown that distributed methods may compete with, or even outperform, their centralized counterparts under certain conditions \cite{chen2015learning,chen2015learning2,lian2017can,pu2017flocking,pu2018swarming}. For instance, in our recent work~\cite{pu2017flocking}, we proposed a flocking-based approach for distributed stochastic optimization which beats a centralized gradient method in real-time assuming that all $f_i$ are identical. 
However, to the best of our knowledge, there is no distributed stochastic gradient method addressing problem (\ref{opt Problem_def}) that shows comparable performance with a centralized approach  for optimizing the sum of smooth and strongly convex objective functions under Assumption \ref{asp: gradient samples} only. In particular, under constant stepsize policies none of the existing algorithms achieve an error bound that is decreasing in the network size $n$. 


A distributed gradient tracking method was proposed in~\cite{di2016next,nedic2017achieving,qu2017harnessing}, where 
the agent-based auxiliary variables $y_i$ were introduced to track the average gradients of $f_i$ assuming accurate gradient information is available. It was shown that the method, with constant stepsize, generates iterates that converge linearly to the optimal solution.
Inspired by the approach, in this paper we consider a distributed stochastic gradient tracking method (DSGT). By comparison, in our proposed algorithm $y_i$ are tracking the stochastic gradient averages of $f_i$.
We are able to show that the iterates generated by each agent reach, in expectation, a neighborhood of the optimal point exponentially fast under a constant stepsize.
Interestingly, with a sufficiently small stepsize, the limiting error bounds on the distance 
between the agent iterates and the optimal solution decrease in the network size, which is comparable to the performance of a centralized stochastic gradient algorithm. 

Gossip-based communication protocols are popular choices for distributed computation due to their low communication costs \cite{boyd2006randomized,lu2011gossip,lee2016asynchronous,mathkar2016nonlinear}.
In the second part of this paper, we consider a gossip-like stochastic gradient tracking algorithm (GSGT) where at each iteration, an agent wakes up uniformly randomly and communicates with one of her neighbors or updates by herself. Similar to DSGT, the method produces iterates that converge to a neighborhood of the optimal point exponentially fast under a sufficiently small constant stepsize. When the network of agents is well-connected (e.g., complete network, almost all regular graphs), GSGT is shown to employ a lower communication burden and similar computational cost when compared to DSGT.

\subsection{Related Work}
We now briefly review the literature on (distributed) stochastic optimization.
First of all, our work is related to the extensive literature in stochastic approximation (SA) methods dating back to the seminal works~\cite{robbins1951stochastic} and~\cite{kiefer1952stochastic}. These works include the analysis of convergence (conditions for convergence, rates of convergence, suitable choice of stepsize) in the context of diverse noise models~\cite{kushner2003stochastic}.
Assuming the objective function $f$ is strongly convex with Lipschitz continuous gradients, the optimal rate of convergence for solving problem (\ref{opt Problem_def}) has been shown to be $\mathcal{O}\left(1/k\right)$ under a diminishing SA stepsize where $k$ denotes the iteration number \cite{nemirovski2009robust}. With a constant stepsize $\alpha>0$ that is sufficiently small, the iterates generated by a stochastic gradient method is attracted to an $\mathcal{O}(\alpha)$-neighborhood of the optimal solution exponentially fast (in expectation).

Distributed implementations of stochastic gradient methods have become increasingly popular in recent years. In \cite{ram2010distributed}, the authors considered minimizing a sum of (possibly nonsmooth) convex objective functions subject to a common convex constraint set. It was shown that when the means of the stochastic subgradient errors diminish, there is mean consensus among the agents and mean convergence to the optimum function value under SA stepsizes.
The work \cite{srivastava2011distributed} used two diminishing stepsizes to deal with communication noises and subgradient errors, respectively. Asymptotic convergence to the optimal set was established; for constant stepsizes asymptotic error bounds were derived.  
In \cite{duchi2012dual}, a distributed dual averaging method was proposed for minimizing (possibly nonsmooth) convex functions. Under a carefully chosen SA stepsize sequence, the method exhibits the convergence rate $\mathcal{O}(\frac{n\log(k)}{(1-\lambda_2(\mW))\sqrt{k}})$, in which $\lambda_2(\mW)$ denotes the second largest singular value of the doubly stochastic mixing matrix $\mW$.
Paper \cite{bianchi2013convergence} considered a projected stochastic gradient algorithm for solving non-convex optimization problems by combining a local stochastic gradient update and a gossip step. It was proved that consensus is asymptotically
achieved in the network and the solutions converge to the
set of KKT points with SA stepsizes.
A distributed online algorithm was devised and analyzed in \cite{cavalcante2013distributed}  for solving dynamic
optimization problems in noisy communication environments. Sufficient conditions were provided for almost sure
convergence of the algorithm. 
In \cite{towfic2014adaptive}, the authors proposed an adaptive diffusion algorithm based on penalty methods. Under a constant stepsize $\alpha$, it was shown that the expected distance between the
optimal solution and that obtained at each node is bounded by $\mathcal{O}(\alpha)$. The work \cite{chen2015learning,chen2015learning2} further showed that under a sufficiently small
stepsize and certain conditions on the stochastic gradients, distributed methods are able to achieve
the same performance level as that of a centralized approach.
Paper \cite{chatzipanagiotis2016distributed} considered the problem of distributed constrained convex optimization subject to multiple noise terms in
both computation and communication stages.
The authors utilized an augmented Lagrangian framework and established the almost sure convergence of the algorithm under a diminishing stepsize policy. In \cite{nedic2016stochastic}, a subgradient-push method was investigated for distributed optimization
over time-varying directed graphs. When the objective function is strongly convex, the scheme exhibits the $\mathcal{O}(\frac{\ln k}{k})$ rate of convergence.

In a recent work \cite{lan2017communication}, a class of decentralized first-order methods for nonsmooth and stochastic optimization was presented. The class was shown to exhibit the $\mathcal{O}(\frac{1}{k})$ (respectively, $\mathcal{O}(\frac{1}{\sqrt{k}})$) rate of convergence for minimizing the sum of strongly convex functions (respectively, general convex functions).
The work in \cite{aybat2017distributed} considered a composite convex optimization problem with noisy gradient information and showed the $O(\frac{1}{\sqrt{k}})$ convergence rate using an ADMM-based approach.
Paper \cite{lian2017can} considered a decentralized stochastic gradient algorithm that achieves the $\mathcal{O}(\frac{1}{k}+\frac{1}{\sqrt{nk}})$ rate of convergence for minimizing the sum of non-convex functions. The rate is comparable to that of a centralized algorithm when $k$ is large enough. At the same time, the communication cost for the decentralized approach is lower. Papers \cite{pu2017flocking,pu2018swarming} also demonstrates the advantage of distributively implementing a stochastic gradient method assuming that all $f_i$ are identical and sampling times are random and non-negligible.
The work \cite{sayin2017stochastic} utilized a time-dependent weighted mixing of stochastic subgradient updates to achieve the convergence rate of $\mathcal{O}(\frac{n\sqrt{n}}{(1-\lambda_2(\mW))k})$ for minimizing the sum of (possibly nonsmooth) strongly convex functions. 
In \cite{sirb2018decentralized}, the authors considered a decentralized consensus-based algorithm with delayed gradient information. The method was shown to achieve the optimal $\mathcal{O}(\frac{1}{\sqrt{k}})$ rate of convergence for general convex functions. In \cite{jakovetic2018convergence}, the $\mathcal{O}(\frac{1}{k})$ convergence rate was established for strongly convex costs and random networks. 


\subsection{Main Contribution}
Our main contribution is summarized as follows.
Firstly, we propose a novel distributed stochastic gradient tracking method (DSGT) for optimizing the sum of smooth and strongly convex objective functions. We employ an auxiliary variable $y_i$ for each agent that tracks the average stochastic gradients of the cost functions.
We show that, under a constant stepsize choice, the algorithm is comparable to a centralized stochastic gradient scheme in terms of their convergence speeds and the ultimate error bounds. In particular, the obtained error bound under DSGT decreases with the network size $n$, which has not been shown in the literature to the best of our knowledge. Moreover, assuming the gradient estimates are accurate, DSGT recovers the linear rate of convergence to the optimal solution \cite{nedic2017achieving,qu2017harnessing}, which is also a unique feature among other distributed stochastic gradient algorithms.

Secondly, with an SA stepsize $\alpha_k\rightarrow 0$, DSGT enjoys the optimal $\mathcal{O}(\frac{1}{k})$ rate of convergence to the optimal point. In addition, we characterize the dependency of the constant factors in the stepsize and the convergence rate on the properties of the mixing matrix as well as the characteristics of the objective functions, such as the strong convexity factor and the Lipschitz constant.

Thirdly, we introduce a gossip-like stochastic gradient tracking method that is efficient in communication. We show that, under a sufficiently small constant stepsize, GSGT also produces iterates that converge to a neighborhood of the optimal point exponentially fast. Again, when the gradient estimates are accurate, GSGT recovers the linear rate of convergence to the optimal solution.
Compared to DSGT, we show that when the network is well-connected (e.g., complete network, almost all regular graphs), GSGT incurs lower communication cost than DSGT by a factor of $\mathcal{O}(\frac{|\mathcal{E}|}{n})$ ($|\mathcal{E}|$ denoting the number of edges in the network) while maintaining a similar computational cost.

Finally, we provide a numerical example that demonstrates the effectiveness of the proposed methods when contrasted with the centralized stochastic gradient algorithm and an existing variant of distributed stochastic gradient method.

\subsection{Notation and Assumptions}
\label{subsec:pre}
Throughout the paper, vectors default to columns if not otherwise specified.
Let each agent $i$ hold a local copy $x_i\in\mathbb{R}^p$ of the decision variable and an auxiliary variable $y_i\in\mathbb{R}^p$. Their values at iteration/time $k$ are denoted by $x_{i,k}$ and $y_{i,k}$, respectively. 
We let
\begin{equation*}
\mx := [x_1, x_2, \ldots, x_n]^{\T}\in\mathbb{R}^{n\times p},\ \ \my := [y_1, y_2, \ldots, y_n]^{\T}\in\mathbb{R}^{n\times p},
\end{equation*}
and
\begin{equation*}
\ox :=  \frac{1}{n}\mathbf{1}^{\T} \mx\in\mathbb{R}^{1\times p},\ \ \oy :=  \frac{1}{n}\mathbf{1}^{\T}\my\in\mathbb{R}^{1\times p},
\end{equation*}
where $\mathbf{1}$ denotes the vector with all entries equal to 1.
We define an aggregate objective function of the local variables:
\begin{equation}
F(\mx):=\sum_{i=1}^nf_i(x_i),
\end{equation}
and let
\begin{equation*}
\nabla F(\mx):=\left[\nabla f_1(x_1), \nabla f_2(x_2), \ldots, \nabla f_n(x_n)\right]^{\T}\in\mathbb{R}^{n\times p}.
\end{equation*}
In addition, denote
\begin{equation*}
\begin{split}
& h(\mx):= \frac{1}{n}\mathbf{1}^{\T}\nabla F(\mx)\in\mathbb{R}^{1\times p},\\
& \boldsymbol{\xi} := [\xi_1, \xi_2, \ldots, \xi_n]^{\T}\in\mathbb{R}^{n\times m},\\
& G(\mx,\boldsymbol{\xi}):= [g_1(x_1,\xi_1), g_2(x_2,\xi_2), \ldots, g_n(x_n,\xi_n)]^{\T}\in\mathbb{R}^{n\times p}.
\end{split}
\end{equation*}

The inner product of two vectors $a,b$ of the same dimension is denoted by $\langle a,b\rangle$. For two matrices $\mA,\mB\in\mathbb{R}^{n\times p}$, we let $\langle \mA,\mB\rangle$ be the Frobenius inner product.
We use $\|\cdot\|$ to denote the $2$-norm of vectors; for matrices, $\|\cdot\|$ represents the Frobenius norm. The spectral radius of a square matrix $\mM$ is denoted by $\rho(\mM)$.

We make the following standing assumption on the individual objective functions $f_i$.
\begin{assumption}
	\label{asp: strconvexity}
	Each $f_i:\mathbb{R}^p\rightarrow \mathbb{R}$ is  $\mu$-strongly convex with $L$-Lipschitz continuous gradients, i.e., for any $x,x'\in\mathbb{R}^p$,
	\begin{equation*}
	\begin{split}
	& \langle \nabla f_i(x)-\nabla f_i(x'),x-x'\rangle\ge \mu\|x-x'\|^2,\\
	& \|\nabla f_i(x)-\nabla f_i(x')\|\le L \|x-x'\|.
	\end{split}
	\end{equation*}
\end{assumption}
We note that, under Assumption~\ref{asp: strconvexity}, problem (\ref{opt Problem_def}) has a unique solution denoted by $x^*\in\mathbb{R}^{1\times p}$.

A graph is a pair $\mathcal{G}=(\mathcal{V},\mathcal{E})$ where $\mathcal{V}=\{1,2,\ldots,n\}$ is the set of vertices (nodes) and $\mathcal{E}\subseteq \mathcal{V}\times \mathcal{V}$ represents the set of edges connecting vertices. We assume agents communicate in an undirected graph, i.e., $(i,j)\in\mathcal{E}$ iff (if and only if) $(j,i)\in\mathcal{E}$.
For each agent $i$, let $\mathcal{N}_i=\{j\mid j\neq i, (i,j)\in\mathcal{E}\}$ be its set of neighbors. The cardinality of $\mathcal{N}_i$, denoted by $\deg(i)$, is referred to as agent $i$'s degree.
We consider the following condition regarding the interaction graph of agents.
\begin{assumption}
	\label{asp: network}
	The graph $\mathcal{G}$ corresponding to the network of agents is undirected and connected, i.e., there exists a path between any two agents.
\end{assumption}

\subsection{Organization of the Paper}
The paper is organized as follows. In Section~\ref{sec: set}, we introduce the distributed stochastic gradient tracking method and present its main convergence results. We perform analysis in Section~\ref{sec:cohesion}. In Section \ref{sec: gossip} we propose the gossip-like stochastic gradient tracking method. A numerical example is provided in Section~\ref{sec: simulation} to illustrate our theoretical findings. 
Section~\ref{sec: conclusion} concludes the paper.

\section{A Distributed Stochastic Gradient Tracking Method (DSGT)}
\label{sec: set}
We consider the following distributed stochastic gradient tracking method: At each step $k\in\mathbb{N}$, 
every agent $i$ independently implements the following two steps:
\begin{equation}
\label{eq: x_i,k}
\begin{split}
x_{i,k+1} = & \sum_{j=1}^{n}w_{ij}(x_{j,k}-\alpha y_{j,k}), \\
y_{i,k+1} = & \sum_{j=1}^{n}w_{ij}y_{j,k}+g_i(x_{i,k+1},\xi_{i,k+1})-g_i(x_{i,k},\xi_{i,k}),
\end{split}
\end{equation}
where $w_{ij}$ are nonnegative weights and $\alpha>0$ is a constant stepsize. Agent $i$ and $j$  are connected iff $w_{ij}, w_{ji}>0$. The iterates are initiated with an arbitrary 
$x_{i,0}$ and $y_{i,0}= g_i(x_{i,0},\xi_{i,0})$ for all~$i\in{\cal N}$.
We can also write (\ref{eq: x_i,k}) in the following compact form:
\begin{equation}
\label{eq: x_k}
\begin{split}
\mx_{k+1} = & \mW(\mx_k-\alpha \my_k), \\
\my_{k+1} = & \mW\my_k+G(\mx_{k+1},\boldsymbol{\xi}_{k+1})-G(\mx_k,\boldsymbol{\xi}_k),
\end{split}
\end{equation}
where $\mW=[w_{ij}]\in\mathbb{R}^{n\times n}$ denotes the coupling matrix of agents. We assume that $\mW$ satisfies the following condition.
\begin{assumption}
	\label{asp: W}
	Nonnegative coupling matrix $\mW$ is doubly stochastic, 
	i.e., $\mW\mathbf{1}=\mathbf{1}$ and $\mathbf{1}^{\T}\mW=\mathbf{1}^{\T}$. 
	In addition, $w_{ii}>0$ for some $i\in\mathcal{N}$.
\end{assumption}
In the subsequent analysis, we will frequently use the following result, which is a direct implication of Assumption \ref{asp: network} and Assumption \ref{asp: W} (see \cite{qu2017harnessing} Section II-B).
\begin{lemma}
	\label{lem: spectral norm}
	Let Assumption \ref{asp: network} and Assumption \ref{asp: W} hold, and let $\rho_w$ denote the spectral norm of 
	the matrix $\mW-\frac{1}{n}\mathbf{1}\mathbf{1}^{\T}$. Then, $\rho_w<1$ and 
	\begin{equation*}
	\|\mW\omega-\mathbf{1}\overline{\omega}\|\le \rho_w\|\omega-\mathbf{1}\overline{\omega}\|
	\end{equation*}
	for all $\omega\in\mathbb{R}^{n\times p}$, where $\overline{\omega}=\frac{1}{n}\mathbf{1}^{\T}\omega$.
\end{lemma}

Algorithm (\ref{eq: x_i,k}) is closely related to the schemes considered in \cite{di2016next,nedic2017achieving,qu2017harnessing}, in which auxiliary variables $y_{i,k}$ 
were introduced to track the average $\frac{1}{n}\sum_{i=1}^{n}\nabla f_i(x_{i,k})$. This design ensures that the algorithms achieve linear convergence under a constant stepsize choice.
Correspondingly, under our approach $y_{i,k}$ are (approximately) tracking $\frac{1}{n}\sum_{i=1}^{n}g_i(x_{i,k},\xi_{i,k})$.
To see why this is the case, note that $\oy_k = \frac{1}{n}\mathbf{1}^{\T} \my_k$.
Since $y_{i,0}=g(x_{i,0},\xi_{i,0}),\forall i$, by induction we have
\begin{equation}
\oy_k=\frac{1}{n}\mathbf{1}^{\T}G(\mx_k,\boldsymbol{\xi}_k)
=\frac{1}{n}\sum_{i=1}^{n}g_i(x_{i,k},\xi_{i,k}),\forall k.
\end{equation}
It will be shown that $\my_k$ is close to $\mathbf{1}\oy_k$ in expectation when $k$ is sufficiently large. 
As a result, $y_{i,k}$ are (approximately) tracking $\frac{1}{n}\sum_{i=1}^{n}g_i(x_{i,k},\xi_{i,k})$.

It is worth noting that compared to the standard distributed subgradient methods \cite{nedic2009distributed}, DSGT incurs two times the communication and storage costs per iteration.

\subsection{Main Results}
\label{subsec: main}
Main convergence properties of DSGT are covered in the following theorem.
\begin{theorem}
	\label{Theorem1}
	Let $\Gamma>1$ be arbitrarily chosen. Suppose Assumptions \ref{asp: gradient samples}-\ref{asp: strconvexity} hold and 
	the stepsize $\alpha$ satisfies
	\begin{equation}
	\label{alpha_ultimate_bound}
	\alpha\le \min\left\{\frac{(1-\rho_w^2)}{12\rho_w L},\frac{(1-\rho_w^2)^2}{2\sqrt{\Gamma}L\max\{6\rho_w\|\mW-\mI\|,1-\rho_w^2\}}, \frac{(1-\rho_w^2)}{3\rho_w^{2/3}L}\left[\frac{\mu^2}{L^2}\frac{(\Gamma-1)}{\Gamma(\Gamma+1)}\right]^{1/3}\right\}.
	\end{equation}
	Then $\sup_{l\ge k}\bE[\|\ox_l-x^*\|^2]$ and $\sup_{l\ge k}\bE[\|\mx_{l}-\mathbf{1}\ox_{l}\|^2]$, respectively, converge to $\limsup_{k\rightarrow\infty}\bE[\|\ox_k-x^*\|^2]$ and $\limsup_{k\rightarrow\infty}\bE[\|\mx_k-\mathbf{1}\ox_k\|^2]$ at the linear rate $\mathcal{O}(\rho(\mA)^k)$, where $\rho(\mA)<1$ is the spectral radius of the matrix $\mA$ given by
	\begin{eqnarray*}
		\mA=\begin{bmatrix}
			1-\alpha\mu & \frac{\alpha L^2}{\mu n}(1+\alpha\mu) & 0\\
			0 & \frac{1}{2}(1+\rho_w^2) & \alpha^2\frac{(1+\rho_w^2)\rho_w^2}{(1-\rho_w^2)}\\
			2\alpha nL^3 & \left(\frac{1}{\beta}+2\right)\|\mW-\mI\|^2 L^2+3\alpha L^3 & \frac{1}{2}(1+\rho_w^2)
		\end{bmatrix},
	\end{eqnarray*}
	where $\beta=\frac{1-\rho_w^2}{2\rho_w^2}-4\alpha L-2\alpha^2 L^2$.
	Furthermore,
	\begin{equation}
	\label{error_bound_ultimate}
	\limsup_{k\rightarrow\infty}\bE[\|\ox_k-x^*\|^2]\le \frac{(\Gamma+1)}{\Gamma}\frac{\alpha\sigma^2}{\mu n}
	+\left(\frac{\Gamma+1}{\Gamma-1}\right)\frac{4\alpha^2 L^2(1+\alpha\mu)(1+\rho_w^2)\rho_w^2}{\mu^2 n(1-\rho_w^2)^3}M_{\sigma},
	\end{equation}
	and
	\begin{equation}
	\label{consensus_error_bound_ultimate}
	\limsup_{k\rightarrow\infty}\bE[\|\mx_k-\mathbf{1}\ox_k\|^2]
	\le \left(\frac{\Gamma+1}{\Gamma-1}\right)\frac{4\alpha^2 (1+\rho_w^2)\rho_w^2(2\alpha^2L^3\sigma^2+\mu M_{\sigma})}{\mu(1-\rho_w^2)^3},
	\end{equation}
	where 
	\begin{equation}
	\label{M_sigma}
	M_{\sigma}:=\left[3\alpha^2L^2+2(\alpha L+1)(n+1)\right]\sigma^2.
	\end{equation}
\end{theorem}
\begin{remark}
	The first term on the right-hand side of (\ref{error_bound_ultimate}) can be interpreted as the error caused by stochastic gradients only, since it does not depend on the network topology. 
	The second term 
	on the right hand side of~(\ref{error_bound_ultimate}) and the bound in (\ref{consensus_error_bound_ultimate}) are network dependent, and they increase with $\rho_w$ (larger $\rho_w$ indicates worse network connectivity).
	
	In light of (\ref{error_bound_ultimate}) and (\ref{consensus_error_bound_ultimate}), we have
	\begin{equation*}
	\limsup_{k\rightarrow\infty}\bE[\|\ox_k-x^*\|^2]=\alpha\mathcal{O}\left(\frac{\sigma^2}{\mu n}\right)+\frac{\alpha^2}{(1-\rho_w)^3}\mathcal{O}\left(\frac{ L^2\sigma^2}{\mu^2}\right),
	\end{equation*}
	and
	\begin{equation*}
	\limsup_{k\rightarrow\infty}\frac{1}{n}\bE[\|\mx_k-\mathbf{1}\ox_k\|^2]=\frac{\alpha^2}{(1-\rho_w)^3}\mathcal{O}\left(\sigma^2\right)+\frac{\alpha^4}{(1-\rho_w)^3}\mathcal{O}\left(\frac{ L^3\sigma^2}{\mu n}\right).
	\end{equation*}
	Let $\frac{1}{n}\bE[\|\mx_k-\mathbf{1}x^*\|^2]$ measure the average quality of solutions obtained by all the agents. We have
	\begin{multline}
	\label{limsup_x-x*}
	\limsup_{k\rightarrow\infty}\frac{1}{n}\bE[\|\mx_k-\mathbf{1}x^*\|^2]=\limsup_{k\rightarrow\infty}\bE[\|\ox_k-x^*\|^2]+\limsup_{k\rightarrow\infty}\frac{1}{n}\bE[\|\mx_k-\mathbf{1}\ox_k\|^2]\\
	=\alpha\mathcal{O}\left(\frac{\sigma^2}{\mu n}\right)+\frac{\alpha^2}{(1-\rho_w)^3}\mathcal{O}\left(\frac{ L^2\sigma^2}{\mu^2}\right),
	\end{multline}
	which is decreasing in the network size $n$ when $\alpha$ is sufficiently small, i.e, when
		\begin{equation*}
		\alpha\frac{\sigma^2}{\mu n}\sim \frac{\alpha^2}{(1-\rho_w)^3}\frac{ L^2\sigma^2}{\mu^2},\quad \text{or}\quad
		\alpha\sim \frac{\mu}{L^2}\frac{(1-\rho_w)^3}{n}.
		\end{equation*}
		
		The spectral gap $1-\rho_w$ depends on the graph topology as well as the choices of weights $[w_{ij}]$. For example, suppose the Lazy Metropolis rule is adopted (see \cite{nedic2018network}). In this case, if the communication graph is a 1) path or ring graph, then $1-\rho_w=\mathcal{O}(1/n^2)$; 2) lattice, then $1-\rho_w=\mathcal{O}(1/n)$; 3) geometric random graph, then $1-\rho_w=\mathcal{O}(1/n\log n)$; 4) Erd\H{o}s-R{\'e}nyi random graph, then $1-\rho_w=\mathcal{O}(1)$; 5) complete graph, then $1-\rho_w=1$; 6) any connected undirected graph, then $1-\rho_w=\mathcal{O}(1/n^2)$.
		
		From the above argument, the condition $\alpha\sim\frac{\mu}{L^2}\frac{(1-\rho_w)^3}{n}$ is in general more strict than (\ref{alpha_ultimate_bound}) in Theorem \ref{Theorem1}, which requires $\alpha\sim \frac{(1-\rho_w)^2}{L}$ (when $1-\rho_w\le(\frac{\mu}{L})^{2/3}$). Such a difference suggests that when implementing DSGT in practice, it can be advantageous to use a larger stepsize in the early stage to achieve a faster convergence speed (see Corollary \ref{cor: speed} below) and then switch to smaller stepsizes for more accuracy on the final solutions.
\end{remark}

\begin{remark}
	Under a centralized stochastic gradient (CSG) algorithm in the form of
	\begin{equation}
	\label{eq: centralized}
	x_{k+1}=x_k-\alpha \frac{1}{n}\sum_{i=1}^n g_i(x_k,\xi_{i,k}),\ \ k\in\mathbb{N},
	\end{equation}
	we would obtain
	\begin{equation*}
	\limsup_{k\rightarrow\infty}\bE[\|x_k-x^*\|^2]=\alpha\mathcal{O}\left(\frac{\sigma^2}{\mu n}\right).
	\end{equation*}
	It can be observed that DSGT is comparable with CSG in their ultimate error bounds (up to constant factors) with sufficiently small stepsizes.
\end{remark}

As shown in Theorem \ref{Theorem1}, the convergence rate of DSGT is determined by the spectral radius $\rho(\mA)<1$. In the corollary below we provide an upper bound of $\rho(\mA)$.
\begin{corollary}
	\label{cor: speed}
	Under the conditions in Theorem {\ref{Theorem1}}, assuming in addition that the stepsize $\alpha$ also satisfies\footnote{This condition is weaker than (\ref{alpha_ultimate_bound}) when $\rho_w^2\ge \frac{\Gamma}{\Gamma+1}\frac{2\mu}{3L}$.}
	\begin{equation}
	\label{alpha condition corollary}
	\alpha\le \frac{(\Gamma+1)}{\Gamma}\frac{(1-\rho_w^2)}{8\mu},
	\end{equation}
	we have
	\begin{equation*}
	\rho(\mA)\le 1-\left(\frac{\Gamma-1}{\Gamma+1}\right)\alpha\mu.
	\end{equation*}
\end{corollary}
Corollary \ref{cor: speed} implies that, for sufficiently small stepsizes, 
the distributed gradient tracking method has a comparable convergence speed to that of a centralized scheme (in which case the linear rate is $\mathcal{O}((1-2\alpha\mu)^k)$).

In the next theorem, we show DSGT achieves the $\mathcal{O}(\frac{1}{k})$ rate of convergence under a diminishing stepsize policy.
\begin{theorem}
	\label{theorem2}
	Let Assumptions \ref{asp: gradient samples}-\ref{asp: strconvexity} hold. 
	Consider the method in~\eqref{eq: x_i,k} where $\alpha$ is replaced with the time-varying stepsize
	$\alpha_k$  given by
	$\alpha_k:=\theta/(m+k)$, where $\theta>1/\mu$ and $m$ satisfies
	\begin{equation}
	\label{m: condition}
	\begin{cases}
	m>\max\left\{\frac{\theta}{2}(\mu+L), \frac{4\theta L\rho_w^2+2\theta L\rho_w\sqrt{1+3\rho_w^2}}{1-\rho_w^2}\right\},\\
	\frac{(1-\rho_w^2)^2}{\theta^2(1+\rho_w^2)\rho_w^2}\left[\frac{(1-\rho_w^2)}{2}-\frac{2m+1}{(m+1)^2}\right]> \frac{1}{(\theta\mu-1)}\left(\frac{1}{\mu}+\frac{\theta}{m}\right) \frac{4\theta^2 L^5}{m^3}+\frac{2C}{m^2},
	\end{cases}
	\end{equation}
	with $C=\left[\left(\frac{1-\rho_w^2}{2\rho_w^2}-\frac{4\theta L}{m}-\frac{2\theta^2 L^2}{m^2}\right)^{-1}+2\right]\|\mW-\mI\|^2 L^2+\frac{3\theta L^3}{m}.$
	Then for all $k\ge 0$, we have
	\begin{subequations}
		\label{SA_rate}
		\begin{align}
		& \bE[\|\ox_{k}-x^*\|^2]
		\le \frac{2\theta^2 \sigma^2}{n(\theta\mu-1)(m+k)}+\frac{\mathcal{O}_k(1)}{(m+k)^{\theta\mu}}+\frac{\mathcal{O}_k(1)}{(m+k)^2},\\
		& \bE[\|\mx_{k}-\mathbf{1}\ox_{k}\|^2]\le \frac{\mathcal{O}_k(1)}{(m+k)^2},
		\end{align}
	\end{subequations}
where $\mathcal{O}_k(1)$ denotes some constant that does not depend on $k$.
\end{theorem}
\begin{remark}
	From (\ref{SA_rate}), noting that $\theta\mu>1$, we have
	\begin{equation*}
	\frac{1}{n}\bE[\|\mx_{k}-\mathbf{1}x^*\|^2]=\bE[\|\ox_{k}-x^*\|^2]+ \frac{1}{n}\bE[\|\mx_{k}-\mathbf{1}\ox_{k}\|^2]=\mathcal{O}_k\left(\frac{1}{k}\right).
	\end{equation*}
	In particular, $\frac{1}{n}\bE[\|\mx_{k}-\mathbf{1}x^*\|^2]$ asymptotically converges to $0$ at the rate $\frac{2\theta^2 \sigma^2}{n(\theta\mu-1)k}$, which does not depend on the spectral norm $\rho_w$ and matches the convergence rate (up to constant factors) for centralized stochastic gradient methods (see \cite{nemirovski2009robust,rakhlin2012making})\footnote{It is worth noting that the transient time for $\frac{1}{n}\bE[\|\mx_{k}-\mathbf{1}x^*\|^2]$ to approach the asymptotic convergence rate depends on the network topology, and its dependence on $n$ can be significant.}.
	\end{remark}
\begin{remark}
	Under a well connected network ($\rho_w\simeq 0$), $m>\frac{\theta}{2}(\mu+L)$ suffices. When $\rho_w\simeq 1$, the lower bound of $m$ is in the order of $\mathcal{O}((1-\rho_w)^{-2})$. 
\end{remark}
Theorem \ref{theorem2} implies the following result if we choose $\alpha_k=\theta/(k+1)$.
\begin{corollary}
	\label{cor:SA}
	Under Assumptions \ref{asp: gradient samples}-\ref{asp: strconvexity} and stepsize policy $\alpha_k:=\theta/(k+1)$ for some $\theta>1/\mu$, we have for all $k\ge m$ where $m$ satisfies condition (\ref{m: condition}), 
	\begin{equation*}
	\bE[\|\ox_{k}-x^*\|^2]\le \frac{\tilde{U}}{k}, \ \ \bE[\|\mx_{k}-\mathbf{1}\ox_{k}\|^2]\le \frac{\tilde{X}}{k^2},\ \ \bE[\|\my_{k}-\mathbf{1}\oy_{k}\|^2]\le \tilde{Y},
	\end{equation*}
	where $\tilde{U}$, $\tilde{X}$, and $\tilde{Y}$ are some positive constants.
\end{corollary}

\section{Analysis}
\label{sec:cohesion}

In this section, we prove Theorem \ref{Theorem1} by studying the evolution of $\bE[\|\ox_k-x^*\|^2]$, $\bE[\|\mx_k-\mathbf{1}\ox_k\|^2]$ and $\bE[\|\my_k-\mathbf{1}\oy_k\|^2]$. Our strategy is to bound the three expressions in terms of linear combinations of their past values, in which way we establish a linear system of inequalities. This approach is different from those employed in \cite{qu2017harnessing,nedic2017achieving}, where the analyses pertain to the examination of $\|\ox_k-x^*\|$, $\|\mx_k-\mathbf{1}\ox_k\|$ and $\|\my_k-\mathbf{1}\oy_k\|$. Such distinction is due to the  stochastic gradients $g_i(x_{i,k},\xi_{i,k})$ whose variances play a crucial role in deriving the main inequalities.

We first introduce some lemmas that will be used later in the analysis. 
Denote by $\mathcal{F}_k$ the $\sigma$-algebra generated by 
$\{\boldsymbol{\xi}_0,\ldots,\boldsymbol{\xi}_{k-1}\}$, and define $\bE[\cdot \mid\mathcal{F}_k]$ 
as the conditional expectation given $\mathcal{F}_k$.
\begin{lemma}
	\label{lem: oy_k-h_k}
	Under Assumption \ref{asp: gradient samples}, recalling that 
	$h(\mx)= \frac{1}{n}\mathbf{1}^{\T}\nabla F(\mx)$,
	we have for all $k\ge0$,
	\begin{align}
	\bE\left[\|\oy_k-h(\mx_k)\|^2\mid\mathcal{F}_k\right] \le \frac{\sigma^2}{n}.
	\end{align}
\end{lemma}
\begin{proof}
	By the definitions of $\oy_k$ and $h(\mx_k)$,
	\begin{equation*}
	\bE\left[\|\oy_k-h(\mx_k)\|^2\mid\mathcal{F}_k\right]\\
	=	\frac{1}{n^2}\sum_{i=1}^n\bE\left[\|g_i(x_{i,k},\xi_{i,k})-\nabla f_i(x_{i,k})\|^2\vert\mathcal{F}_k\right]\le \frac{\sigma^2}{n}.
	\end{equation*}
\end{proof}
\begin{lemma}
	\label{lem: strong_convexity}
	Under Assumption \ref{asp: strconvexity}, we have for all $k\ge0$,
	\begin{align}
	\| \nabla f(\ox_k)-h(\mx_k)\| \le \frac{L}{\sqrt{n}}\|\mx_k-\mathbf{1}\ox_k\|.
	\end{align}
	If in addition $\alpha<2/(\mu+L)$, then
	\begin{equation*}
	\|x-\alpha\nabla f(x)-x^*\|\le (1-\alpha \mu)\|x-x^*\|,\, \forall x\in\mathbb{R}^p.
	\end{equation*}
\end{lemma}
\begin{proof}
	See  \cite{qu2017harnessing} Lemma 10 for reference.
\end{proof}

In the following lemma, we establish bounds on $\|\mx_{k+1}-\mathbf{1}\ox_{k+1}\|^2$ and on the conditional expectations of $\|\ox_{k+1}-x^*\|^2$ and $\|\my_{k+1}-\mathbf{1}\oy_{k+1}\|^2$, respectively.
\begin{lemma}
	\label{lem: Main_Inequalities}
	Suppose Assumptions \ref{asp: gradient samples}-\ref{asp: strconvexity} hold and $\alpha<2/(\mu+L)$. We have the following inequalities:
	\begin{equation}
	\label{First_Main_Inequality}
	\bE[\|\ox_{k+1}-x^*\|^2\mid \mathcal{F}_k]
	\le \left(1-\alpha\mu\right)\|\ox_k-x^*\|^2\\
	+\frac{\alpha L^2}{\mu n}\left(1+\alpha\mu\right)\|  \mx_k-\mathbf{1}\ox_k\|^2
	+\frac{\alpha^2\sigma^2}{n},
	\end{equation}
	\begin{equation}
	\label{Second_Main_Inequality}
	\|\mx_{k+1}-\mathbf{1}\ox_{k+1}\|^2
	\le \frac{(1+\rho_w^2)}{2}\|\mx_k-\mathbf{1}\ox_k\|^2\\+\alpha^2\frac{(1+\rho_w^2)\rho_w^2}{(1-\rho_w^2)}\|\my_{k}-\mathbf{1}\oy_k\|^2,
	\end{equation}
	and for any $\beta>0$,
	\begin{multline}
	\label{Third_Main_Inquality}
	\bE[\|\my_{k+1}-\mathbf{1}\oy_{k+1}\|^2\mid \mathcal{F}_k]
	\le \left(1+4\alpha L+2\alpha^2 L^2+\beta\right)\rho_w^2\bE[\|\my_{k}-\mathbf{1}\oy_{k}\|^2\mid \mathcal{F}_k]\\
	+\left(\frac{1}{\beta}\|\mW-\mI\|^2 L^2+2\|\mW-\mI\|^2 L^2+3\alpha L^3\right)\|\mx_k-\mathbf{1}\ox_k\|^2+2\alpha nL^3\|\ox_k-x^*\|^2+M_{\sigma}.
	\end{multline}
\end{lemma}
\begin{proof}
	See Appendix \ref{proof lem: Main_Inequalities}.
\end{proof}

\subsection{Proof of Theorem \ref{Theorem1}}
Taking full expectation on both sides of (\ref{First_Main_Inequality}), (\ref{Second_Main_Inequality}) and (\ref{Third_Main_Inquality}), we obtain the following linear system of inequalities
\begin{eqnarray}
\label{linear_system}
\begin{bmatrix}
\bE[\|\ox_{k+1}-x^*\|^2]\\
\bE[\|\mx_{k+1}-\mathbf{1}\ox_{k+1}\|^2]\\
\bE[\|\my_{k+1}-\mathbf{1}\oy_{k+1}\|^2]
\end{bmatrix}
\le 
\mA\begin{bmatrix}
\bE[\|\ox_{k}-x^*\|^2]\\
\bE[\|\mx_{k}-\mathbf{1}\ox_{k}\|^2]\\
\bE[\|\my_{k}-\mathbf{1}\oy_{k}\|^2]
\end{bmatrix}+\begin{bmatrix}
\frac{\alpha^2\sigma^2}{n}\\
0\\
M_{\sigma}
\end{bmatrix},
\end{eqnarray}
where the inequality is to be taken component-wise, and the entries of the matrix
$\mA=[a_{ij}]$ are given by
\begin{eqnarray*}
	& \begin{bmatrix}
		a_{11}\\
		a_{21}\\
		a_{31}
	\end{bmatrix}  =  
	\begin{bmatrix}
		1-\alpha\mu\\
		0\\
		2\alpha nL^3 
	\end{bmatrix},
	\begin{bmatrix}
		a_{12}\\
		a_{22}\\
		a_{32}
	\end{bmatrix} = 
	\begin{bmatrix}
		\frac{\alpha L^2}{\mu n}(1+\alpha\mu)\\
		\frac{1}{2}(1+\rho_w^2)\\
		\left(\frac{1}{\beta}+2\right)\|\mW-\mI\|^2 L^2+3\alpha L^3
	\end{bmatrix},\\
	& \begin{bmatrix}
		a_{13}\\
		a_{23}\\
		a_{33}
	\end{bmatrix} = 
	\begin{bmatrix}
		0 \\
		\alpha^2\frac{(1+\rho_w^2)\rho_w^2}{(1-\rho_w^2)}\\
		\left(1+4\alpha L+2\alpha^2 L^2+\beta\right)\rho_w^2
	\end{bmatrix},
\end{eqnarray*}
and $M_{\sigma}$ is given in (\ref{M_sigma}). Therefore, by induction we have
\begin{eqnarray}
\label{linear_system_bound}
\begin{bmatrix}
\bE[\|\ox_{k}-x^*\|^2]\\
\bE[\|\mx_{k}-\mathbf{1}\ox_{k}\|^2]\\
\bE[\|\my_{k}-\mathbf{1}\oy_{k}\|^2]
\end{bmatrix}
\le 
\mA^k\begin{bmatrix}
\bE[\|\ox_{0}-x^*\|^2]\\
\bE[\|\mx_{0}-\mathbf{1}\ox_{0}\|^2]\\
\bE[\|\my_{0}-\mathbf{1}\oy_{0}\|^2]
\end{bmatrix}+\sum_{l=0}^{k-1}\mA^l\begin{bmatrix}
\frac{\alpha^2\sigma^2}{n}\\
0\\
M_{\sigma}
\end{bmatrix}.
\end{eqnarray}
If the spectral radius of $\mA$ satisfies $\rho(\mA)<1$, then $\mA^k$ converges to $\mathbf{0}$ at the linear rate $\mathcal{O}(\rho(\mA)^k)$  (see \cite{horn1990matrix}), in which case $\sup_{l\ge k}\bE[\|\ox_l-x^*\|^2]$, $\sup_{l\ge k}\bE[\|\mx_l-\mathbf{1}\ox_l\|^2]$ and $\sup_{l\ge k}\bE[\|\my_l-\mathbf{1}\oy_l\|^2]$ all converge to a neighborhood of $0$ at the linear rate $\mathcal{O}(\rho(\mA)^k)$. 
The next lemma provides conditions that ensure $\rho(\mA)<1$.
\begin{lemma}
	\label{lem: rho_M}
	Let $\mS=[s_{ij}]\in\mathbb{R}^{3\times 3}$ be a nonnegative, irreducible matrix with 
	$s_{ii}<\lambda^*$ for some~{$\lambda^*>0$} for all $i=1,2,3$.
	Then $\rho(\mS)<\lambda^*$ iff $\mathrm{det}(\lambda^* \mI-\mS)>0$.
\end{lemma}
\begin{proof}
	See Appendix \ref{subsec: proof lemma rho_M}.
\end{proof}
We now derive the conditions such that $\rho(\mA)<1$.
Suppose $\alpha$ and $\beta$ meet the following relations\footnote{Matrix $\mA$ in Theorem~\ref{Theorem1} 
	corresponds to such a choice of $\alpha$ and $\beta$.}:
\begin{equation}
\label{beta}
a_{33}=\left(1+4\alpha L+2\alpha^2 L^2+\beta\right)\rho_w^2=\frac{1+\rho_w^2}{2}<1,
\end{equation}
\begin{equation}
\label{alpha,beta}
a_{23}a_{32}=\alpha^2\frac{(1+\rho_w^2)\rho_w^2}{(1-\rho_w^2)}\left[\left(\frac{1}{\beta}+2\right)\|\mW-\mI\|^2 L^2+3\alpha L^3\right]
\le\frac{1}{\Gamma}(1-a_{22})(1-a_{33}),
\end{equation}
 and
\begin{equation}
\label{alpha last condition}
a_{12}a_{23}a_{31}=\frac{2\alpha^4 L^5(1+\alpha\mu)}{\mu}\frac{(1+\rho_w^2)}{(1-\rho_w^2)}\rho_w^2
\le \frac{1}{\Gamma+1}(1-a_{11})[(1-a_{22})(1-a_{33})-a_{23}a_{32}].
\end{equation}
Then,
\begin{multline*}
\text{det}(\mI-\mA)=(1-a_{11})(1-a_{22})(1-a_{33})-(1-a_{11})a_{23}a_{32}-a_{12}a_{23}a_{31}\\
\ge \frac{\Gamma}{(\Gamma+1)}(1-a_{11})[(1-a_{22})(1-a_{33})-a_{23}a_{32}]
\ge \left(\frac{\Gamma-1}{\Gamma+1}\right)(1-a_{11})(1-a_{22})(1-a_{33})>0.
\end{multline*}
Given that $a_{11},a_{22},a_{33}<1$, in light of Lemma \ref{lem: rho_M}, we have $\rho(\mA)<1$.
In addition, by denoting $B:=[\frac{\alpha^2\sigma^2}{n}, 0, M_{\sigma}]^{\T}$ and $[(\mI-\mA)^{-1}B]_j$ the $j$-th element of the vector $[(\mI-\mA)^{-1}B]$, we obtain from (\ref{linear_system_bound}) that
\begin{align}
\label{error_bound_preliminary}
\lim\sup_{k\rightarrow\infty}\bE[\|\ox_k-x^*\|^2]\le & [(\mI-\mA)^{-1}B]_1 \notag\\
= & \frac{1}{\text{det}(\mI-\mA)}\left\{\left[(1-a_{22})(1-a_{33})-a_{23}a_{32}\right]\frac{\alpha^2\sigma^2}{n}+a_{12}a_{23}M_{\sigma}\right\} \notag\\
\le & \frac{(\Gamma+1)}{\Gamma}\frac{\alpha\sigma^2}{\mu n}+\left(\frac{\Gamma+1}{\Gamma-1}\right)\frac{ a_{12}a_{23}M_{\sigma}}{(1-a_{11})(1-a_{22})(1-a_{33})}\notag\\
= & \frac{(\Gamma+1)}{\Gamma}\frac{\alpha\sigma^2}{\mu n}+ \left(\frac{\Gamma+1}{\Gamma-1}\right)\frac{\alpha^3 L^2(1+\alpha\mu)(1+\rho_w^2)\rho_w^2M_{\sigma}}{\mu n(1-\rho_w^2)(1-a_{11})(1-a_{22})(1-a_{33})}\notag\\
= & \frac{(\Gamma+1)}{\Gamma}\frac{\alpha\sigma^2}{\mu n}+\left(\frac{\Gamma+1}{\Gamma-1}\right)\frac{4\alpha^2 L^2(1+\alpha\mu)(1+\rho_w^2)\rho_w^2}{\mu^2 n(1-\rho_w^2)^3}M_{\sigma},
\end{align}
and
\begin{multline*}
\lim\sup_{k\rightarrow\infty}\bE[\|\mx_k-\mathbf{1}\ox_k\|^2] \le [(\mI-\mA)^{-1}B]_2
=\frac{1}{\text{det}(\mI-\mA)}\left[a_{23}a_{31}\frac{\alpha^2\sigma^2}{n}+a_{23}(1-a_{11})M_{\sigma}\right]\\
\le \left(\frac{\Gamma+1}{\Gamma-1}\right)\frac{a_{23}}{(1-a_{11})(1-a_{22})(1-a_{33})}\left(2\alpha nL^3\frac{\alpha^2 \sigma^2}{n}+\alpha\mu M_{\sigma}\right)\\
= \frac{4(\Gamma+1)\alpha^2 (1+\rho_w^2)\rho_w^2(2\alpha^2L^3\sigma^2+\mu M_{\sigma})}{(\Gamma-1)\mu(1-\rho_w^2)^3}.
\end{multline*}
It remains to show that (\ref{beta}), (\ref{alpha,beta}) and (\ref{alpha last condition}) are satisfied under condition (\ref{alpha_ultimate_bound}).  By (\ref{beta}), we need
\begin{equation*}
\beta=\frac{1-\rho_w^2}{2\rho_w^2}-4\alpha L-2\alpha^2 L^2.
\end{equation*}
Since $\alpha\le \frac{1-\rho_w^2}{12\rho_w L}$ by (\ref{alpha_ultimate_bound}), we know that
\begin{equation}
\label{beta_bound}
\beta\ge \frac{1-\rho_w^2}{2\rho_w^2}-\frac{1-\rho_w^2}{3\rho_w}-\frac{(1-\rho_w^2)^2}{72\rho_w^2}\ge \frac{11(1-\rho_w^2)}{72\rho_w^2}\ge \frac{1-\rho_w^2}{8\rho_w^2}>0.
\end{equation}
Condition (\ref{alpha,beta}) leads to the inequality below:
\begin{equation*}
\alpha^2\frac{(1+\rho_w^2)\rho_w^2}{(1-\rho_w^2)}\left[\left(\frac{1}{\beta}+2\right)\|\mW-\mI\|^2 L^2+3\alpha L^3\right]\le\frac{(1-\rho_w^2)^2}{4\Gamma}.
\end{equation*}
By (\ref{beta_bound}), we only need
\begin{equation*}
\alpha^2\left[\frac{(2+6\rho_w^2)}{(1-\rho_w^2)}\|\mW-\mI\|^2 L^2+\frac{(1-\rho_w^2)}{4\rho_w^2}L^2\right]\le\frac{(1-\rho_w^2)^3}{4\Gamma(1+\rho_w^2)\rho_w^2}.
\end{equation*}
The preceding inequality is equivalent to
\begin{equation*}
\alpha\le \frac{(1-\rho_w^2)^2}{L\sqrt{\Gamma(1+\rho_w^2)}\sqrt{4\rho_w^2(2+6\rho_w^2)\|\mW-\mI\|^2+(1-\rho_w^2)^2}},
\end{equation*}
implying that it is sufficient to have
\begin{equation*}
\alpha\le\frac{(1-\rho_w^2)^2}{2\sqrt{\Gamma}L\max(6\rho_w\|\mW-\mI\|,1-\rho_w^2)}.
\end{equation*}
To see that relation (\ref{alpha last condition}) holds, consider a stronger condition
\begin{equation*}
\frac{2\alpha^4 L^5(1+\alpha\mu)}{\mu}\frac{(1+\rho_w^2)}{(1-\rho_w^2)}\rho_w^2
\le \frac{(\Gamma-1)}{\Gamma(\Gamma+1)}(1-a_{11})(1-a_{22})(1-a_{33}),
\end{equation*}
or equivalently,
\begin{equation*}
\frac{2\alpha^3 L^5(1+\alpha\mu)}{\mu^2}\frac{(1+\rho_w^2)}{(1-\rho_w^2)}\rho_w^2 \le \frac{(\Gamma-1)}{4\Gamma(\Gamma+1)}(1-\rho_w^2)^2.
\end{equation*}
It suffices that
\begin{equation}
\alpha\le \frac{(1-\rho_w^2)}{3\rho_w^{2/3}L}\left[\frac{\mu^2}{L^2}\frac{(\Gamma-1)}{\Gamma(\Gamma+1)}\right]^{1/3}.
\end{equation}

\subsection{Proof of Corollary \ref{cor: speed}}
We derive an upper bound of $\rho(\mA)$ under conditions (\ref{alpha_ultimate_bound}) and (\ref{alpha condition corollary}). Note that the characteristic function of $\mA$ is given by
\begin{equation*}
\text{det}(\lambda \mI-\mA)=(\lambda-a_{11})(\lambda-a_{22})(\lambda-a_{33})
-(\lambda-a_{11})a_{23}a_{32}-a_{12}a_{23}a_{31}.
\end{equation*}
Since $\text{det}(\mI-\mA)> 0$ and $\text{det}(\max\{a_{11},a_{22},a_{33}\} \mI-\mA)=\text{det}(a_{11}\mI-\mA)<0$, we have $\rho(\mA)\in(a_{11},1)$. By (\ref{alpha,beta}) and (\ref{alpha last condition}),
\begin{multline*}
\text{det}(\lambda \mI-\mA)
\ge (\lambda-a_{11})(\lambda-a_{22})(\lambda-a_{33})-(\lambda-a_{11})a_{23}a_{32}-\frac{1}{\Gamma+1}(1-a_{11})[(1-a_{22})(1-a_{33})-a_{23}a_{32}]\\
\ge (\lambda-a_{11})(\lambda-a_{22})(\lambda-a_{33})-\frac{1}{\Gamma}(\lambda-a_{11})(1-a_{22})(1-a_{33})
-\frac{(\Gamma-1)}{\Gamma(\Gamma+1)}(1-a_{11})(1-a_{22})(1-a_{33}).
\end{multline*}
Suppose $\lambda=1-\epsilon$ for some $\epsilon\in(0,\alpha\mu)$, satisfying
\begin{equation*}
\text{det}(\lambda \mI-\mA)
\ge \frac{1}{4}(\alpha\mu-\epsilon)\left(1-\rho_w^2-2\epsilon\right)^2
-\frac{1}{4\Gamma}(\alpha\mu-\epsilon)(1-\rho_w^2)^2-\frac{(\Gamma-1)\alpha\mu}{4\Gamma(\Gamma+1)}(1-\rho_w^2)^2\ge 0,
\end{equation*}
or equivalently,
\begin{equation*}
\frac{(\alpha\mu-\epsilon)}{\alpha\mu}\left[\frac{\left(1-\rho_w^2-2\epsilon\right)^2}{(1-\rho_w^2)^2}
-\frac{1}{\Gamma}\right]\ge \frac{(\Gamma-1)}{\Gamma(\Gamma+1)}.
\end{equation*}
It suffices that
\begin{equation*}
\epsilon\le \left(\frac{\Gamma-1}{\Gamma+1}\right)\alpha\mu.
\end{equation*}
To see why this is the case, note that $\frac{(\alpha\mu-\epsilon)}{\alpha\mu}\ge\frac{2}{\Gamma+1}$, and under (\ref{alpha condition corollary}),
\begin{equation*}
\epsilon\le \left(\frac{\Gamma-1}{\Gamma+1}\right)\frac{(\Gamma+1)}{\Gamma}\frac{(1-\rho_w^2)}{8\mu}\mu=\frac{(\Gamma-1)(1-\rho_w^2)}{8\Gamma}.
\end{equation*}
As a result,
\begin{equation*}
1-\rho_w^2-2\epsilon\ge 1-\rho_w^2-\frac{(\Gamma-1)(1-\rho_w^2)}{4\Gamma}=\frac{(3\Gamma+1)(1-\rho_w^2)}{4\Gamma}.
\end{equation*}
We then have
\begin{multline*}
\frac{(\alpha\mu-\epsilon)}{\alpha\mu}\left[\frac{\left(1-\rho_w^2-2\epsilon\right)^2}{(1-\rho_w^2)^2}
-\frac{1}{\Gamma}\right]\ge \frac{2}{(\Gamma+1)} \left[\frac{(3\Gamma+1)^2}{16\Gamma^2}-\frac{1}{\Gamma}\right]=\frac{1}{\Gamma(\Gamma+1)} \left[\frac{(3\Gamma+1)^2}{8\Gamma}-2\right]\\
=\frac{1}{\Gamma(\Gamma+1)} \left(\Gamma-1+\frac{\Gamma}{8}+\frac{1}{8\Gamma}-\frac{1}{4}\right)\ge \frac{(\Gamma-1)}{\Gamma(\Gamma+1)}.
\end{multline*}
Denote
\begin{equation*}
\tilde{\lambda}=1-\left(\frac{\Gamma-1}{\Gamma+1}\right)\alpha\mu.
\end{equation*}
Then $\text{det}(\tilde{\lambda} \mI-\mA)\ge 0$ so that $\rho(\mA)\le \tilde{\lambda}$.

\subsection{Proof of Theorem \ref{theorem2}}

Similar to (\ref{linear_system}), under stepsize policy $\alpha_k=\theta/(m+k)$ where $m>\frac{\theta}{2}(\mu+L)$, we have
\begin{eqnarray}
\label{linear_system_SA}
\begin{bmatrix}
\bE[\|\ox_{k+1}-x^*\|^2]\\
\bE[\|\mx_{k+1}-\mathbf{1}\ox_{k+1}\|^2]\\
\bE[\|\my_{k+1}-\mathbf{1}\oy_{k+1}\|^2]
\end{bmatrix}
\le 
\mA_k\begin{bmatrix}
\bE[\|\ox_{k}-x^*\|^2]\\
\bE[\|\mx_{k}-\mathbf{1}\ox_{k}\|^2]\\
\bE[\|\my_{k}-\mathbf{1}\oy_{k}\|^2]
\end{bmatrix}+\begin{bmatrix}
\frac{\alpha_k^2\sigma^2}{n}\\
0\\
M_k
\end{bmatrix},
\end{eqnarray}
where
\begin{eqnarray*}
	\mA_k=\begin{bmatrix}
		1-\alpha_k\mu & \frac{\alpha_k L^2}{\mu n}(1+\alpha_k\mu) & 0\\
		0 & \frac{1}{2}(1+\rho_w^2) & \alpha_k^2\frac{(1+\rho_w^2)\rho_w^2}{(1-\rho_w^2)}\\
		2\alpha_k nL^3 & \left(\frac{1}{\beta_k}+2\right)\|\mW-\mI\|^2 L^2+3\alpha_k L^3 & \frac{1}{2}(1+\rho_w^2)
	\end{bmatrix},
\end{eqnarray*}
$\beta_k=\frac{1-\rho_w^2}{2\rho_w^2}-4\alpha_k L-2\alpha_k^2 L^2>0$, and
\begin{equation}
\label{M_k}
M_k:=\left[3\alpha_k^2L^2+2(\alpha_k L+1)(n+1)\right]\sigma^2.
\end{equation}
The condition $\beta_k>0$ is satisfied when
\begin{equation*}
\beta_0=\frac{1-\rho_w^2}{2\rho_w^2}-\frac{4\theta L}{m}-\frac{2\theta^2 L^2}{m^2}>0,
\end{equation*}
or equivalently,
\begin{equation}
\label{m condition1}
m>\frac{4\theta L\rho_w^2+2\theta L\rho_w\sqrt{1+3\rho_w^2}}{1-\rho_w^2}.
\end{equation}

We first show that $\bE[\|\ox_{k}-x^*\|^2]\le \hat{U}/(m+k)$ and $\bE[\|\mx_{k}-\mathbf{1}\ox_{k}\|^2\le \hat{X}/(m+k)^2$ for some $\hat{U}$, $\hat{X}>0$ by induction.
Denote $U_k:=\bE[\|\ox_{k}-x^*\|^2]$, $X_k:=\bE[\|\mx_{k}-\mathbf{1}\ox_{k}\|^2]$, and $Y_k:=\bE[\|\my_{k}-\mathbf{1}\oy_{k}\|^2]$. 
Suppose for some $k\ge 0$,
\begin{equation}
\label{Induction_assumption}
U_k\le \hat{U}/(m+k), \ \ X_k\le \hat{X}/(m+k)^2,\ \ Y_k\le \hat{Y}.
\end{equation}
We want to prove that
\begin{subequations}
	\label{Condition: UXY}
	\begin{align}
	& U_{k+1}\le \frac{(1-\alpha_k\mu)\hat{U}}{(m+k)}+\frac{\alpha_k L^2}{\mu n}\frac{(1+\alpha_k\mu)\hat{X}}{(m+k)^2}+\frac{\alpha_k^2\sigma^2}{n}\le \frac{\hat{U}}{(m+k+1)},\\
	& X_{k+1}\le \frac{(1+\rho_w^2)\hat{X}}{2(m+k)^2}+\alpha_k^2 \frac{(1+\rho_w^2)\rho_w^2 \hat{Y}}{(1-\rho_w^2)}\le \frac{\hat{X}}{(m+k+1)^2},\\
	& Y_{k+1}\le \frac{2\alpha_k nL^3\hat{U}}{(m+k)}+\frac{C\hat{X}}{(m+k)^2}+\frac{(1+\rho_w^2)\hat{Y}}{2}+M_0\le \hat{Y},
	\end{align}
\end{subequations}
where $C=(\frac{1}{\beta_0}+2)\|\mW-\mI\|^2 L^2+3\alpha_0 L^3$. Plugging in $\alpha_k=\theta/(m+k)$, it suffices to show
\begin{subequations}
	\label{Ultimate_Condtion: UXV}
	\begin{align}
	& \hat{U}\ge \frac{1}{n(\theta\mu-1)}\left[\left(\frac{\theta L^2}{\mu m}+\frac{\theta^2 L^2}{m^2}\right)\hat{X}+\theta^2\sigma^2\right], \label{Ultimate_Condtion: UXV_line1}\\
	& \hat{Y}\le \frac{(1-\rho_w^2)}{\theta^2(1+\rho_w^2)\rho_w^2}\left[\frac{(1-\rho_w^2)}{2}-\frac{2m+1}{(m+1)^2}\right]\hat{X}, \label{Ultimate_Condtion: UXV_line2}\\
	& \hat{Y}\ge \frac{2}{(1-\rho_w^2)}\left(\frac{2\theta nL^3\hat{U}}{m^2}+\frac{C\hat{X}}{m^2}+\left[\frac{3\theta^2L^2}{m^2}+2\left(\frac{\theta L}{m}+1\right)(n+1)\right]\sigma^2\right).\label{Ultimate_Condtion: UXV_line3}
	\end{align}
\end{subequations}
Let
\begin{equation*}
\hat{U}:= \max\left\{\frac{1}{n(\theta\mu-1)}\left[\left(\frac{\theta L^2}{\mu m}+\frac{\theta^2 L^2}{m^2}\right)\hat{X}+\theta^2\sigma^2\right],m\|\ox_0-x^*\|^2\right\},
\end{equation*}
condition (\ref{Ultimate_Condtion: UXV}) admits a solution iff
\begin{equation}
\label{m condtion2}
\frac{(1-\rho_w^2)^2}{\theta^2(1+\rho_w^2)\rho_w^2}\left[\frac{(1-\rho_w^2)}{2}-\frac{2m+1}{(m+1)^2}\right]> \frac{1}{(\theta\mu-1)}\left(\frac{1}{\mu}+\frac{\theta}{m}\right) \frac{4\theta^2 L^5}{m^3}+\frac{2C}{m^2},
\end{equation}
in which case $\hat{X}$ and $\hat{Y}$ are lower bounded under constraints \{(\ref{Ultimate_Condtion: UXV_line2}), (\ref{Ultimate_Condtion: UXV_line3}), $\hat{X}\ge m^2 X_0$, $\hat{Y}\ge Y_0$\}. Specifically, $\hat{X}$ can be chosen as follows:
\begin{equation*}
\hat{X}:=\max\left\{\frac{C_3}{C_1-C_2},\frac{C_5}{C_1-C_4},m^2 X_0,\frac{Y_0}{C_1}\right\},
\end{equation*}
where
\begin{subequations}
	\label{Constants C's}
	\begin{align}
	C_1 & =\frac{(1-\rho_w^2)}{\theta^2(1+\rho_w^2)\rho_w^2}\left[\frac{(1-\rho_w^2)}{2}-\frac{2m+1}{(m+1)^2}\right],\\
	C_2 & =\frac{2C}{(1-\rho_w^2)m^2},\\
	C_3 & = \frac{2}{(1-\rho_w^2)}\left\{\frac{2\theta nL^3\|\ox_{0}-x^*\|^2}{m}+\left[\frac{3\theta^2L^2}{m^2}+2\left(\frac{\theta L}{m}+1\right)(n+1)\right]\sigma^2\right\},\\
	C_4 & = \frac{2}{(1-\rho_w^2)}\left[\frac{2\theta^2 L^5}{m^3(\theta\mu-1)}\left(\frac{1}{\mu}+\frac{\theta}{m}\right)+\frac{C}{m^2}\right],\\
	C_5& = \frac{2}{(1-\rho_w^2)}\left[\frac{2\theta^3 L^3}{m^2(\theta\mu -1)}+\frac{3\theta^2L^2}{m^2}+2\left(\frac{\theta L}{m}+1\right)(n+1)\right]\sigma^2.
	\end{align}
\end{subequations}
Noticing that relation (\ref{Induction_assumption}) holds trivially when $k=0$, the induction is complete.

We further improve the bound on $U_k$ (inspired by \cite{olshevsky2018robust,pu2019sharp}). From (\ref{linear_system_SA}),
\begin{equation*}
U(k+1)\le (1-\alpha_k\mu)U_k+\frac{2\alpha_k L^2}{\mu n}X_k+\frac{\alpha_k^2 \sigma^2}{n}.
\end{equation*}
Since $X_k\le \hat{X}/(m+k)^2$,
\begin{equation*}
U(k+1)\le \left(1-\frac{\theta\mu}{m+k}\right)U_k+\frac{2\theta L^2\hat{X}}{\mu n(m+k)^3}+\frac{\theta^2 \sigma^2}{n(m+k)^2}.
\end{equation*}
Hence
\begin{equation*}
U_k\le \prod_{t=0}^{k-1}\left(1-\frac{\theta\mu}{m+t}\right)U_0+\sum_{t=0}^{k-1}\left(\prod_{j=t+1}^{k-1}\left(1-\frac{\theta\mu}{m+j}\right)\right)\left(\frac{2\theta L^2\hat{X}}{\mu n(m+t)^3}+\frac{\theta^2 \sigma^2}{n(m+t)^2}\right).
\end{equation*}
In light of Lemma 4.1 in \cite{olshevsky2019non},
\begin{equation*}
\prod_{t=0}^{k-1}\left(1-\frac{\theta\mu}{m+t}\right)\le \frac{m^{\theta\mu}}{(m+k)^{\theta\mu}},\quad \prod_{j=t+1}^{k-1}\left(1-\frac{\theta\mu}{m+j}\right)\le \frac{(m+t+1)^{\theta\mu}}{(m+k)^{\theta\mu}}.
\end{equation*}
Then,
\begin{align*}
U_k\le &  \frac{m^{\theta\mu}}{(m+k)^{\theta\mu}}U_0+\sum_{t=0}^{k-1}\frac{(m+t+1)^{\theta\mu}}{(m+k)^{\theta\mu}}\left(\frac{2\theta L^2\hat{X}}{\mu n(m+t)^3}+\frac{\theta^2 \sigma^2}{n(m+t)^2}\right)\\
= & \frac{m^{\theta\mu}}{(m+k)^{\theta\mu}}U_0+\frac{1}{(m+k)^{\theta\mu}}\left(\frac{2\theta L^2\hat{X}}{\mu n}\sum_{t=0}^{k-1}\frac{(m+t+1)^{\theta\mu}}{(m+t)^3}+\frac{\theta^2 \sigma^2}{n}\sum_{t=0}^{k-1}\frac{(m+t+1)^{\theta\mu}}{(m+t)^2}\right)\\
\le & \frac{m^{\theta\mu}}{(m+k)^{\theta\mu}}U_0+\frac{1}{(m+k)^{\theta\mu}}\left(\frac{4\theta L^2\hat{X}}{\mu n}\sum_{t=0}^{k-1}(m+t)^{\theta\mu-3}+\frac{2\theta^2 \sigma^2}{n}\sum_{t=0}^{k-1}(m+t)^{\theta\mu-2}\right).
\end{align*}
Note that
\begin{align*}
& \sum_{t=0}^{k-1}(m+t)^{\theta\mu-3}\le \int_{t=-1}^{k}(m+t)^{\theta\mu-3}dt\le \max\left\{\frac{1}{\theta\mu-2}(m+k)^{\theta\mu-2},\frac{1}{2-\theta\mu}(m-1)^{\theta\mu-2}\right\},\\
& \sum_{t=0}^{k-1}(m+t)^{\theta\mu-2}\le \int_{t=-1}^{k}(m+t)^{\theta\mu-2}dt\le \frac{1}{\theta\mu-1}(m+k)^{\theta\mu-1}.
\end{align*}
We conclude that
\begin{align*}
U_k
\le & \frac{2\theta^2 \sigma^2}{n(\theta\mu-1)(m+k)}+\frac{m^{\theta\mu}}{(m+k)^{\theta\mu}}U_0+\frac{4\theta L^2\hat{X}}{\mu n}\max\left\{\frac{1}{(\theta\mu-2)(m+k)^2},\frac{(m-1)^{\theta\mu-2}}{(2-\theta\mu)(m+k)^{\theta\mu}}\right\}\\
\le & \frac{2\theta^2 \sigma^2}{n(\theta\mu-1)(m+k)}+\frac{\mathcal{O}_k(1)}{(m+k)^{\theta\mu}}+\frac{\mathcal{O}_k(1)}{(m+k)^2}.
\end{align*}

\section{A Gossip-Like Stochastic Gradient Tracking Method (GSGT)}
\label{sec: gossip}
In this section, we consider a gossip-like stochastic gradient tracking method (GSGT): Initialize with an arbitrary $x_{i,0}$ and $y_{i,0}=g_i(x_{i,0},\xi_{i,0})$ for all $i\in\mathcal{N}$. At each round $k\in\mathbb{N}$, agent $i_k\in\mathcal{N}$ wakes up with probability $1/n$. Then $i_k$ either communicates with one of its neighbors $j_k$ (with probability $\pi_{i_k j_k}$) or not (with probability $\pi_{i_k i_k}=1-\sum_{j\in\mathcal{N}_{i_k}}\pi_{i_k j}$). In the former situation, the update rule for $i\in\{i_k,j_k\}$ is as follows,
\begin{subequations}\label{gossip algorithm}
	\begin{align}
	x_{i,k+1}
	= & \frac{1}{2}(x_{i_k,k}+x_{j_k,k})-\alpha y_{i,k}, \label{eq:x-update gossip}\\
	y_{i,k+1}= & \frac{1}{2}(y_{i_k,k}+y_{j_k,k})+g_i(x_{i,k+1},\xi_{i,k+1})-g_i(x_{i,k},\xi_{i,k}), \label{eq:y-update gossip}
	\end{align}
\end{subequations}
and for $i\notin \{i_k,j_k\}$, $x_{i,k+1}=x_{i,k}$, $y_{i,k+1}=y_{i,k}$, and $\xi_{i,k+1}=\xi_{i,k}$\footnote{In practice, this means agent $i$ holds vectors $x_i$, $y_i$ and $g_i(x_i,\xi_i)$ if it does not wake up.}. In the latter situation, agent $i_k$ performs update based on its own information:
\begin{subequations}\label{gossip algorithm2}
	\begin{align}
	x_{i_k,k+1}
	= & x_{i_k,k}-2\alpha y_{i_k,k}, \label{eq:x-update gossip2}\\
	y_{i_k,k+1}= & y_{i_k,k}+g_{i_k}(x_{i_k,k+1},\xi_{i_k,k+1})-g_{i_k}(x_{i_k,k},\xi_{i_k,k}), \label{eq:y-update gossip2}
	\end{align}
\end{subequations}
while no action is taken by agent $i\neq i_k$. For ease of analysis, we denote $j_k=i_k$ in this case, and let $\mathds{1}_k$ be the indicator function for the event $\{j_k\neq i_k\}$, i.e., $\mathds{1}_k=1$ iff $j_k\neq i_k$.

The use of stepsize $2\alpha$ instead of $\alpha$ in (\ref{eq:x-update gossip2}) can be understood as follows. At each iteration, GSGT performs two gradient updates within the network. This can be achieved either by two different agents respectively updating their solutions, or by one agent using a doubled stepsize. The method is different from a standard gossip algorithm where exactly two agents update at each round. This difference allows us to design the probabilities $\pi_{ij}$ with more flexibility. In particular, it is possible to construct a doubly stochastic 
probability matrix $\mPi=[\pi_{ij}]$ for any graph $\mathcal{G}$ under Assumption \ref{asp: network}.

We can present GSGT in the following compact matrix form, in which we adopt the notation previously used.
\begin{subequations}\label{gossip algorithm compact}
	\begin{align}
	\mx_{k+1}
	= & \mathbf{W}_k\mx_k-\alpha \mathbf{D}_k\my_k, \label{eq:x-update gossip compact}\\
	\my_{k+1}= & \mathbf{W}_k\my_k+\tilde{\mathbf{D}}_k (G(\mx_{k+1},\xi_{k+1})-G(\mx_k,\xi_k)), \label{eq:y-update gossip compact}
	\end{align}
\end{subequations}
where the random coupling matrix $\mathbf{W}_k$ is defined as
\begin{equation*}
\mathbf{W}_k:=\mathbf{I}-\frac{(e_{i_k}-e_{j_k})(e_{i_k}-e_{j_k})^{\T}}{2},
\end{equation*}
in which $e_i=[0\, \cdots\, 0\, 1\, 0\,\cdots]^{\T}\in\mathbb{R}^{n\times 1}$ is a unit vector with the $i$th component equal to $1$. By definition, each $\mW_k$ is symmetric and doubly stochastic. The matrices $\mathbf{D}_k$ and $\tilde{\mathbf{D}}_k$ are diagonal with their $i_k$th and $j_k$th  diagonal entries equal to $1$ and all other entries equal to $0$ if $j_k\neq i_k$, otherwise the $i_k$th entry of $\mathbf{D}_k$ (respectively, $\tilde{\mathbf{D}}_k$) equals $2$ (respectively, $1$) while all other entries equal $0$.

We assume the following condition on the probability matrix $\mPi$:
\begin{assumption}
	\label{asp: Pi}
	Nonnegative matrix $\mPi$ is doubly stochastic.
\end{assumption}
Let
\begin{equation}
\bar{\mW}:=\bE[\mathbf{W}_k^{\T} \mathbf{W}_k].
\end{equation}
It can be shown that (see \cite{boyd2006randomized})
\begin{equation}
\label{bar W}
\bar{\mW}=\left(1-\frac{1}{n}\right)\mathbf{I}+\frac{\mPi+\mPi^{\T}}{2n},
\end{equation}
which is doubly stochastic.
\begin{lemma}
	\label{lem: spectral norm_gossip}
	Let Assumption \ref{asp: network} and Assumption \ref{asp: Pi} hold, and let $\rho_{\bar{w}}$ denote the spectral norm of 
	the matrix $\bar{\mW}-\frac{1}{n}\mathbf{1}\mathbf{1}^{\T}$. Then, $\rho_{\bar{w}}\in[1-2/n,1)$.
\end{lemma}
\begin{proof}
	Since $\bar{W}$ is doubly stochastic, $\rho_{\bar{w}}<1$ follows from Lemma \ref{lem: spectral norm}.
	To see $\rho_{\bar{w}}\ge 1-2/n$, note that $\rho(\frac{\mPi+\mPi^{\T}}{2n})=\frac{1}{n}$, and $1$ is an algebraically simple eigenvalue of $\frac{\mPi+\mPi^{\T}}{2}$ and $\frac{1}{n}\mathbf{1}\mathbf{1}^{\T}$. We have
	\begin{equation}
	\label{equation:rho_bar w}
	\rho_{\bar{w}}=1-\frac{1}{n}+\frac{1}{n}\lambda_2\left(\frac{\mPi+\mPi^{\T}}{2}\right),
	\end{equation}
	where $\lambda_2(\frac{\mPi+\mPi^{\T}}{2})$ is the second largest eigenvalue of $\frac{\mPi+\mPi^{\T}}{2}$.
	Since $\lambda_2(\frac{\mPi+\mPi^{\T}}{2})\in[-1,1]$, we conclude that $\rho_{\bar{w}}\ge 1-2/n$.
\end{proof}
Before proceeding, it is worth noting that for GSGT, we still have the following relation:
\begin{equation}
\oy_k=\frac{1}{n}\mathbf{1}^{\T}G(\mx_k,\boldsymbol{\xi}_k),\forall k.
\end{equation}

\subsection{Main Results}

We present the main convergence results of GSGT in the following theorem.
\begin{theorem}
	\label{thm_gossip}
	Let $\Gamma>1$ be arbitrarily chosen. Suppose Assumptions \ref{asp: gradient samples}-\ref{asp: Pi} hold, 
	and assume that the stepsize $\alpha$ satisfies
	\begin{equation}
	\label{alpha_ultimate_bound_gossip}
	\alpha\le \frac{2n(1-\rho_{\bar{w}})}{\sqrt{\Gamma}L}\left\{\left[27(2\eta+3) Q n+16(8\eta+9)\right]Q(1-\rho_{\bar{w}})+48(6\eta+1)(8\eta+3)+96Q(1-\rho_{\bar{w}})\right\}^{-1/2},
	\end{equation}
	where $\eta=\frac{1}{n(1-\rho_{\bar{w}})}$ and $Q=L/\mu$. Then $\sup_{l\ge k}\bE[\|\ox_l-x^*\|^2]$ and $\sup_{l\ge k}\bE[\|\mx_{l}-\mathbf{1}\ox_{l}\|^2]$, respectively, converge to $\limsup_{k\rightarrow\infty}\bE[\|\ox_k-x^*\|^2]$ and $\limsup_{k\rightarrow\infty}\bE[\|\mx_k-\mathbf{1}\ox_k\|^2]$ at the linear rate $\mathcal{O}(\rho(\mA_g)^k)$, where $\rho(\mA_g)<1$ is the spectral radius of the matrix $\mA_g$ given by
	\begin{equation*}
	\mA_g=\begin{bmatrix}
	1-\frac{2\alpha\mu}{n} & \frac{2\alpha L^2}{\mu n^2}\left(1+\frac{2\alpha\mu}{n}\right) & \frac{4\alpha^2}{n^3}\\
	8\alpha^2 L^2 & \frac{1}{2}(1+\rho_{\bar{w}}) & \frac{2\alpha}{n}\left(\frac{1}{\beta_1}+\alpha\right)\\
	8\alpha^2L^4+4\alpha L^3 & \frac{L^2}{n}\left(4+\frac{2}{\beta_2}+8\alpha^2 L^2+4\alpha L\right) & \frac{1}{2}(1+\rho_{\bar{w}})
	\end{bmatrix},
	\end{equation*}
	in which $\beta_1=\frac{n(1-\rho_{\bar{w}})}{4\alpha}-4\alpha L^2$ and $\beta_2=\frac{n(1-\rho_{\bar{w}})}{4}-2\alpha L-2\alpha^2 L^2$.
	In addition,
	\begin{equation}
	\label{error_bound_ultimate_gossip}
	\limsup_{k\rightarrow\infty}\bE[\|\ox_k-x^*\|^2]\le \frac{\Gamma}{(\Gamma-1)}\frac{\sigma^2}{n^2}\left[\frac{20\alpha}{\mu (1-\rho_{\bar{w}})}+\frac{42(6\eta+1)\alpha^2 L^2}{\mu^2 (1-\rho_{\bar{w}})^2}\right],
	\end{equation}
	and
	\begin{equation}
	\label{consensus_error_bound_ultimate_gossip}
	\limsup_{k\rightarrow\infty}\bE[\|\mx_k-\mathbf{1}\ox_k\|^2]
	\le \frac{4\Gamma \sigma^2}{(\Gamma-1)(1-\rho_{\bar{w}})^2}\left[\frac{9(6\eta+1)\alpha^2}{n}+\frac{72\alpha^3 L^2}{\mu n^2}\right].
	\end{equation}
\end{theorem}
\begin{remark}
	Notice that $\eta=\frac{1}{n(1-\rho_{\bar{w}})}$ and $1-\rho_{\bar{w}}\le \frac{2}{n}$ from Lemma \ref{lem: spectral norm_gossip}.
	We can see from (\ref{error_bound_ultimate_gossip}) and (\ref{consensus_error_bound_ultimate_gossip}) that
	\begin{equation}
	\label{limsup_error_bound_ultimate_gossip}
	\limsup_{k\rightarrow\infty}\bE[\|\ox_k-x^*\|^2]=\frac{\alpha}{(1-\rho_{\bar{w}})}\mathcal{O}\left(\frac{\sigma^2}{\mu n^2}\right)+\frac{\alpha^2}{(1-\rho_{\bar{w}})^3}\mathcal{O}\left(\frac{ L^2\sigma^2}{\mu^2 n^3}\right),
	\end{equation}
	and
	\begin{equation}
	\label{limsup_consensus_error_bound_ultimate_gossip}
	\limsup_{k\rightarrow\infty}\frac{1}{n}\bE[\|\mx_k-\mathbf{1}\ox_k\|^2]=\frac{\alpha^2}{(1-\rho_{\bar{w}})^3}\mathcal{O}\left(\frac{\sigma^2}{n^3}\right)+\frac{\alpha^3}{(1-\rho_{\bar{w}})^2}\mathcal{O}\left(\frac{L^2\sigma^2}{\mu n^3}\right).
	\end{equation}
	Since 
	\begin{equation*}
	\frac{ L^2\sigma^2}{\mu^2 n^3}\ge \frac{\sigma^2}{n^3},
	\end{equation*}
	and by (\ref{alpha_ultimate_bound_gossip}),
	\begin{equation*}
	\frac{ L^2\sigma^2}{\mu^2 n^3}\ge \alpha(1-\rho_{\bar{w}})\frac{L^2\sigma^2}{\mu n^3},
	\end{equation*}
   the second term on the right-hand side of (\ref{limsup_error_bound_ultimate_gossip}) dominates the two terms on the right-hand side of (\ref{limsup_consensus_error_bound_ultimate_gossip}). Thus, we have
	\begin{multline}
	\label{limsup_x-x*_gossip}
	\limsup_{k\rightarrow\infty}\frac{1}{n}\bE[\|\mx_k-\mathbf{1}x^*\|^2]
	=\limsup_{k\rightarrow\infty}\bE[\|\ox_k-x^*\|^2]+\limsup_{k\rightarrow\infty}\frac{1}{n}\bE[\|\mx_k-\mathbf{1}\ox_k\|^2]\\
	=\frac{\alpha}{(1-\rho_{\bar{w}})}\mathcal{O}\left(\frac{\sigma^2}{\mu n^2}\right)+\frac{\alpha^2}{(1-\rho_{\bar{w}})^3}\mathcal{O}\left(\frac{ L^2\sigma^2}{\mu^2 n^3}\right).
	\end{multline}
\end{remark}
The corollary below provides an upper bound for $\rho(\mA_g)$.
\begin{corollary}
	\label{cor: speed_gossip}
	Under the conditions in Theorem {\ref{theorem2}} where $\Gamma>3/2$, we have
	\begin{equation*}
	\rho(\mA_g)\le 1-\frac{(2\Gamma-3)}{\Gamma}\frac{\alpha\mu}{n}.
	\end{equation*}
\end{corollary}
\begin{remark}
	Compared to DSGT, the convergence speed of GSGT is slower than DSGT under the same stepsize $\alpha$ (see Corollary \ref{cor: speed}). This is due to the fact that in GSGT, only two agents update their iterates at each iteration.
	\end{remark}

\subsection{Performance Comparison between DSGT and GSGT}
\label{subsec: comparison}

In this section, we compare the performances of the two proposed algorithms in terms of their required computation and communication efforts for achieving an $\epsilon$-solution (with constant stepsizes), that is, we compute the number of stochastic gradient computations and communications needed to obtain $\frac{1}{n}\bE[\|\mx_k-\mathbf{1}x^*\|^2]\le \epsilon$. Without loss of generality, for each method we first choose stepsize $\alpha$ such that $\frac{1}{n}\limsup_{k\rightarrow\infty}\bE[\|\mx_k-\mathbf{1}x^*\|^2]\le \epsilon/2$ and then compute the number of iterations $K$ such that $\frac{1}{n}\bE[\|\mx_K-\mathbf{1}x^*\|^2]\le \epsilon$.

For DSGT, when $\epsilon$ is small enough, we have $\alpha=\mathcal{O}(\frac{n\mu\epsilon}{\sigma^2})$ from (\ref{limsup_x-x*}).
Then, noting that $\sup_{l\ge k}\bE[\|\mx_l-\mathbf{1}x^*\|^2]$ converges linearly at the rate $\mathcal{O}(\rho(\mA)^k)$ where $\rho(\mA)=1-\mathcal{O}(\alpha\mu)$, we obtain the number of required iterations:
\begin{equation*}
K_d=\mathcal{O}\left(\frac{\ln(\frac{1}{\epsilon})}{\alpha\mu}\right)=\mathcal{O}\left(\frac{\sigma^2}{n\mu^2}\frac{\ln(\frac{1}{\epsilon})}{\epsilon}\right).
\end{equation*}
In $K_d$ iterations, the number of stochastic gradient computations is $N_d=nK_d=\mathcal{O}\left(\frac{\sigma^2}{\mu^2}\frac{\ln(\frac{1}{\epsilon})}{\epsilon}\right)$ and the number of communications is $N_d^c=2|\mathcal{E}|K_d=\mathcal{O}\left(\frac{|\mathcal{E}|}{n}\frac{\sigma^2}{\mu^2}\frac{\ln(\frac{1}{\epsilon})}{\epsilon}\right)$ where $|\mathcal{E}|$ stands for the number of edges in the graph.

For GSGT we need $\alpha=\mathcal{O}(\frac{n^2\mu\epsilon(1-\rho_{\bar{w}})}{\sigma^2})$ from (\ref{limsup_x-x*_gossip}).  Given that $\rho(\mA_g)=1-\mathcal{O}(\frac{\alpha\mu}{n})$, the number of required iterations $K_g$ can be calculated as follows:
\begin{equation*}
K_g=\mathcal{O}\left(\frac{n\ln(\frac{1}{\epsilon})}{\alpha\mu}\right)=\mathcal{O}\left(\frac{\sigma^2}{n(1-\rho_{\bar{w}})\mu^2}\frac{\ln(\frac{1}{\epsilon})}{\epsilon}\right).
\end{equation*}
In $K_g$ iterations, the number of gradient computations and communications are both bounded by $N_g=N_g^c=2K_g=\mathcal{O}(\frac{1}{n(1-\rho_{\bar{w}})}\frac{\sigma^2}{\mu^2}\frac{\ln(\frac{1}{\epsilon})}{\epsilon})$.

Suppose the Metropolis rule is applied to define the weights $\pi_{ij}$ \cite{sayed2014adaptive}.
We first compare the number of stochastic gradient computations for DSGT and GSGT, respectively. Noticing that $1-\rho_{\bar{w}}\le \frac{2}{n}$ by Lemma \ref{lem: spectral norm_gossip}, $N_g$ is at most in the same order of $N_d=\mathcal{O}(\frac{\sigma^2}{\mu^2}\frac{\ln(\frac{1}{\epsilon})}{\epsilon})$, which happens when $1-\rho_{\bar{w}}=\mathcal{O}(\frac{1}{n})$. Given that 
$1-\rho_{\bar{w}}=(1-\lambda_2(\mPi))/n$ from (\ref{equation:rho_bar w}),  we have $1-\rho_{\bar{w}}=\mathcal{O}(\frac{1}{n})$ for complete networks, almost all regular graphs \cite{friedman1989second}, among others.

We then compare the number of required communications $N_d^c=\mathcal{O}(\frac{|\mathcal{E}|}{n}\frac{\sigma^2}{\mu^2}\frac{\ln(\frac{1}{\epsilon})}{\epsilon})$ and $N_g^c=\mathcal{O}(\frac{1}{n(1-\rho_{\bar{w}})}\frac{\sigma^2}{\mu^2}\frac{\ln(\frac{1}{\epsilon})}{\epsilon})$. When $1-\rho_{\bar{w}}=\mathcal{O}(\frac{1}{n})$, we have $N_g^c=\mathcal{O}(\frac{\sigma^2}{\mu^2}\frac{\ln(\frac{1}{\epsilon})}{\epsilon})$. By contrast,  $N_d^c$ is $\mathcal{O}(\frac{|\mathcal{E}|}{n})$ times larger than $N_g^c$. In particular when $|\mathcal{E}|=\mathcal{O}(n^2)$ (e.g., complete network), the number of communications for GSGT is $\mathcal{O}(n)$ times smaller than that of DSGT.


\subsection{Proof of Theorem \ref{thm_gossip}}

We first derive a linear system of inequalities regarding $\bE[\|\ox_k-x^*\|^2]$, $\bE[\|\mx_k-\mathbf{1}\ox_k\|^2]$, $\bE[\|\my_k-\mathbf{1}\oy_k\|^2]$ and their values in the last iteration. 
\begin{lemma}
	\label{lem: gossip three main inequalities}
	Suppose Assumptions \ref{asp: gradient samples}-\ref{asp: Pi} hold and the stepsize satisfies $\alpha<n/(\mu+L)$. Then, we have the following inequalities:
	\begin{multline}
	\label{gossip: first main inequality}
	\bE[\|\ox_{k+1}-x^*\|^2]\le
	\left(1-\frac{2\alpha\mu}{n}\right)\bE[\|\ox_k-x^*\|^2]+\frac{2\alpha L^2}{\mu n^2}\left(1+\frac{2\alpha\mu}{n}\right)\bE[\|\mx_k-\mathbf{1}\ox_k\|^2]+\frac{4\alpha^2}{n^3}\bE[\|\my_k-\mathbf{1}\oy_k\|^2]\\
	+\frac{4\alpha^2\sigma^2}{n^3}.
	\end{multline}
	For any $\beta_1,\beta_2>0$,
	\begin{multline}
	\label{gossip: second main inequality}
	\bE[\|\mx_{k+1}-\mathbf{1}\ox_{k+1}\|^2]
	\le\left(\rho_{\bar{w}}+\frac{2\alpha}{n}\beta_1+\frac{8\alpha^2 L^2}{n}\right)\bE[\|\mx_k-\mathbf{1}\ox_k\|^2]+\frac{2\alpha}{n}\left(\frac{1}{\beta_1}+\alpha\right)\bE[\|\my_k-\mathbf{1}\oy_k\|^2]\\
	+8\alpha^2 L^2\bE[\|\ox_k-x^*\|^2]+\frac{4\alpha^2 \sigma^2}{n},
	\end{multline}
	\begin{multline}
	\label{gossip: third main inequality}
	\bE[\|\my_{k+1}-\mathbf{1}\oy_{k+1}\|^2]\le \left(\rho_{\bar{w}}+\frac{2}{n}\beta_2+\frac{4\alpha L}{n}+\frac{4\alpha^2L^2}{n}\right)\bE[\|\my_k-\mathbf{1}\oy_k\|^2]\\
	+\frac{L^2}{n}\left(4+\frac{2}{\beta_2}+8\alpha^2 L^2+4\alpha L\right)\bE[\|\mx_k-\mathbf{1}\ox_k\|^2]+(8\alpha^2L^4+4\alpha L^3)\bE[\|\ox_k-x^*\|^2]
	+\frac{(4\alpha^2L^2+2\alpha L)\sigma^2}{n}\\
	+4(\alpha L+1)\sigma^2.
	\end{multline}
\end{lemma}
\begin{proof}
	See Appendix \ref{appendix: lem gossip three main inequalities}.
\end{proof}
In light of Lemma \ref{lem: gossip three main inequalities}, 
we have the following linear system of inequalities:
\begin{equation*}
\begin{bmatrix}
\bE[\|\ox_{k+1}-x^*\|^2]\\
\bE[\|\mx_{k+1}-\mathbf{1}\ox_{k+1}\|^2]\\
\bE[\|\my_{k+1}-\mathbf{1}\oy_{k+1}\|^2]
\end{bmatrix}
\le
\mA_g \begin{bmatrix}
\bE[\|\ox_{k}-x^*\|^2]\\
\bE[\|\mx_{k}-\mathbf{1}\ox_{k}\|^2]\\
\bE[\|\my_{k}-\mathbf{1}\oy_{k}\|^2]
\end{bmatrix}
+\begin{bmatrix}
\frac{4\alpha^2\sigma^2}{n^3}\\
\frac{4\alpha^2 \sigma^2}{n}\\
M_g
\end{bmatrix},
\end{equation*}
where
\begin{equation*}
\mA_g=[b_{ij}]=\begin{bmatrix}
1-\frac{2\alpha\mu}{n} & \frac{2\alpha L^2}{\mu n^2}\left(1+\frac{2\alpha\mu}{n}\right) & \frac{4\alpha^2}{n^3}\\
8\alpha^2 L^2 & \rho_{\bar{w}}+\frac{2\alpha}{n}\beta_1+\frac{8\alpha^2 L^2}{n} & \frac{2\alpha}{n}\left(\frac{1}{\beta_1}+\alpha\right)\\
8\alpha^2L^4+4\alpha L^3 & \frac{L^2}{n}\left(4+\frac{2}{\beta_2}+8\alpha^2 L^2+4\alpha L\right) & \rho_{\bar{w}}+\frac{2}{n}\beta_2+\frac{4\alpha L}{n}+\frac{4\alpha^2L^2}{n}
\end{bmatrix},
\end{equation*}
and $M_g=\frac{(4\alpha^2L^2+2\alpha L)\sigma^2}{n}+4(\alpha L+1)\sigma^2$. 
Suppose $\alpha$, $\beta_1,\beta_2>0$ satisfy
\begin{align}
& b_{22}=\rho_{\bar{w}}+\frac{2\alpha}{n}\beta_1+\frac{8\alpha^2 L^2}{n}  = \frac{1+\rho_{\bar{w}}}{2},\label{beta_1 condition}\\
& b_{33}=\rho_{\bar{w}}+\frac{2}{n}\beta_2+\frac{4\alpha L}{n}+\frac{4\alpha^2L^2}{n} = \frac{1+\rho_{\bar{w}}}{2},\label{beta_2 condition}
\end{align}
and
\begin{multline}
\label{det mI-mA>0}
\text{det}(\mI-\mA_g)=(1-b_{11})(1-b_{22})(1-b_{33})-b_{12}b_{23}b_{31}-b_{13}b_{21}b_{32}-(1-b_{11})b_{23}b_{32}-(1-b_{22})b_{13}b_{31}\\
-(1-b_{33})b_{12}b_{21}\ge (1-1/\Gamma)(1-b_{11})(1-b_{22})(1-b_{33})>0.
\end{multline}
Then, by Lemma \ref{lem: rho_M}, the spectral radius of $\mA_g$ is smaller than $1$,
and we have
\begin{eqnarray}
\label{linear_system_bound_gossip}
\begin{bmatrix}
\bE[\|\ox_{k}-x^*\|^2]\\
\bE[\|\mx_{k}-\mathbf{1}\ox_{k}\|^2]\\
\bE[\|\my_{k}-\mathbf{1}\oy_{k}\|^2]
\end{bmatrix}
\le 
\mA_g^k\begin{bmatrix}
\bE[\|\ox_{0}-x^*\|^2]\\
\bE[\|\mx_{0}-\mathbf{1}\ox_{0}\|^2]\\
\bE[\|\my_{0}-\mathbf{1}\oy_{0}\|^2]
\end{bmatrix}+\sum_{l=0}^{k-1}\mA_g^l
\begin{bmatrix}
\frac{4\alpha^2\sigma^2}{n^3}\\
\frac{4\alpha^2 \sigma^2}{n}\\
M_g
\end{bmatrix}.
\end{eqnarray}
Hence, 
$\sup_{l\ge k}\bE[\|\ox_l-x^*\|^2]$, $\sup_{l\ge k}\bE[\|\mx_l-\mathbf{1}\ox_l\|^2]$ and $\sup_{l\ge k}\bE[\|\my_l-\mathbf{1}\oy_l\|^2]$ all converge to a neighborhood of $0$ at the linear rate $\mathcal{O}(\rho(\mA_g)^k)$. 
Moreover,
\begin{multline}
\label{gossip_convergence pre}
\begin{bmatrix}
\limsup_{k\rightarrow\infty}\bE[\|\ox_k-x^*\|^2]\\
\limsup_{k\rightarrow\infty}\bE[\|\mx_k-\mathbf{1}\ox_{k}\|^2]\\
\limsup_{k\rightarrow\infty}\bE[\|\my_k-\mathbf{1}\oy_{k}\|^2]
\end{bmatrix}
\le
(\mI-\mA_g)^{-1}
\begin{bmatrix}
\frac{4\alpha^2\sigma^2}{n^3}\\
\frac{2\alpha^2 \sigma^2}{n}\\
M_g
\end{bmatrix}
=\frac{1}{\text{det}(\mI-\mA)}	\\
\cdot\begin{bmatrix}
(1-b_{22})(1-b_{33})-b_{23}b_{32} & b_{13}b_{32}+b_{12}(1-b_{33}) & b_{12}b_{23}+b_{13}(1-b_{22})\\
b_{23}b_{31}+b_{21}(1-b_{33}) & (1-b_{11})(1-b_{33})-b_{13}b_{31} & b_{13}b_{21}+b_{23}(1-b_{11})\\
b_{21}b_{32}+b_{31}(1-b_{22}) & b_{12}b_{31}+b_{32}(1-b_{11}) & (1-b_{11})(1-b_{22})-b_{12}b_{21}
\end{bmatrix}
\begin{bmatrix}
\frac{4\alpha^2\sigma^2}{n^3}\\
\frac{2\alpha^2 \sigma^2}{n}\\
M_g
\end{bmatrix}.
\end{multline}
We now show (\ref{beta_1 condition}), (\ref{beta_2 condition}), and (\ref{gossip_convergence pre}) are satisfied under condition (\ref{alpha_ultimate_bound_gossip}).
First, relation (\ref{alpha_ultimate_bound_gossip}) implies that
\begin{align}
& 4\alpha^2 L^2\le \frac{n(1-\rho_{\bar{w}})}{12}, \label{alpha_bound_pre1}\\
& 2\alpha L+2\alpha^2 L^2\le  \frac{n(1-\rho_{\bar{w}})}{12}. \label{alpha_bound_pre2}
\end{align}
Therefore, from (\ref{beta_1 condition}) and (\ref{beta_2 condition}) we have
\begin{align}
& \beta_1=\frac{n(1-\rho_{\bar{w}})}{4\alpha}-4\alpha L^2\ge \frac{n(1-\rho_{\bar{w}})}{6\alpha}>0, \label{beta_1}\\
& \beta_2=\frac{n(1-\rho_{\bar{w}})}{4}-2\alpha L-2\alpha^2 L^2\ge  \frac{n(1-\rho_{\bar{w}})}{6}>0. \label{beta_2}
\end{align}
By (\ref{alpha_bound_pre1})-(\ref{beta_2}) and the fact that $\rho_{\bar{w}}\ge 1-2/n$ obtained from (\ref{bar W}), we have
\begin{subequations}
	\label{bounds on bs}
	\begin{align}
	b_{12}\le & \frac{2\alpha L^2}{\mu n^2}\left(1+\frac{1}{8}\right) \le \frac{9\alpha L^2}{4\mu n^2},\\
	b_{23}\le & \frac{2\alpha^2}{n}\left(6\eta+1\right),\\
	b_{31}= & \alpha L^3(8\alpha L+4)\le 6\alpha L^3,\\
	b_{32}\le & \frac{L^2}{n}\left(12\eta+\frac{9}{2}\right).
	\end{align}
\end{subequations}
Then, for relation (\ref{det mI-mA>0}) to hold, it is sufficient that
\begin{multline*}
\frac{1}{\Gamma}\frac{\alpha\mu}{2n}(1-\rho_{\bar{w}})^2\ge \frac{27(6\eta+1)\alpha^4L^5}{\mu n^3}+\frac{48(8\eta+3)\alpha^4 L^4}{n^4}+\frac{6(6\eta+1)(8\eta+3)\alpha^3\mu L^2}{n^3}
+\frac{12\alpha^3 L^3}{n^3}(1-\rho_{\bar{w}})\\
+\frac{9\alpha^3L^4}{\mu n^2}(1-\rho_{\bar{w}}).
\end{multline*}
In light of (\ref{alpha_bound_pre1}), $\alpha L\le n(1-\rho_{\bar{w}})/24$. We only need
\begin{multline*}
\frac{1}{\Gamma}\frac{\mu n^2}{2}(1-\rho_{\bar{w}})^2\ge \frac{9(6\eta+1)\alpha^2L^4 n}{8\mu}(1-\rho_{\bar{w}})+2(8\eta+3)\alpha^2 L^3(1-\rho_{\bar{w}})+6(6\eta+1)(8\eta+3)\alpha^2\mu L^2\\
+12\alpha^2 L^3(1-\rho_{\bar{w}})+\frac{9\alpha^2L^4 n}{\mu}(1-\rho_{\bar{w}}),
\end{multline*}
which gives
\begin{equation*}
\alpha\le \frac{2n(1-\rho_{\bar{w}})}{\sqrt{\Gamma }L}\left\{\left[27(2\eta+3) Q n+16(8\eta+9)\right]Q(1-\rho_{\bar{w}})+48(6\eta+1)(8\eta+3)+96Q(1-\rho_{\bar{w}})\right\}^{-1/2}.
\end{equation*}

We now derive the bounds for $\limsup_{k\rightarrow\infty}\bE[\|\ox_{k+1}-x^*\|^2]$ and $\limsup_{k\rightarrow\infty}\bE[\|\mx_{k+1}-\mathbf{1}\ox_{k+1}\|^2]$. By (\ref{gossip_convergence pre}) and (\ref{bounds on bs}),
\begin{multline*}
\limsup_{k\rightarrow\infty}\bE[\|\ox_{k+1}-x^*\|^2] \le \frac{\Gamma}{(\Gamma-1)(1-b_{11})(1-b_{22})(1-b_{33})}\\
\cdot\left\{[(1-b_{22})(1-b_{33})-b_{23}b_{32}]\frac{4\alpha^2\sigma^2}{n^3}+[b_{13}b_{32}+b_{12}(1-b_{33})]\frac{2\alpha^2 \sigma^2}{n}+[b_{12}b_{23}+b_{13}(1-b_{22})]M_g\right\}\\
\le \frac{2\Gamma n}{(\Gamma-1)\alpha\mu(1-\rho_{\bar{w}})^2}\Bigg\{\frac{\alpha^2\sigma^2(1-\rho_{\bar{w}})^2}{n^3}+\left[\frac{6\alpha^2L^2}{n^4}(8\eta+3)+\frac{9\alpha L^2}{8\mu n^2}(1-\rho_{\bar{w}})\right]\frac{2\alpha^2 \sigma^2}{n}\\
+\left[\frac{9(6\eta+1)\alpha^3 L^2}{2\mu n^3}+\frac{2\alpha^2}{n^3}(1-\rho_{\bar{w}})\right]\frac{9}{2}\sigma^2\Bigg\}\\
\le \frac{2\Gamma n\sigma^2}{(\Gamma-1)\alpha\mu(1-\rho_{\bar{w}})^2}\left[\frac{10\alpha^2(1-\rho_{\bar{w}})}{n^3}+\frac{21(6\eta+1)\alpha^3 L^2}{\mu n^3}\right]
=\frac{\Gamma}{(\Gamma-1)}\frac{\sigma^2}{n^2}\left[\frac{20\alpha}{\mu (1-\rho_{\bar{w}})}+\frac{42(6\eta+1)\alpha^2 L^2}{\mu^2 (1-\rho_{\bar{w}})^2}\right].
\end{multline*}
\begin{multline*}
\limsup_{k\rightarrow\infty}\bE[\|\mx_k-\mathbf{1}\ox_k\|^2]\le \frac{\Gamma}{(\Gamma-1)(1-b_{11})(1-b_{22})(1-b_{33})}\\
\cdot\left\{[b_{23}b_{31}+b_{21}(1-b_{33})]\frac{4\alpha^2\sigma^2}{n^3}+[(1-b_{11})(1-b_{33})-b_{13}b_{31}]\frac{2\alpha^2 \sigma^2}{n}+[b_{13}b_{21}+b_{23}(1-b_{11})]M_g\right\}\\
\le \frac{2\Gamma n}{(\Gamma-1)\alpha\mu(1-\rho_{\bar{w}})^2}\Bigg\{\left[\frac{12(6\eta+1)\alpha^3 L^3}{n}+4\alpha^2 L^2(1-\rho_{\bar{w}})\right]\frac{4\alpha^2\sigma^2}{n^3}\\
+\left[\frac{\alpha\mu(1-\rho_{\bar{w}})}{n}-\frac{24\alpha^3L^3}{n^3}\right]\frac{2\alpha^2 \sigma^2}{n}+\left[\frac{32\alpha^4 L^2}{n^3}+\frac{4(6\eta+1)\alpha^3\mu}{n^2}\right]\frac{17}{4}\sigma^2\Bigg\}\\
\le \frac{2\Gamma n\sigma^2}{(\Gamma-1)\alpha\mu(1-\rho_{\bar{w}})^2}\left[\frac{18(6\eta+1)\alpha^3\mu}{n^2}+\frac{136\alpha^4 L^2}{n^3}\right]=\frac{4\Gamma \sigma^2}{(\Gamma-1)(1-\rho_{\bar{w}})^2}\left[\frac{9(6\eta+1)\alpha^2}{n}+\frac{72\alpha^3 L^2}{\mu n^2}\right].
\end{multline*}

\subsection{Proof of Corollary \ref{cor: speed_gossip}}
\label{subsec: proof cor_speed_gossip}
The characteristic function of $\mA_g$ is 
\begin{multline}
\text{det}(\lambda\mI-\mA_g)=(\lambda-b_{11})(\lambda-b_{22})(\lambda-b_{33})-b_{12}b_{23}b_{31}-b_{13}b_{21}b_{32}-(\lambda-b_{11})b_{23}b_{32}-(\lambda-b_{22})b_{13}b_{31}\\
-(\lambda-b_{33})b_{12}b_{21}.
\end{multline}
By (\ref{det mI-mA>0}),
\begin{multline}
\text{det}(\lambda\mI-\mA_g)\ge(\lambda-b_{11})(\lambda-b_{22})(\lambda-b_{33})+(1-\lambda)b_{23}b_{32}+(1-\lambda)b_{13}b_{31}+(1-\lambda)b_{12}b_{21}\\
-\frac{1}{\Gamma}(1-b_{11})(1-b_{22})(1-b_{33})\ge(\lambda-b_{11})(\lambda-b_{22})(\lambda-b_{33})
-\frac{1}{\Gamma}(1-b_{11})(1-b_{22})(1-b_{33}).
\end{multline}
Let $\lambda=1-\epsilon$ for some $\epsilon\in(0,2\alpha\mu/n)$ that satisfies
\begin{equation*}
\text{det}(\lambda\mI-\mA_g)\ge\left(\frac{2\alpha\mu}{n}-\epsilon\right)\left[\frac{1-\rho_{\bar{w}}}{2}-\epsilon\right]^2- \frac{1}{\Gamma}\frac{2\alpha\mu}{n}\frac{(1-\rho_{\bar{w}})^2}{4}\ge 0.
\end{equation*}
Under condition (\ref{alpha_ultimate_bound_gossip}), it suffices that 
\begin{equation*}
\epsilon\le \frac{(2\Gamma-3)}{\Gamma}\frac{\alpha\mu}{n}.
\end{equation*}
Denote $\tilde{\lambda}=1-\frac{(2\Gamma-3)}{\Gamma}\frac{\alpha\mu}{n}$. We have $\text{det}(\tilde{\lambda}\mI-\mA_g)\ge 0$, and therefore $\rho(\mA_g)\le \tilde{\lambda}$.

\section{Numerical Example}
\label{sec: simulation}

In this section, we provide a numerical example to illustrate our theoretic findings. 
Consider the \emph{on-line} Ridge regression problem, i.e.,
\begin{equation}
\label{Ridge Regression}
\min_{x\in \mathbb{R}^{p}}f(x)=\frac{1}{n}\sum_{i=1}^nf_i(x)\left(=\mathbb{E}_{u_i,v_i}\left[\left(u_i^{\T} x-v_i\right)^2+\rho\|x\|^2\right]\right),
\end{equation}
where $\rho>0$ is a penalty parameter.
For each agent $i$, samples in the form of $(u_i,v_i)$ are gathered continuously with $u_i\in\mathbb{R}^p$ representing the features and $v_i\in\mathbb{R}$ being the observed outputs. We assume that each $u_i\in[0.3,0.4]^p$ is uniformly distributed, and $v_i$ is drawn according to $v_i=u_i^{\T} \tilde{x}_i+\varepsilon_i$, where $\tilde{x}_i$ are predefined parameters evenly located in $[0,10]^p$, and $\varepsilon_i$ are independent Gaussian noises with mean $0$ and variance $1$.
Given a pair $(u_i,v_i)$, agent $i$ can compute an estimated gradient of $f_i(x)$: $g_i(x,u_i,v_i)=2(u_i^{\T}x -v_i)u_i+2\rho x$, which is unbiased.
Problem (\ref{Ridge Regression}) has a unique solution $x^*$ given by $x^*=(\sum_{i=1}^n\mathbb{E}_{u_i}[u_iu_i^{\T}]+n\rho\mathbf{I})^{-1}\sum_{i=1}^n\mathbb{E}_{u_i}[u_iu_i^{\T}]\tilde{x}_i$.

In addition to DSGT, GSGT and CSG, we consider the following distributed stochastic gradient (DSG) algorithm, which is similar to the ones studied in \cite{jakovetic2018convergence,lian2017can}:
\begin{equation}
\label{DSG}
\mx_{k+1} = \mW\mx_k-\alpha G(\mx_k,\boldsymbol{\xi}_k).
\end{equation}
Noticing that some existing algorithms for deterministic distributed optimization can also be adapted to the stochastic gradient setting, e.g., EXTRA \cite{shi2015extra} and decentralized ADMM (DLM) \cite{ling2015dlm}, we also include them in our experiments for comparison.

In the experiments, we consider $3$ instances with $p=20$ and $n\in\{10,25,100\}$, respectively. Under each instance, we let $\mx_0=\mathbf{0}$ and the penalty parameter $\rho=0.1$. For the distributed methods, we assume that $n$ agents constitute a random network, in which each two agents are linked with probability $0.4$. The Metropolis rule is applied to define the weights $w_{ij}$ (and $\pi_{ij}$) where applicable \cite{sayed2014adaptive}:
\begin{equation*}
w_{ij}=\begin{cases}
1/\max\{\degree(i),\degree(j)\} & \text{if }i\in \mathcal{N}_i,  \\
1- \sum_{j\in\mathcal{N}_i}w_{ij} & \text{if }i=j,\\
0 & \text{otherwise}.
\end{cases}
\end{equation*}
For EXTRA, we choose $\tilde{\mW}=\frac{\mI+\mW}{2}$ as recommended by \cite{shi2015extra}.
For DLM, we tune the free parameters to make its convergence speed comparable to the other algorithms.
In each instance, we use two different stepsizes $\alpha=5\times10^{-3}$ and $\alpha=5\times10^{-2}$, respectively. We run the simulations $50$ times for DSGT, CSG, DSG, EXTRA and DLM and $100$ times for GSGT and average the results to approximate the expected errors.

In Figure \ref{fig: comparison} (a)-(f), we compare the average performances of DSGT, GSGT, CSG, DSG, EXTRA and DLM with the same parameters. It can be seen that DSGT and CSG are comparable in their convergence speeds as well as the ultimate error bounds (almost indistinguishable). EXTRA and DLM are  worse than DSGT and CSG in their final error bounds. The performance gap increases with the network size and the stepsize.
GSGT is slower as expected but still reaches a comparable error level under small stepsize $\alpha=5\times10^{-3}$. In addition, the error bounds for DSGT, GSGT and CSG decrease in $n$ as expected from our theoretical analysis. The performance of DSG is not favorable given its largest final errors.\footnote{DSG still holds the advantage over DSGT in the early stage for achieving a similar convergence speed  with lower communication and storage costs.}

In Figure \ref{fig: comparison} (g)(h)(i) (respectively, (j)(k)(l)), we further compare the solutions obtained under DSGT and GSGT with the same number of stochastic gradient evaluations (respectively, inter-node communications) under small stepsize $\alpha=5\times10^{-3}$. 
We see the two methods are comparable in their speeds of convergence w.r.t the number of gradient evaluations. However, GSGT is much faster than DSGT assuming the same number of communications. These numerical results verified our arguments in Section \ref{subsec: comparison}.
\begin{figure}
	\centering
	\subfigure[Instance $(p,n)=(20,10)$, $\alpha=5\times10^{-3}$.]{\includegraphics[width=2.1in]{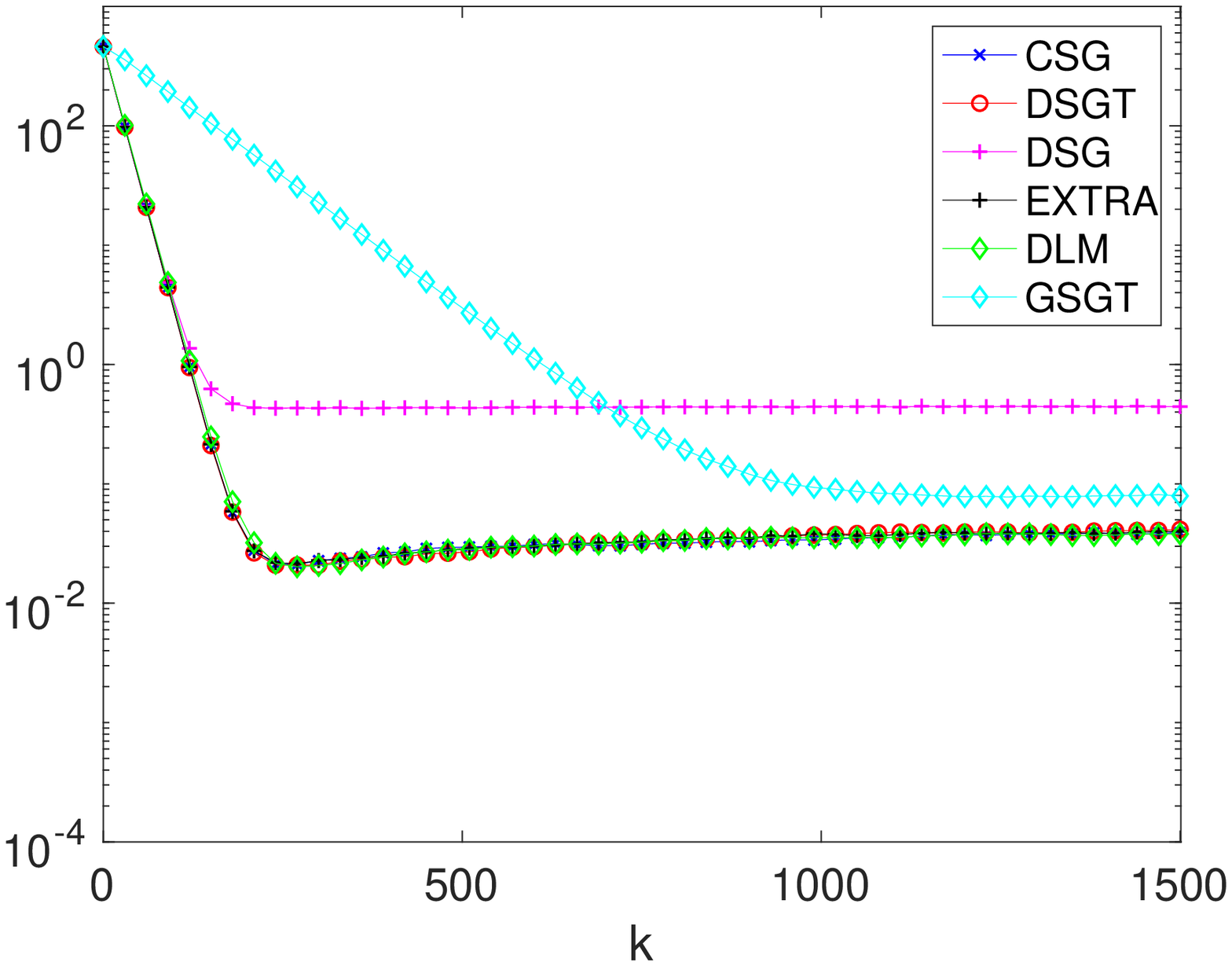}} 
	\subfigure[Instance $(p,n)=(20,25)$, $\alpha=5\times10^{-3}$.]{\includegraphics[width=2.1in]{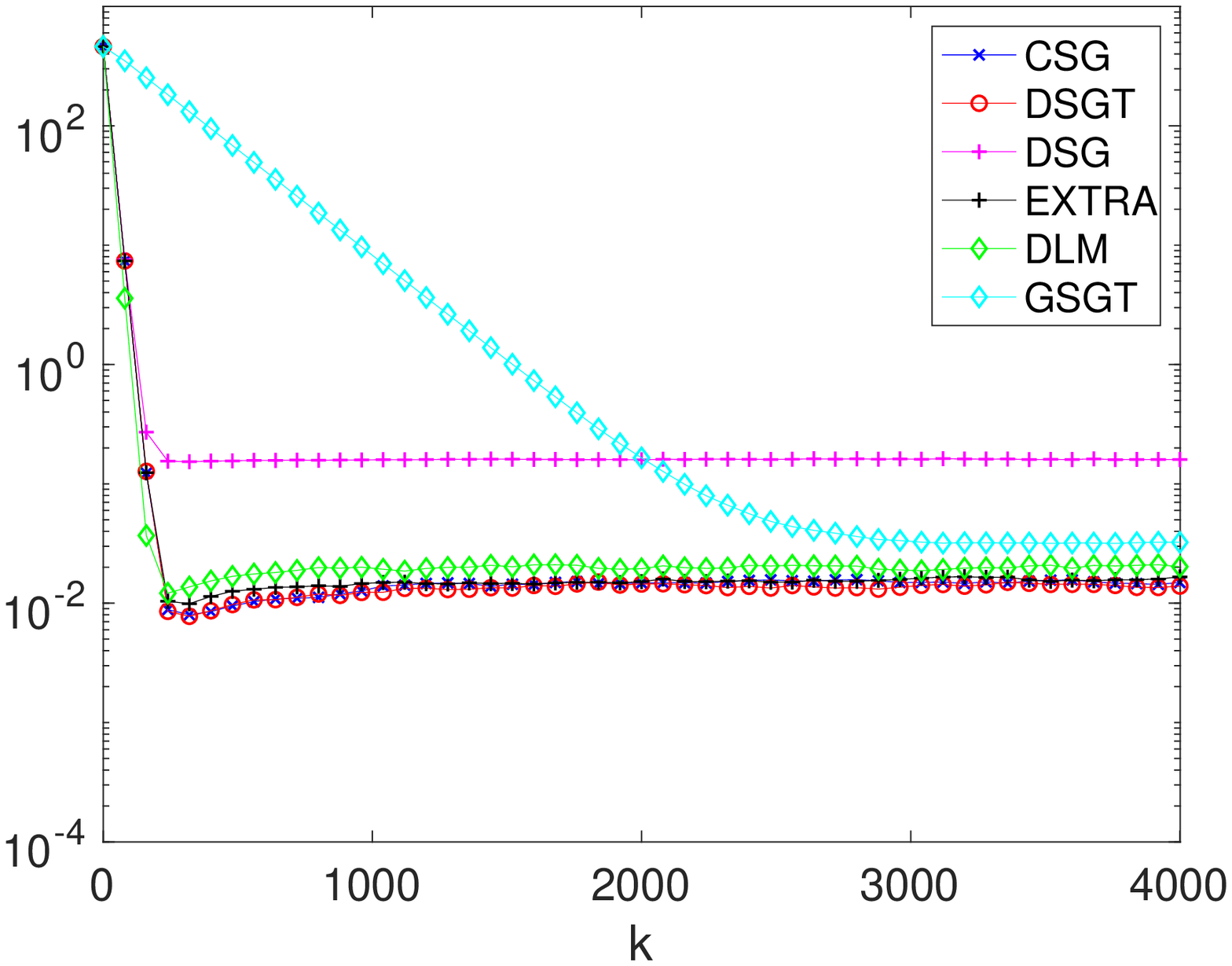}} 
	\subfigure[Instance $(p,n)=(20,100)$, $\alpha=5\times10^{-3}$.]{\includegraphics[width=2.1in]{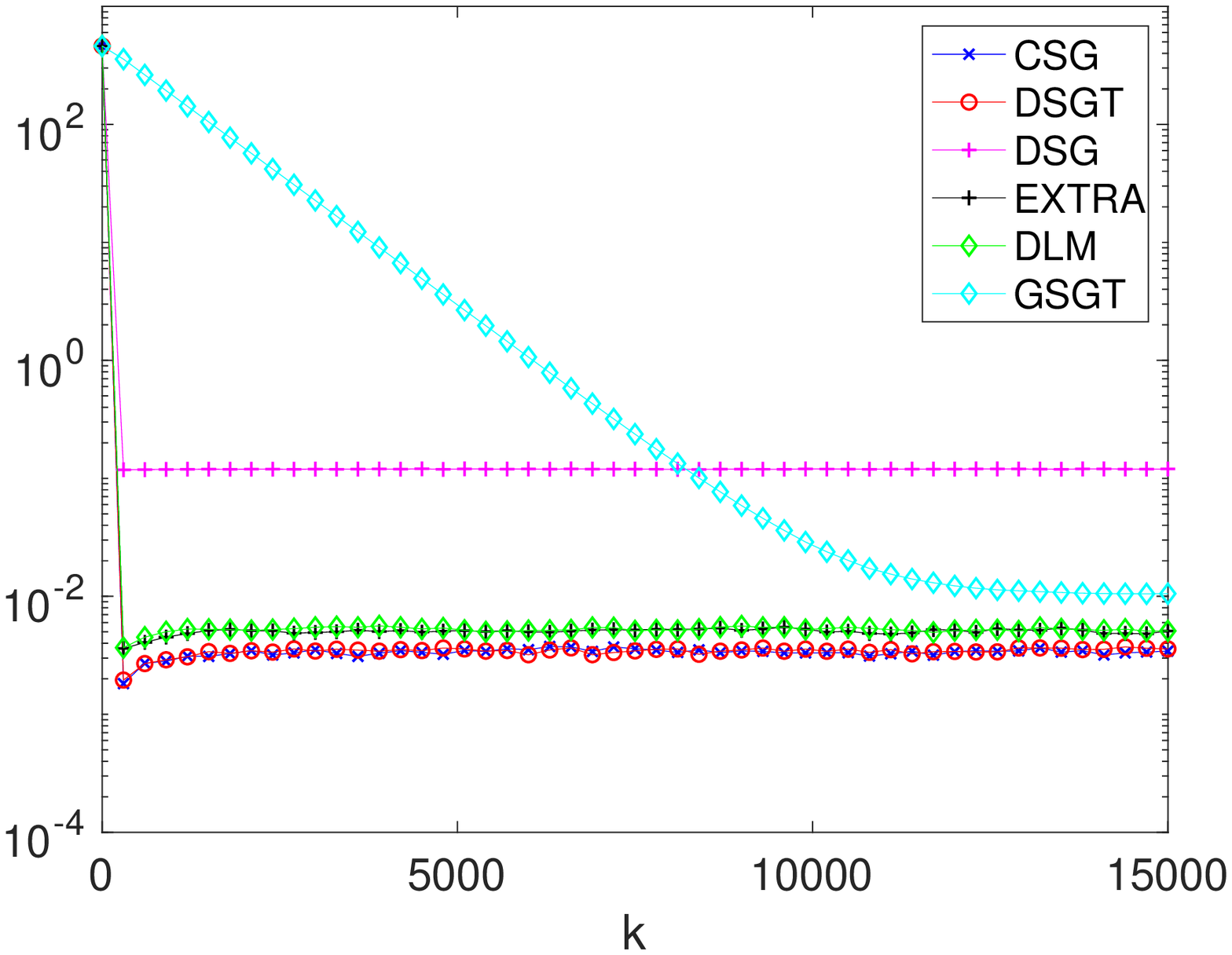}} 
	
	\subfigure[Instance $(p,n)=(20,10)$, $\alpha=5\times10^{-2}$.]{\includegraphics[width=2.1in]{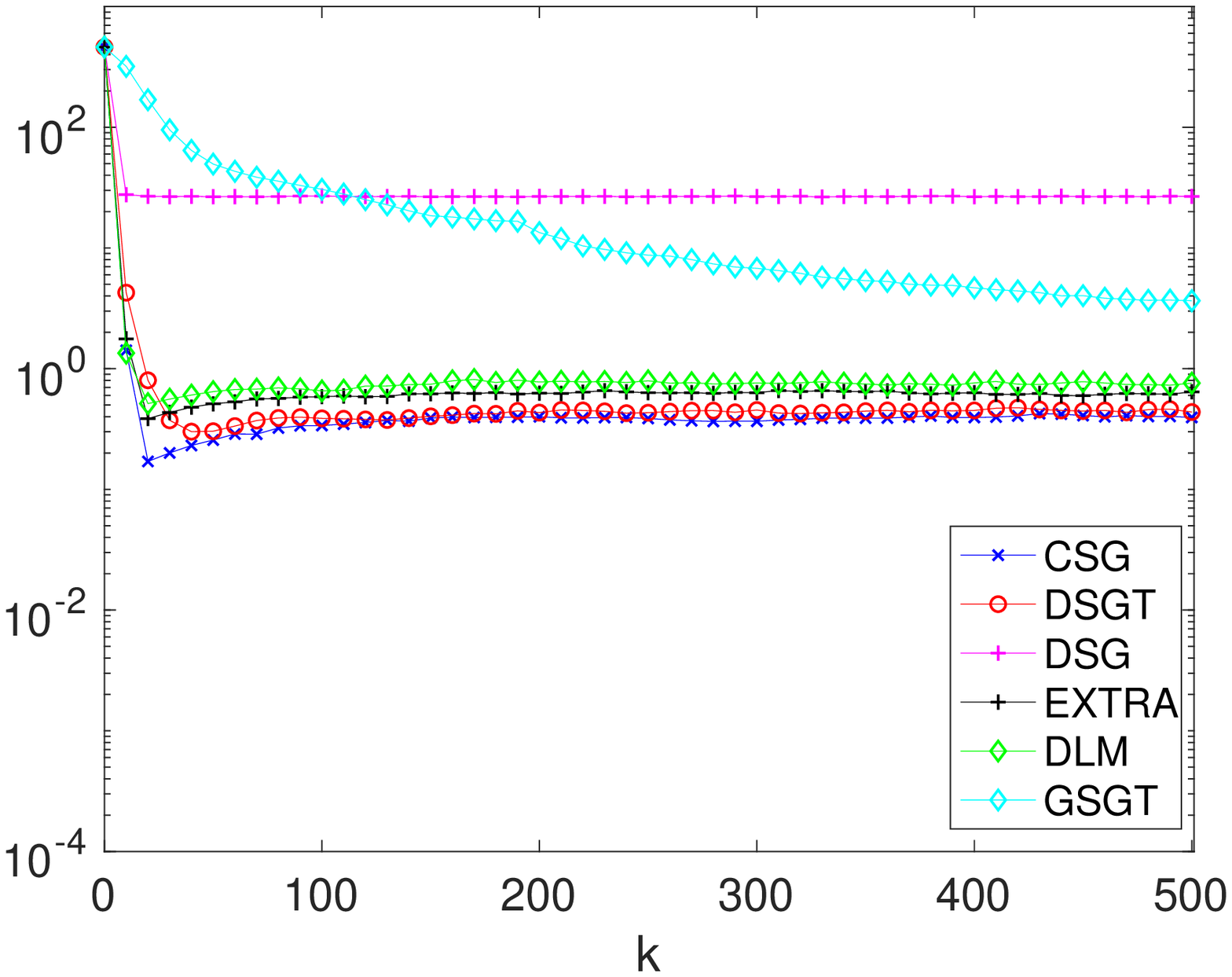}} 
	\subfigure[Instance $(p,n)=(20,25)$, $\alpha=5\times10^{-2}$.]{\includegraphics[width=2.1in]{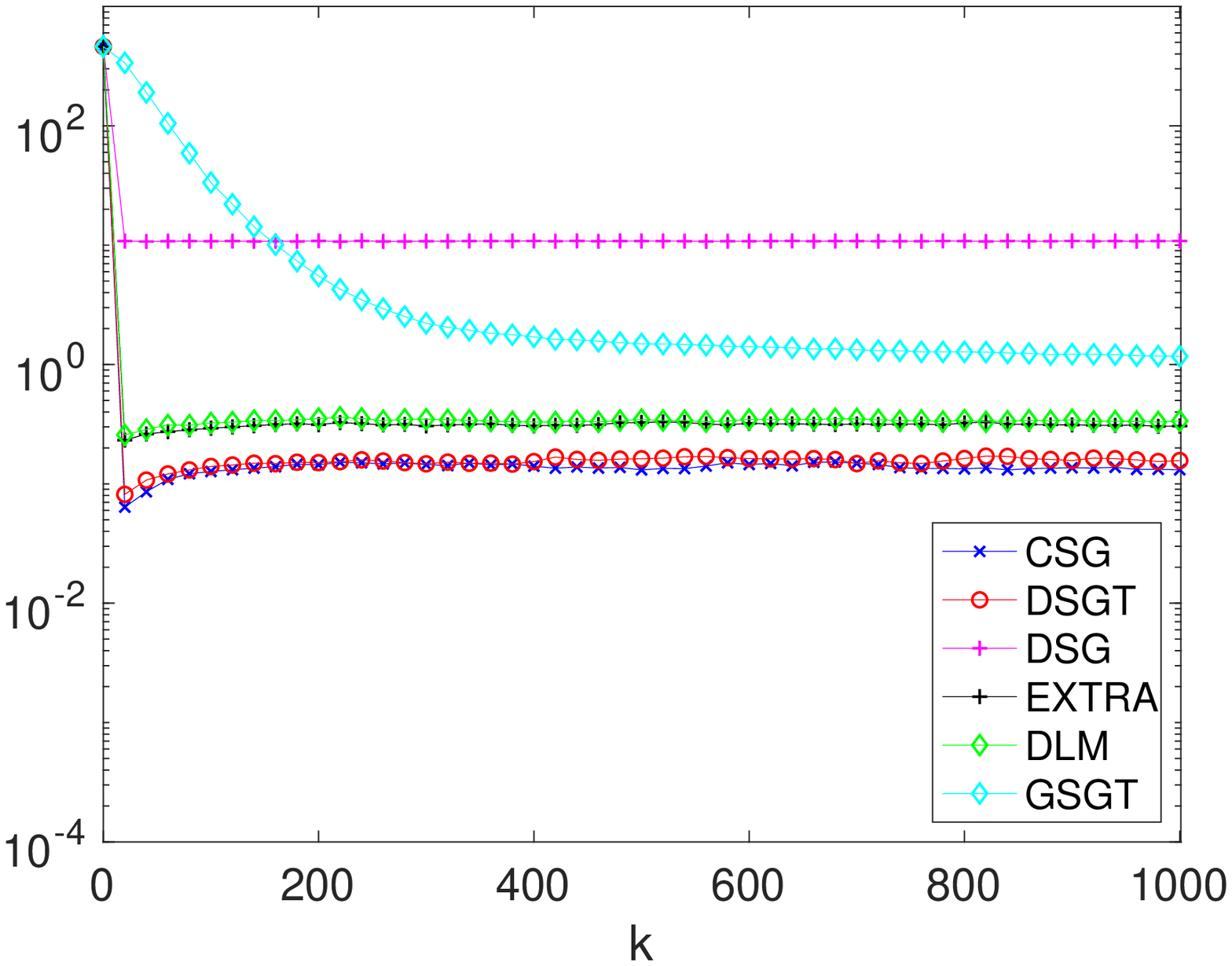}} 
	\subfigure[Instance $(p,n)=(20,100)$, $\alpha=5\times10^{-2}$.]{\includegraphics[width=2.1in]{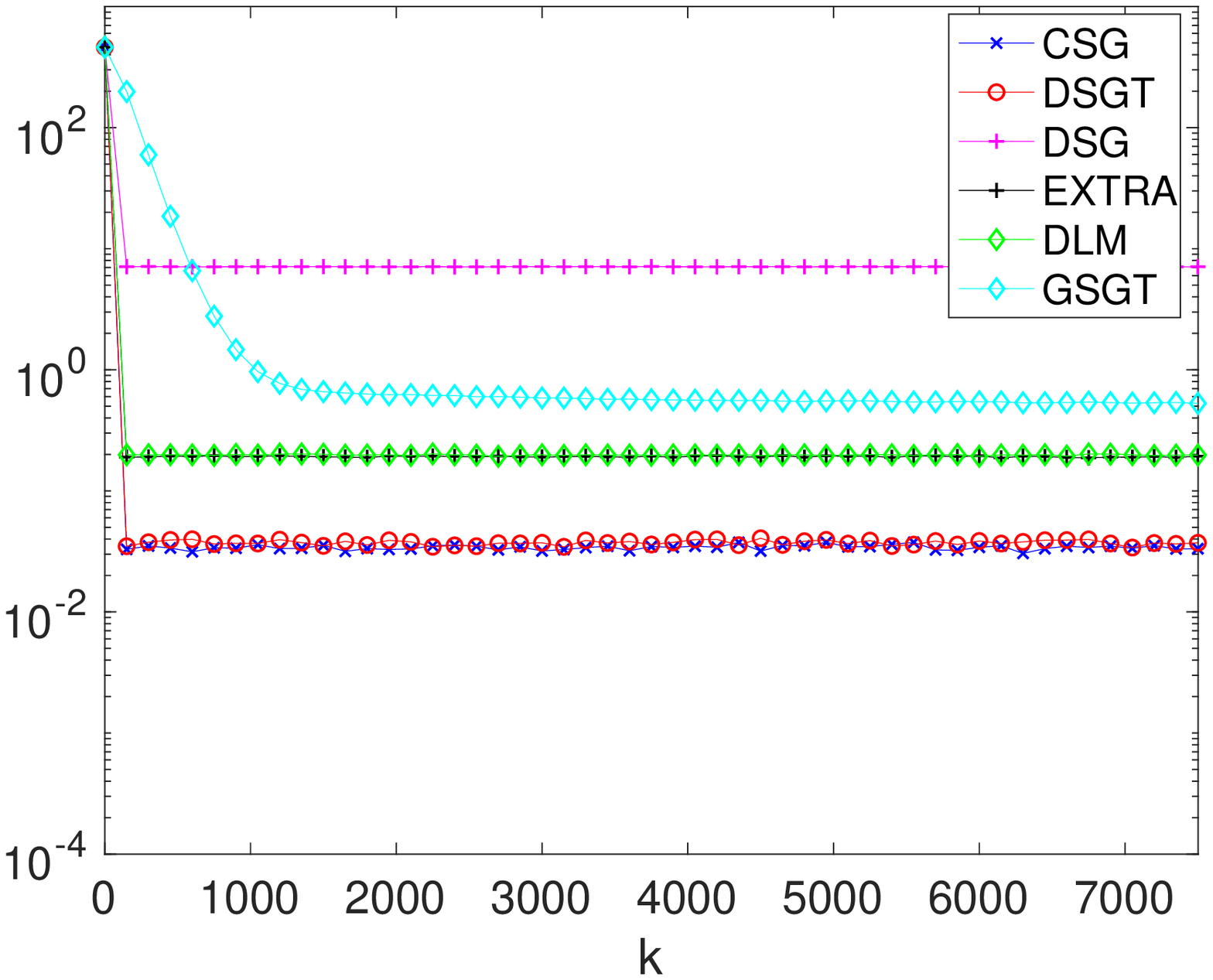}} 
	
	\subfigure[Instance $(p,n)=(20,10)$, $\alpha=5\times10^{-3}$.]{\includegraphics[width=2.1in]{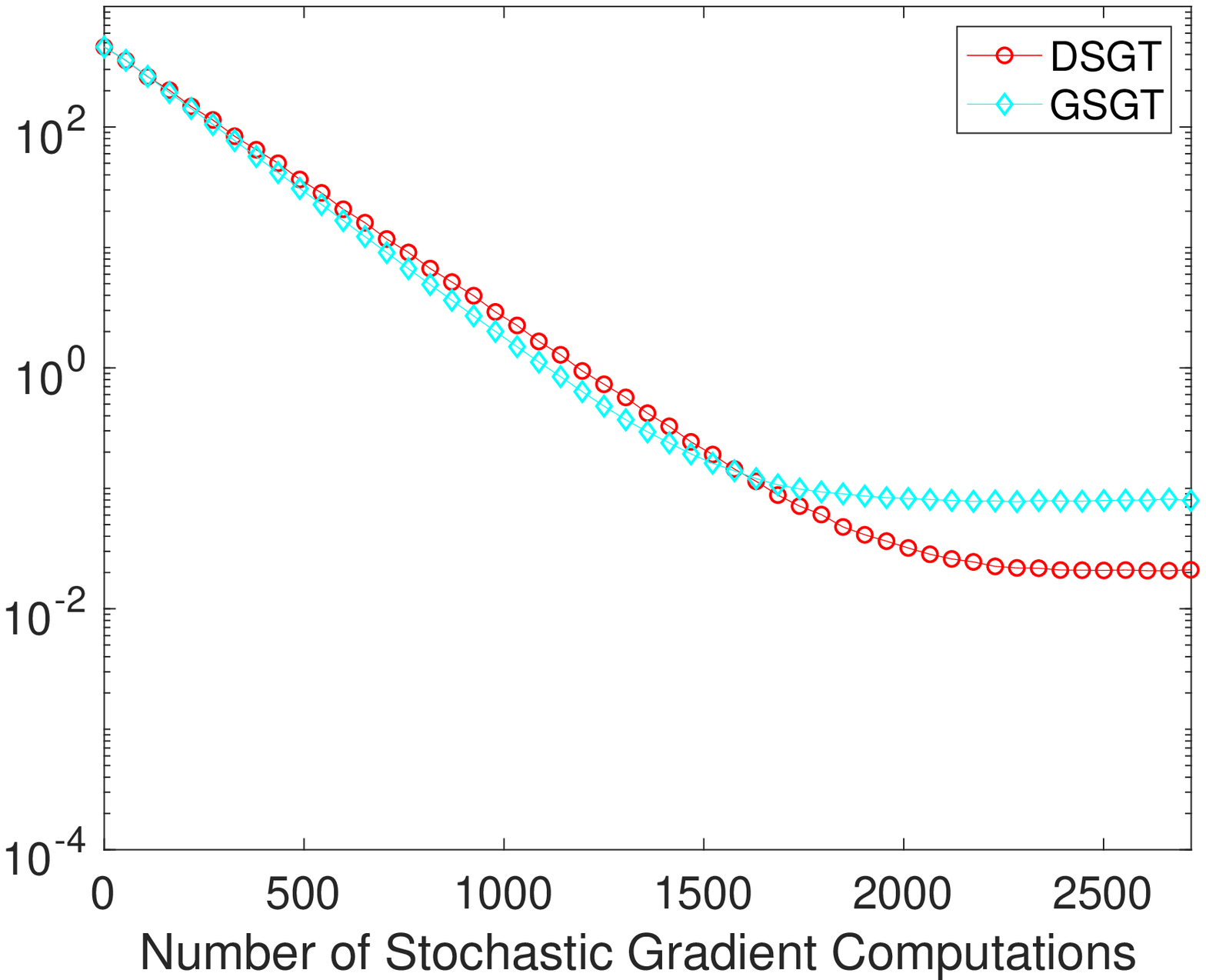}}
	\subfigure[Instance $(p,n)=(20,25)$, $\alpha=5\times10^{-3}$.]{\includegraphics[width=2.1in]{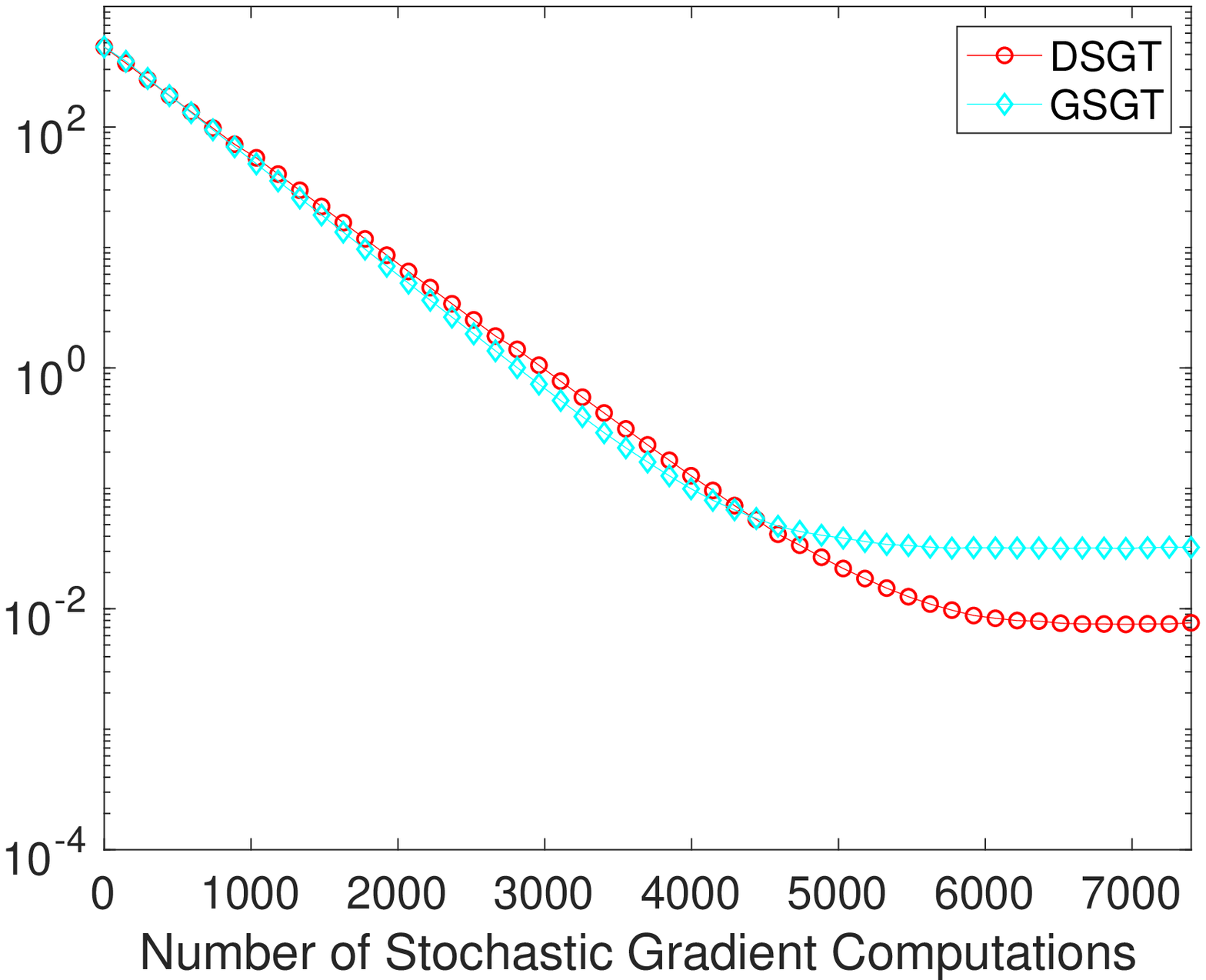}}
	\subfigure[Instance $(p,n)=(20,100)$, $\alpha=5\times10^{-3}$.]{\includegraphics[width=2.1in]{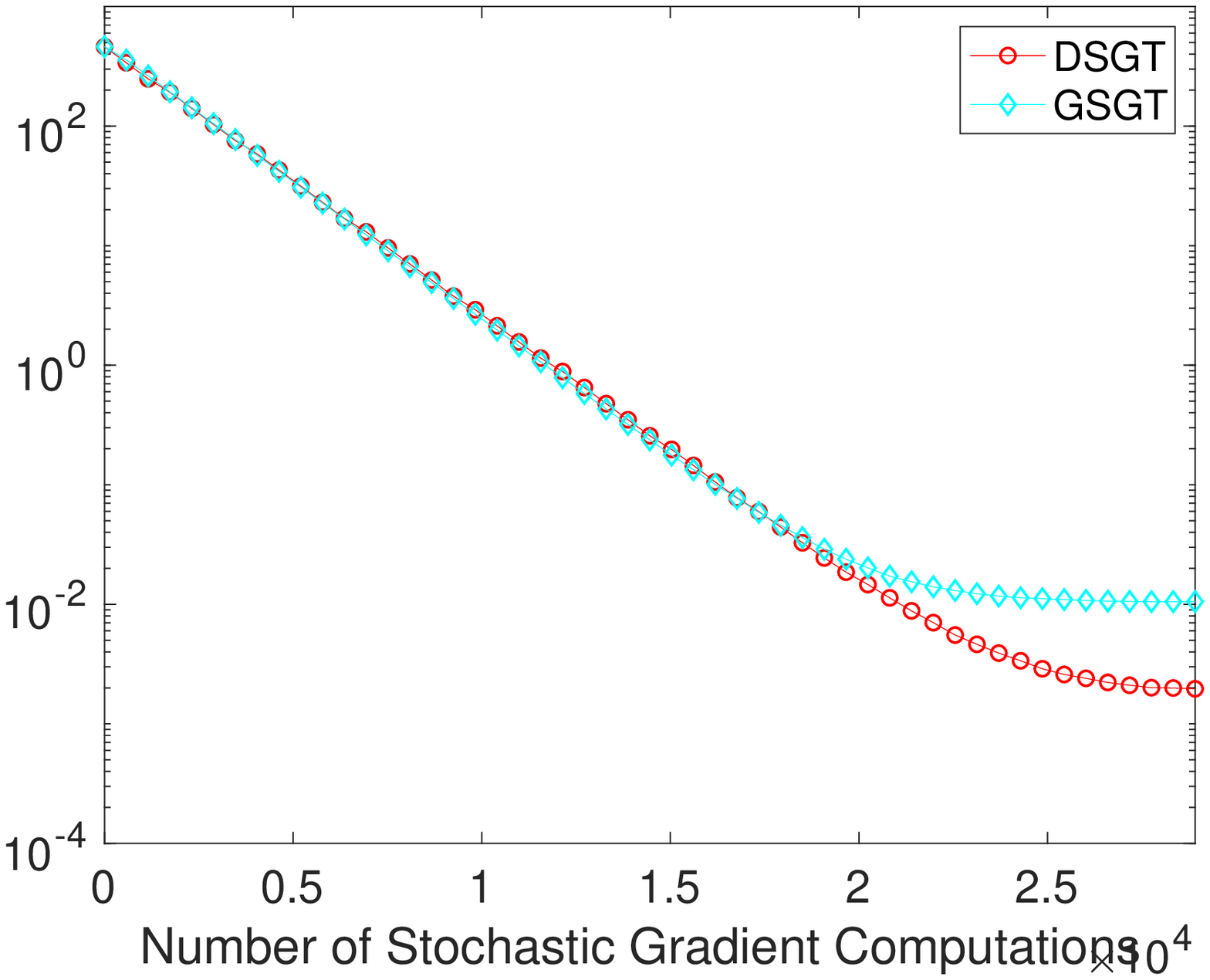}}
	
	\subfigure[Instance $(p,n)=(20,10)$, $\alpha=5\times10^{-3}$.]{\includegraphics[width=2.1in]{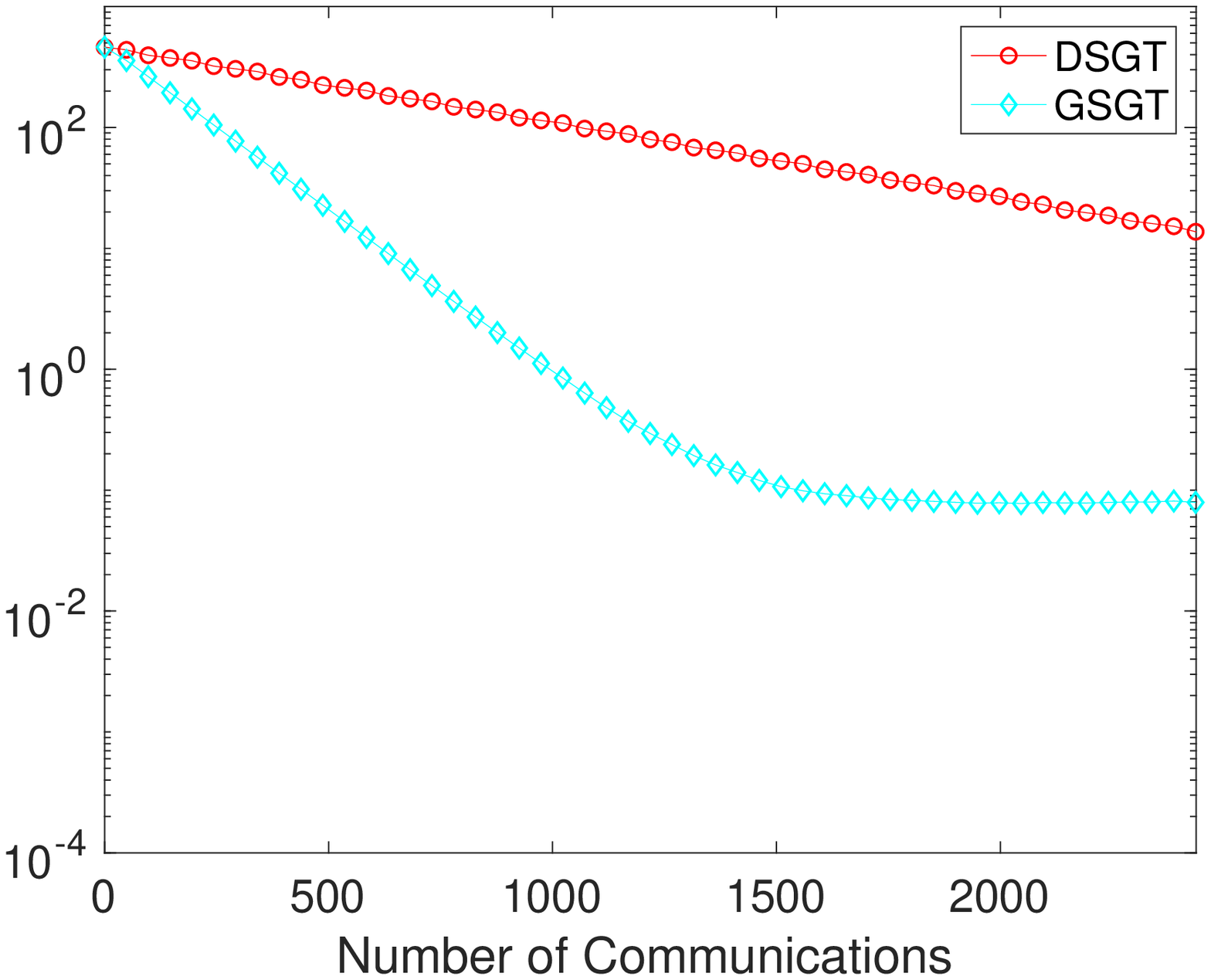}} 
	\subfigure[Instance $(p,n)=(20,25)$, $\alpha=5\times10^{-3}$.]{\includegraphics[width=2.1in]{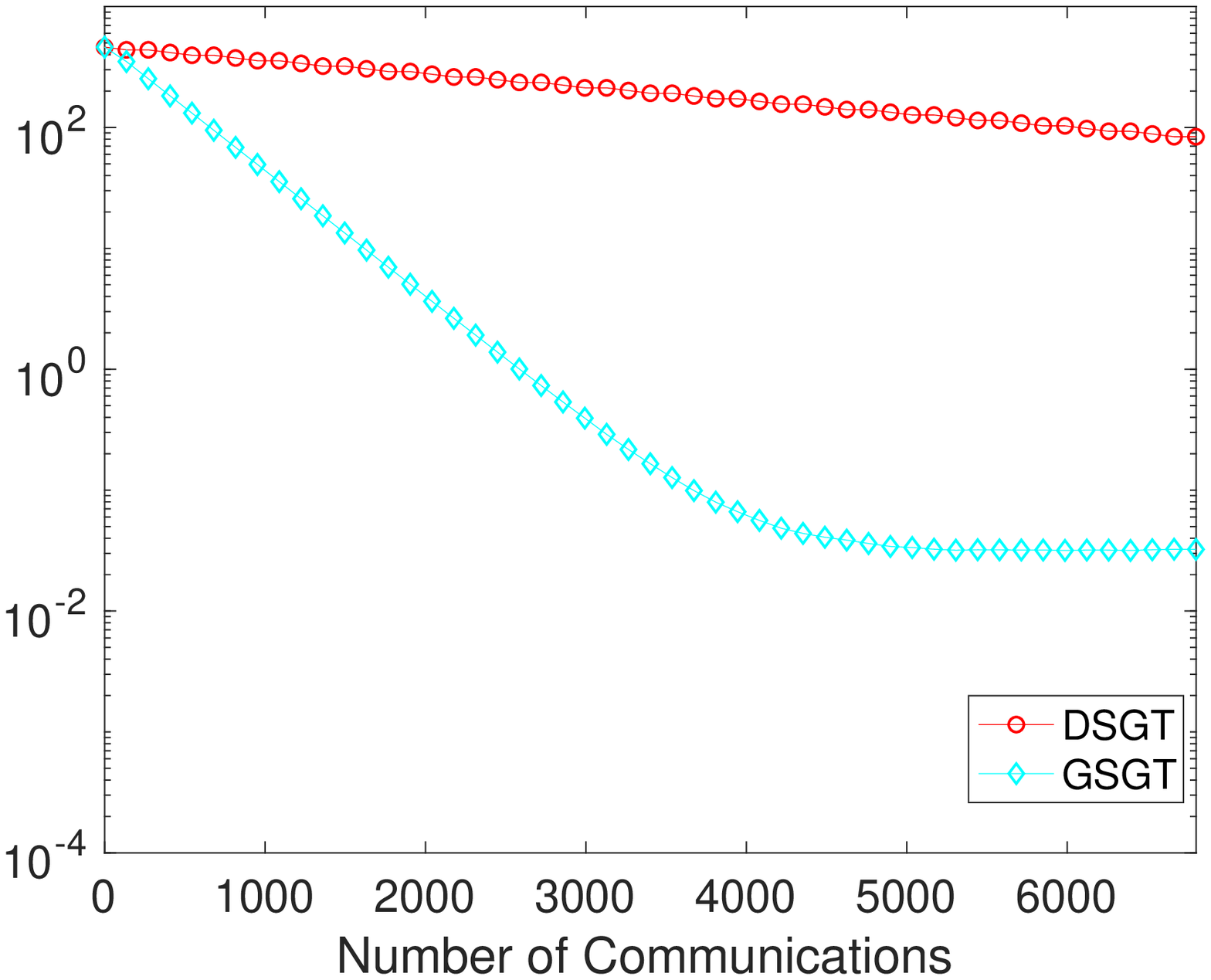}}
	\subfigure[Instance $(p,n)=(20,100)$, $\alpha=5\times10^{-3}$.]{\includegraphics[width=2.1in]{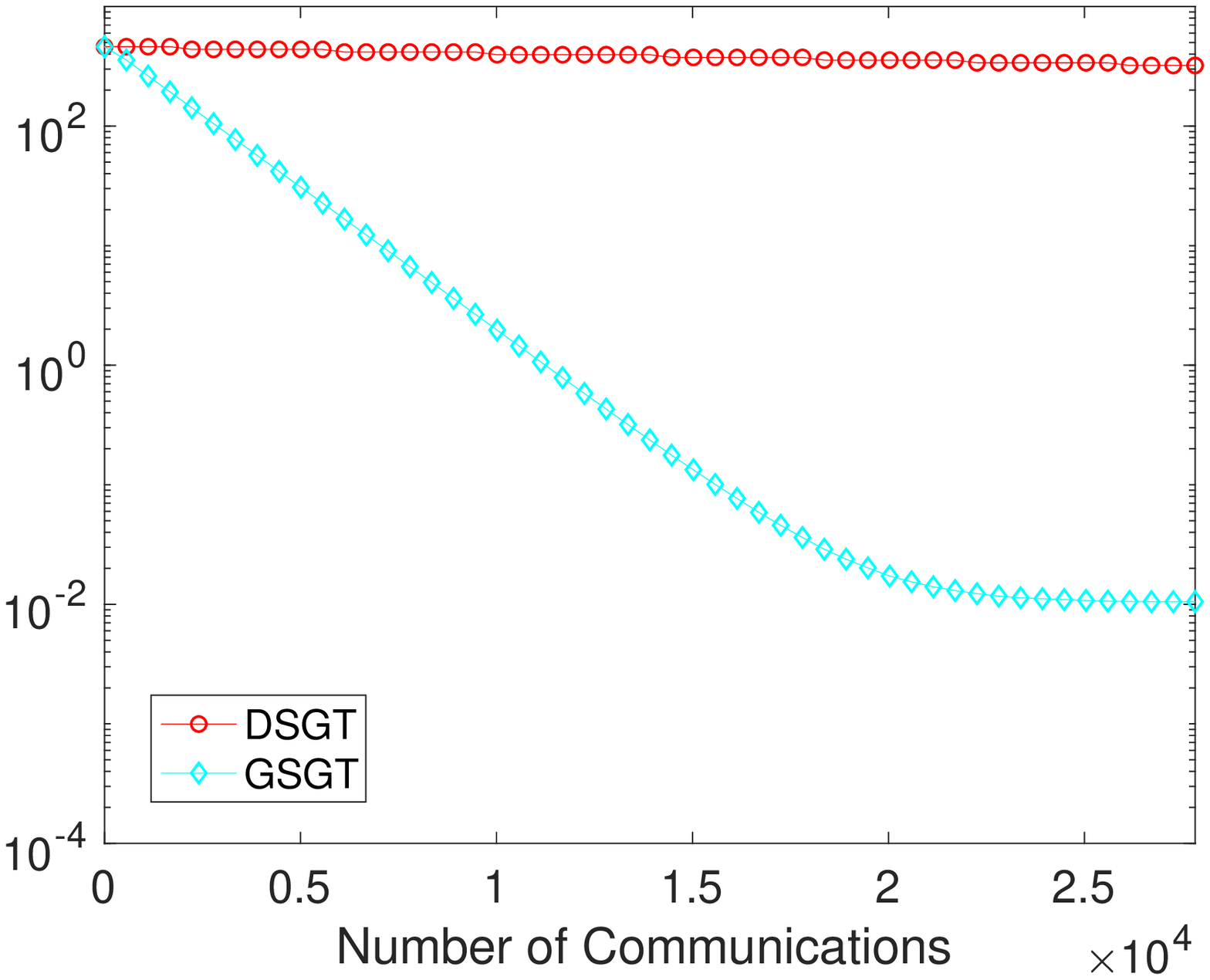}}
	\caption{Performance comparison between DSGT, GSGT, CSG, DSG, EXTRA \cite{shi2015extra} and DLM \cite{ling2015dlm} for on-line Ridge regression. 
		For CSG, the plots show $\|x_k-x^*\|^2$.
	For the other methods, the plots show $\frac{1}{n}\sum_{i=1}^n\|x_{i,k}-x^*\|^2$. 
	}
	\label{fig: comparison}
\end{figure}

\section{Conclusions and Future Work}
\label{sec: conclusion}
This paper considers distributed multi-agent optimization over a network, where each agent only has access to inexact gradients of its local cost function. 
We propose a distributed stochastic gradient tracking method (DSGT) and show that the iterates obtained by each agent, using a constant stepsize value, reach a neighborhood of the optimum (in expectation) exponentially fast. More importantly, in a limit, the error bounds for the distances between the iterates and the optimal solution decrease in the network size, which is comparable with the performance of a centralized stochastic gradient algorithm. With a diminishing stepsize, the method exhibits the optimal $\mathcal{O}(1/k)$ rate of convergence. 
In the second part of this paper, we discuss a gossip-like stochastic gradient tracking method (GSGT) that is communication-efficient. Under a well-connected interaction graph, we show GSGT requires fewer communications than DSGT to reach an $\epsilon$ error level. Finally, we provide a numerical example that demonstrates the effectiveness of both algorithms.
In our future work, we will deal with directed and/or time-varying interaction graphs among agents. We also plan to explore other more flexible randomized algorithms such as broadcast-based protocols with possible transmission failures.




\section{APPENDIX}

\subsection{Proof of Lemma \ref{lem: Main_Inequalities}}
\label{proof lem: Main_Inequalities}
By (\ref{eq: x_i,k}),
\begin{equation}
\ox_{k+1}=\ox_k-\alpha \oy_k.
\end{equation}
It follows that
\begin{equation}
\|\ox_{k+1}-x^*\|^2=\|\ox_k-\alpha \oy_k-x^*\|^2
=\|\ox_k-x^*\|^2-2\alpha\langle \ox_k-x^*,\oy_k \rangle+\alpha^2\|\oy_k\|^2.
\end{equation}
Notice that $\bE[\oy_k\mid \mathcal{F}_k]=h(\mx_k)$, and
\begin{equation*}
\bE[\|\oy_k\|^2\mid\mathcal{F}_k]=\bE[\|\oy_k-h(\mx_k)\|^2\mid\mathcal{F}_k]+\|h(\mx_k)\|^2.
\end{equation*}
We have
\begin{multline}
\bE[\|\ox_{k+1}-x^*\|^2\mid \mathcal{F}_k]
=\|\ox_k-x^*\|^2-2\alpha\langle \ox_k-x^*,h(\mx_k) \rangle+\alpha^2\bE[\|\oy_k-h(\mx_k)\|^2\mid\mathcal{F}_k]+\alpha^2\|h(\mx_k)\|^2\\
\le \|\ox_k-x^*\|^2-2\alpha\langle \ox_k-x^*,h(\mx_k) \rangle+\alpha^2\|h(\mx_k)\|^2+\frac{\alpha^2\sigma^2}{n},
\end{multline}
where the inequality follows from Lemma \ref{lem: oy_k-h_k}. Denote $\lambda=1-\alpha\mu$. In light of Lemma \ref{lem: strong_convexity},
\begin{align*}
& \bE[\|\ox_{k+1}-x^*\|^2\mid \mathcal{F}_k]\\
\le &\|\ox_k-x^*\|^2-2\alpha\langle \ox_k-x^*,\nabla f(\ox_k)\rangle+2\alpha\langle \ox_k-x^*,\nabla f(\ox_k)-h(\mx_k) \rangle+\alpha^2\|\nabla f(\ox_k)-h(\mx_k)\|^2\\
&+\alpha^2\|\nabla f(\ox_k)\|^2-2\alpha^2\langle \nabla f(\ox_k),\nabla f(\ox_k)-h(\mx_k)\rangle+\frac{\alpha^2\sigma^2}{n}\\
= &\|\ox_k-\alpha \nabla f(\ox_k)-x^*\|^2+\alpha^2\| \nabla f(\ox_k)-h(\mx_k)\|^2+\frac{\alpha^2\sigma^2}{n}+2\alpha \langle \ox_k-\alpha \nabla f(\ox_k)-x^*, \nabla f(\ox_k)-h(\mx_k)\rangle\\
\le &\lambda^2\|\ox_k-x^*\|^2+2\alpha \lambda\| \ox_k-x^*\| \| \nabla f(\ox_k)-h(\mx_k)\|+\alpha^2\| \nabla f(\ox_k)-h(\mx_k)\|^2+\frac{\alpha^2\sigma^2}{n}\\
\le &\lambda^2\|\ox_k-x^*\|^2+\frac{2\alpha \lambda L}{\sqrt{n}}\| \ox_k-x^*\|\| \mx_k-\mathbf{1}\ox_k\|+\frac{\alpha^2 L^2}{n}\|  \mx_k-\mathbf{1}\ox_k\|^2+\frac{\alpha^2\sigma^2}{n}\\
\le &\lambda^2\|\ox_k-x^*\|^2+\alpha\left(\lambda^2\mu\| \ox_k-x^*\|^2+\frac{L^2}{\mu n}\| \mx_k-\mathbf{1}\ox_k\|^2 \right)+\frac{\alpha^2 L^2}{n}\|  \mx_k-\mathbf{1}\ox_k\|^2+\frac{\alpha^2\sigma^2}{n}\\
= &\lambda^2\left(1+\alpha\mu\right)\|\ox_k-x^*\|^2
+\frac{\alpha L^2}{\mu n}\left(1+\alpha\mu\right)\|  \mx_k-\mathbf{1}\ox_k\|^2
+\frac{\alpha^2\sigma^2}{n}\\
\le &\left(1-\alpha\mu\right)\|\ox_k-x^*\|^2+\frac{\alpha L^2}{\mu n}\left(1+\alpha\mu\right)\|  \mx_k-\mathbf{1}\ox_k\|^2
+\frac{\alpha^2\sigma^2}{n}.
\end{align*}
Relation (\ref{Second_Main_Inequality}) follows from the following argument:
\begin{align*}
\|\mx_{k+1}-\mathbf{1}\ox_{k+1}\|^2= & \|W\mx_k-\alpha W\my_k-\mathbf{1}\ox_k+\alpha\mathbf{1}\oy_k\|^2\\
\le & \|W\mx_k-\mathbf{1}\ox_k\|^2-2\alpha\langle W\mx_k-\mathbf{1}\ox_k, W\my_k-\mathbf{1}\oy_k\rangle+\alpha^2\|W\my_{k}-\mathbf{1}\oy_k\|^2\\
\le & \rho_w^2\|\mx_k-\mathbf{1}\ox_k\|^2+\alpha\rho_w^2\left[\frac{(1-\rho_w^2)}{2\alpha\rho_w^2}\|\mx_k-\mathbf{1}\ox_k\|^2+\frac{2\alpha\rho_w^2}{(1-\rho_w^2)}\|\my_{k}-\mathbf{1}\oy_k\|^2\right]+\alpha^2\rho_w^2\|\my_{k}-\mathbf{1}\oy_k\|^2\\
\le & \frac{(1+\rho_w^2)}{2}\|\mx_k-\mathbf{1}\ox_k\|^2+\alpha^2\frac{(1+\rho_w^2)\rho_w^2}{(1-\rho_w^2)}\|\my_{k}-\mathbf{1}\oy_k\|^2,
\end{align*}
where we used Lemma \ref{lem: spectral norm}.

To prove (\ref{Third_Main_Inquality}), we need some preparations first. For ease of exposition we will write $G_k:=G(\mx_k,\boldsymbol{\xi}_k)$ and $\nabla_k:=\nabla F(\mx_k)$ for short. From (\ref{eq: x_k}) and Lemma \ref{lem: spectral norm},
\begin{align*}
\|\my_{k+1}-\mathbf{1}\oy_{k+1}\|^2
= & \|W\my_k+G_{k+1}-G_k-\mathbf{1}\oy_k+\mathbf{1}(\oy_k-\oy_{k+1})\|^2\\
= & |W\my_k-\mathbf{1}\oy_k\|^2+\|G_{k+1}-G_k\|^2+n\|\oy_k-\oy_{k+1}\|^2+2\langle W\my_k-\mathbf{1}\oy_k,G_{k+1}-G_k\rangle\\
&+2\langle W\my_k-\mathbf{1}\oy_k,\mathbf{1}(\oy_k-\oy_{k+1})\rangle+2\langle G_{k+1}-G_k,\mathbf{1}(\oy_k-\oy_{k+1})\rangle\\
= &\|W\my_k-\mathbf{1}\oy_k\|^2+\|G_{k+1}-G_k\|^2-n\|\oy_k-\oy_{k+1}\|^2+2\langle W\my_k-\mathbf{1}\oy_k,G_{k+1}-G_k\rangle\\
\le &\rho_w^2\|\my_k-\mathbf{1}\oy_k\|^2+\|G_{k+1}-G_k\|^2+2\langle W\my_k-\mathbf{1}\oy_k,G_{k+1}-G_k\rangle.
\end{align*}
Notice that 
\begin{align*}
\bE[\|G_{k+1}-G_k\|^2\mid \mathcal{F}_k]= & \bE[\|\nabla_{k+1}-\nabla_k \|^2\mid \mathcal{F}_k]
+2\bE[\langle \nabla_{k+1}-\nabla_k , G_{k+1}-\nabla_{k+1}-G_k+\nabla_k \rangle \mid \mathcal{F}_k]\\
&+\bE[\|G_{k+1}-\nabla_{k+1}-G_k+\nabla_k \|^2\mid \mathcal{F}_k]\\
\le & \bE[\|\nabla_{k+1}-\nabla_k \|^2\mid \mathcal{F}_k]+2\bE[\langle \nabla_{k+1}, -G_k+\nabla_k \rangle \mid \mathcal{F}_k]+2n\sigma^2
\end{align*}
by Assumption \ref{asp: gradient samples},
and
\begin{multline*}
\bE[\langle W\my_k-\mathbf{1}\oy_k,G_{k+1}-G_k\rangle\mid \mathcal{F}_k]
= \bE[\langle W\my_k-\mathbf{1}\oy_k,\nabla_{k+1}-G_k\rangle\mid \mathcal{F}_k]\\
= \bE[\langle W\my_k-\mathbf{1}\oy_k,\nabla_{k+1}-\nabla_k \rangle\mid \mathcal{F}_k]
+\bE[\langle W\my_k-\mathbf{1}\oy_k, -G_k+\nabla_k \rangle\mid \mathcal{F}_k].
\end{multline*}
We have
\begin{multline}
\label{my_k+1-oy_k+1 pre}
\bE[\|\my_{k+1}-\mathbf{1}\oy_{k+1}\|^2\mid\mathcal{F}_k]
\le \rho_w^2\bE[\|\my_k-\mathbf{1}\oy_k\|^2\mid \mathcal{F}_k]
+\bE[\|\nabla_{k+1}-\nabla_k \|^2\mid \mathcal{F}_k]+2\bE[\langle \nabla_{k+1}, -G_k+\nabla_k \rangle \mid \mathcal{F}_k]\\
+2\bE[\langle W\my_k-\mathbf{1}\oy_k,\nabla_{k+1}-\nabla_k \rangle\mid \mathcal{F}_k]
+2\bE[\langle W\my_k-\mathbf{1}\oy_k, -G_k+\nabla_k \rangle\mid \mathcal{F}_k]+2n\sigma^2.
\end{multline}
Two additional lemmas are in hand.
\begin{lemma}
	\label{lem: 3 nablas}
	\begin{equation*}
	\bE[\langle \nabla_{k+1}, -G_k+\nabla_k \rangle \mid \mathcal{F}_k]\le \alpha L n\sigma^2.
	\end{equation*}
\end{lemma}
\begin{proof}
From (\ref{eq: x_i,k}),
\begin{align*}
	\nabla f_i(x_{i,k+1})= & \nabla f_i\left(\sum_{j=1}^n w_{ij}(x_{j,k}-\alpha y_{j,k})\right)\\
	= & \nabla f_i\left(\sum_{j=1}^n w_{ij}x_{j,k}-\alpha\sum_{j=1}^n w_{ij}\left[\sum_{l=1}^{n}w_{jl}y_{l,k-1}+g_j(x_{j,k},\xi_{j,k})-g_j(x_{j,k-1},\xi_{j,k-1})\right]\right)\\
	= & \nabla f_i\left(\sum_{j=1}^n w_{ij}x_{j,k}-\alpha\sum_{j=1}^n w_{ij}\sum_{l=1}^{n}w_{jl}y_{l,k-1}-\alpha\sum_{j=1}^n w_{ij}g_j(x_{j,k},\xi_{j,k})+\alpha\sum_{j=1}^n w_{ij}g_j(x_{j,k-1},\xi_{j,k-1})\right).
	\end{align*}
	In light of Assumption \ref{asp: strconvexity},	
	\begin{multline}
	\left\|\nabla f_i(x_{i,k+1})-\nabla f_i\left(\sum_{j=1}^n w_{ij}x_{j,k}-\alpha\sum_{j=1}^n w_{ij}\sum_{l=1}^{n}w_{jl}y_{l,k-1}-\alpha\sum_{j\neq i}^n w_{ij}g_j(x_{j,k},\xi_{j,k})-\alpha w_{ii}\nabla f_i(x_{i,k})\right.\right.\\
	\left.\left.+\alpha\sum_{j=1}^n w_{ij}g_j(x_{j,k-1},\xi_{j,k-1})\right)\right\|
	\le \alpha L\|g_i(x_{i,k},\xi_{i,k})-\nabla f_i(x_{i,k})\|.
	\end{multline}
	Then,
	\begin{multline}
	\bE[\langle \nabla f_i(x_{i,k+1}), -g_i(x_{i,k},\xi_{i,k})+\nabla f_i(x_{i,k})\rangle \mid \mathcal{F}_k]
	\le \alpha L \bE[\|g_i(x_{i,k},\xi_{i,k})-\nabla f_i(x_{i,k})\|^2\mid \mathcal{F}_k]\le \alpha L \sigma^2.
	\end{multline}
	The desired result then follows.
\end{proof}
\begin{lemma}
	\label{lem: Vy_k-oy_k, 4 nablas}
	\begin{equation}
	\bE[\langle W\my_k-\mathbf{1}\oy_k, -G_k+\nabla_k \rangle\mid \mathcal{F}_k]\le \sigma^2.
	\end{equation}
\end{lemma}
\begin{proof}
	By (\ref{eq: x_i,k}), we have
	\begin{equation*}
	\bE[\langle W\my_k-\mathbf{1}\oy_k, -G_k+\nabla_k \rangle\mid \mathcal{F}_k]\\
	=\sum_{i=1}^n\bE\left[\Big\langle \sum_{j=1}^n w_{ij}y_{j,k}-\oy_k, \nabla f_i(x_{i,k})-g_i(x_{i,k},\xi_{i,k})\Big\rangle\Big\vert \mathcal{F}_k\right].
	\end{equation*}
	On one hand,
	\begin{multline*}
	\bE[\langle y_{j,k},\nabla f_i(x_{i,k})-g_i(x_{i,k},\xi_{i,k})\rangle\mid \mathcal{F}_k]\\
	=\bE\left[\Big\langle \sum_{n=1}^{n}w_{jn}y_{n,k-1}+g_j(x_{j,k},\xi_{j,k})-g_j(x_{j,k-1},\xi_{j,k-1}),\nabla f_i(x_{i,k})-g_i(x_{i,k},\xi_{i,k})\Big\rangle\Big\vert \mathcal{F}_k\right]\\
	=\bE\left[\left\langle  g_j(x_{j,k},\xi_{j,k}),\nabla f_i(x_{i,k})-g_i(x_{i,k},\xi_{i,k})\right\rangle\big\vert \mathcal{F}_k\right],
	\end{multline*}
	which gives
	\begin{equation*}
	\bE\left[\Big\langle \sum_{j=1}^n w_{ij}y_{j,k},\nabla f_i(x_{i,k})-g_i(x_{i,k},\xi_{i,k})\Big\rangle\Big\vert \mathcal{F}_k\right]
	=\bE[\langle w_{ii}g_i(x_{i,k},\xi_{i,k}),\nabla f_i(x_{i,k})-g_i(x_{i,k},\xi_{i,k})\rangle\mid \mathcal{F}_k]\le 0.
	\end{equation*}
	On the other hand,
	\begin{multline*}
	\bE[\langle \oy_k, \nabla f_i(x_{i,k})-g_i(x_{i,k},\xi_{i,k})\rangle\mid \mathcal{F}_k]
	=\bE\left[\Big\langle \frac{1}{n}\sum_{j=1}^{n}g_j(x_{j,k},\xi_{j,k}), \nabla f_i(x_{i,k})-g_i(x_{i,k},\xi_{i,k})\Big\rangle\Big\vert \mathcal{F}_k\right]\\
	=\bE\left[\Big\langle \frac{1}{n}g_i(x_{i,k},\xi_{i,k}), \nabla f_i(x_{i,k})-g_i(x_{i,k},\xi_{i,k})\Big\rangle\Big\vert \mathcal{F}_k\right].
	\end{multline*}
	We have
	\begin{equation}
	\bE[\langle W\my_k-\mathbf{1}\oy_k, -G_k+\nabla_k \rangle\mid \mathcal{F}_k]
	\le -\frac{1}{n}\sum_{i=1}^n\bE[\langle g_i(x_{i,k},\xi_{i,k}), \nabla f_i(x_{i,k})-g_i(x_{i,k},\xi_{i,k})\rangle\mid \mathcal{F}_k]
	\le \sigma^2.
	\end{equation}
\end{proof}
By (\ref{my_k+1-oy_k+1 pre}), Lemma \ref{lem: 3 nablas} and Lemma \ref{lem: Vy_k-oy_k, 4 nablas}, we obtain
\begin{multline}
\label{my_k+1-oy_k+1}
\bE[\|\my_{k+1}-\mathbf{1}\oy_{k+1}\|^2\mid \mathcal{F}_k]\le \rho_w^2\bE[\|\my_{k}-\mathbf{1}\oy_{k}\|^2\mid \mathcal{F}_k]
+\bE[\|\nabla_{k+1}-\nabla_k \|^2\mid\mathcal{F}_k]\\
+2\bE[\langle W\my_k-\mathbf{1}\oy_k, \nabla_{k+1}-\nabla_k \rangle\mid \mathcal{F}_k]+2(n+\alpha Ln+1)\sigma^2.
\end{multline}

Now we bound $\|\nabla_{k+1}-\nabla_k \|^2$ and $\langle W\my_k-\mathbf{1}\oy_k, \nabla_{k+1}-\nabla_k \rangle$.
First, by Assumption \ref{asp: strconvexity} and Lemma \ref{lem: spectral norm},
\begin{multline*}
\|\nabla_{k+1}-\nabla_k \|^2\le L^2\|\mx_{k+1}-\mx_k\|^2=L^2\|W\mx_k-\mx_k-\alpha W\my_k\|^2=L^2\|(\mW-\mI)(\mx_k-\mathbf{1}\ox_k)-\alpha W\my_k\|^2\\
= \|\mW-\mI\|^2L^2\|\mx_k-\mathbf{1}\ox_k\|^2-2\alpha L^2\langle (\mW-\mI)(\mx_k-\mathbf{1}\ox_k), W\my_k\rangle+\alpha^2L^2\|W\my_k\|^2\\
=\|\mW-\mI\|^2L^2\|\mx_k-\mathbf{1}\ox_k\|^2-2\alpha L^2\langle (\mW-\mI)(\mx_k-\mathbf{1}\ox_k), W\my_k-\mathbf{1}\oy_k\rangle+\alpha^2L^2\|W\my_k-\mathbf{1}\oy_k\|^2+\alpha^2 nL^2\|\oy_k\|^2\\
\le \|\mW-\mI\|^2L^2\|\mx_k-\mathbf{1}\ox_k\|^2+2\alpha \|\mW-\mI\| L^2\rho_w\|\mx_k-\mathbf{1}\ox_k\|\|\my_k-\mathbf{1}\oy_k\|
+\alpha^2 L^2\rho_w^2\|\my_k-\mathbf{1}\oy_k\|^2+\alpha^2 nL^2\|\oy_k\|^2\\
\le 2\|\mW-\mI\|^2L^2\|\mx_k-\mathbf{1}\ox_k\|^2+2\alpha^2 L^2\rho_w^2\|\my_k-\mathbf{1}\oy_k\|^2+\alpha^2 nL^2\|\oy_k\|^2.
\end{multline*}
Second,
\begin{multline*}
\langle W\my_k-\mathbf{1}\oy_k, \nabla_{k+1}-\nabla_k \rangle\le L\rho_w\|\my_k-\mathbf{1}\oy_k\| \|(\mW-\mI)(\mx_k-\mathbf{1}\ox_k)-\alpha W\my_k\|\\
\le \|\mW-\mI\|L\rho_w\|\my_k-\mathbf{1}\oy_k\|\|\mx_k-\mathbf{1}\ox_k\|+\alpha L\rho_w\|\my_k-\mathbf{1}\oy_k\|\|W\my_k-\mathbf{1}\oy_k+\mathbf{1}\oy_k\|\\
\le \|\mW-\mI\|L\rho_w\|\my_k-\mathbf{1}\oy_k\|\|\mx_k-\mathbf{1}\ox_k\|+\alpha L\rho_w^2\|\my_k-\mathbf{1}\oy_k\|^2+\alpha \sqrt{n}L\rho_w\|\my_k-\mathbf{1}\oy_k\|\|\oy_k\|.
\end{multline*}
Notice that
\begin{equation*}
\|\oy_k\|\le \|\oy_k-h(\mx_k)\|+\|h(\mx_k)-\nabla f(\ox_k)\|+\|\nabla f(\ox_k)\|\le \|\oy_k-h(\mx_k)\|+\frac{L}{\sqrt{n}}\|\mx_k-\mathbf{1}\ox_k\|+L\|\ox_k-x^*\|.
\end{equation*}
We have 
\begin{multline*}
\sqrt{n}L\rho_w\|\my_k-\mathbf{1}\oy_k\|\|\oy_k\|\le \sqrt{n}L\rho_w\|\my_k-\mathbf{1}\oy_k\|\left(\|\oy_k-h(\mx_k)\|+\frac{L}{\sqrt{n}}\|\mx_k-\mathbf{1}\ox_k\|+L\|\ox_k-x^*\|\right)\\
\le L\rho_w^2\|\my_k-\mathbf{1}\oy_k\|^2+nL\|\oy_k-h(\mx_k)\|^2+ L^3\|\mx_k-\mathbf{1}\ox_k\|^2+\frac{1}{2}n L^3\|\ox_k-x^*\|^2,
\end{multline*}
and
\begin{equation*}
\|\oy_k\|^2\le 3\|\oy_k-h(\mx_k)\|^2+\frac{3L^2}{n}\|\mx_k-\mathbf{1}\ox_k\|^2+3L^2\|\ox_k-x^*\|^2.
\end{equation*}
By (\ref{my_k+1-oy_k+1}) and the above relations,
\begin{multline}
\bE[\|\my_{k+1}-\mathbf{1}\oy_{k+1}\|^2\mid \mathcal{F}_k]
\le \rho_w^2\bE[\|\my_{k}-\mathbf{1}\oy_{k}\|^2\mid \mathcal{F}_k]+2\|\mW-\mI\|^2L^2\|\mx_k-\mathbf{1}\ox_k\|^2
+2\alpha^2 L^2\rho_w^2\bE[\|\my_k-\mathbf{1}\oy_k\|^2\mid\mathcal{F}_k]\\+\alpha^2 nL^2\left(3\bE[\|\oy_k-h(\mx_k)\|^2\mid\mathcal{F}_k]+\frac{3L^2}{n}\|\mx_k-\mathbf{1}\ox_k\|^2+3L^2\|\ox_k-x^*\|^2\right)\\
+2\left(\|\mW-\mI\|L\rho_w\|\my_k-\mathbf{1}\oy_k\|\|\mx_k-\mathbf{1}\ox_k\|+\alpha L\rho_w^2\|\my_k-\mathbf{1}\oy_k\|^2\right)\\
+2\left(\alpha L\rho_w^2\bE[\|\my_k-\mathbf{1}\oy_k\|^2\mid\mathcal{F}_k]+\alpha nL\bE[\|\oy_k-h(\mx_k)\|^2\mid\mathcal{F}_k]+\alpha L^3\|\mx_k-\mathbf{1}\ox_k\|^2+\frac{1}{2}\alpha n L^3\|\ox_k-x^*\|^2\right)\\
+2(n+\alpha Ln+1)\sigma^2\\
\le \left(\rho_w^2+4\alpha L\rho_w^2+2\alpha^2 L^2\rho_w^2\right)\bE[\|\my_{k}-\mathbf{1}\oy_{k}\|^2\mid \mathcal{F}_k]
+\left(\beta\rho_w^2\bE[\|\my_k-\mathbf{1}\oy_k\|^2\mid\mathcal{F}_k]+\frac{1}{\beta}\|\mW-\mI\|^2 L^2\|\mx_k-\mathbf{1}\ox_k\|^2\right)\\
+\left(2\|\mW-\mI\|^2L^2+2\alpha L^3+3\alpha^2 L^4\right)\|\mx_k-\mathbf{1}\ox_k\|^2+\left(3\alpha^2 nL^4+\alpha nL^3\right)\|\ox_k-x^*\|^2\\
+\left[3\alpha^2L^2+2\alpha L+2(n+\alpha Ln+1)\right]\sigma^2\\
= \left(1+4\alpha L+2\alpha^2 L^2+\beta\right)\rho_w^2\bE[\|\my_{k}-\mathbf{1}\oy_{k}\|^2\mid \mathcal{F}_k]
+\left(\frac{1}{\beta}\|\mW-\mI\|^2 L^2+2\|\mW-\mI\|^2 L^2+3\alpha L^3\right)\|\mx_k-\mathbf{1}\ox_k\|^2\\
+2\alpha nL^3\|\ox_k-x^*\|^2+M_{\sigma}
\end{multline}
for any $\beta>0$.

\subsection{Proof of Lemma \ref{lem: rho_M}}
\label{subsec: proof lemma rho_M}
The characteristic function of $\mS$ is given by
\begin{multline}
g(\lambda):=\text{det}(\lambda \mI-\mS)=(\lambda-s_{11})(\lambda-s_{22})(\lambda-s_{33})
-a_{23}a_{32}(\lambda-s_{11})-a_{13}a_{31}(\lambda-s_{22})\\
-a_{12}a_{21}(\lambda-s_{33})-a_{12}a_{23}a_{31}-a_{13}a_{32}a_{21}.
\end{multline}
Necessity is trivial since $\text{det}(\lambda^* \mI-\mS)\le 0$ implies $g(\lambda)=0$ for some $\lambda\ge \lambda^*$. We now show $\text{det}(\lambda^* \mI-\mS)>0$ is also a sufficient condition.
Given that $g(\lambda^*)=\text{det}(\lambda^* \mI-\mS)>0$,
\begin{equation*}
(\lambda^*-s_{11})(\lambda^*-s_{22})(\lambda^*-s_{33})> a_{23}a_{32}(\lambda^*-s_{11})+a_{13}a_{31}(\lambda^*-s_{22})+a_{12}a_{21}(\lambda^*-s_{33}).
\end{equation*}
It follows that
\begin{equation}
\begin{array}{ccc}
\gamma_1(\lambda^*-s_{22})(\lambda^*-s_{33}) & > & a_{23}a_{32}\\
\gamma_2(\lambda^*-s_{11})(\lambda^*-s_{33}) & > & a_{13}a_{31}\\
\gamma_3(\lambda^*-s_{11})(\lambda^*-s_{22}) & > & a_{12}a_{21}
\end{array}
\end{equation}
for some $\gamma_1,\gamma_2,\gamma_3>0$ with $\gamma_1+\gamma_2+\gamma_3\le 1$.
Consider
\begin{equation*}
g'(\lambda)=(\lambda-s_{22})(\lambda-s_{33})+(\lambda-s_{11})(\lambda-s_{33})+(\lambda-s_{11})(\lambda-s_{22})-a_{23}a_{32}-a_{13}a_{31}-a_{12}a_{21}.
\end{equation*}
We have $g'(\lambda)>0$ for $\lambda\in(-\infty,-\lambda^*]\cup[\lambda^*,+\infty)$. Notice that
\begin{equation*}
g(-\lambda^*)\le -(\lambda^*+s_{11})(\lambda^*+s_{22})(\lambda^*+s_{33})+a_{23}a_{32}(1+s_{11})+a_{13}a_{31}(\lambda^*+s_{22})+a_{12}a_{21}(\lambda^*+s_{33})< 0.
\end{equation*}
All the real roots of $g(\lambda)=0$ lie in the interval $(-\lambda^*,\lambda^*)$. By the Perron-Frobenius theorem, $\rho(\mS)\in\mathbb{R}$ is an eigenvalue of $\mS$. We conclude that $\rho(\mS)<\lambda^*$.

\subsection{Proof of Lemma \ref{lem: gossip three main inequalities}}
\label{appendix: lem gossip three main inequalities}
First we prove relation (\ref{gossip: first main inequality}).
In light of (\ref{eq:x-update gossip compact}), we have
\begin{equation}
\label{eq: x_k+1 gossip}
\bar{x}_{k+1}=\frac{1}{n}\left[(x_{i_k,k}+x_{j_k,k})-\alpha (y_{i_k,k}+y_{j_k,k})\right]+\frac{1}{n}\sum_{i\neq\{i_k,j_k\}}x_{i,k}
=\ox_k-\frac{\alpha}{n}(y_{i_k,k}+y_{j_k,k}).
\end{equation}
Then,
\begin{equation*}
\|\ox_{k+1}-x^*\|^2=\left\|\ox_k-\frac{\alpha}{n}(y_{i_k,k}+y_{j_k,k})-x^*\right\|^2=\|\ox_k-x^*\|^2-\frac{2\alpha}{n}\langle \ox_k-x^*, y_{i_k,k}+y_{j_k,k}\rangle+\frac{\alpha^2}{n^2}\|y_{i_k,k}+y_{j_k,k}\|^2.
\end{equation*}
Taking conditional expectations of $(y_{i_k,k}+y_{j_k,k})$ and $\|y_{i_k,k}+y_{j_k,k}\|^2$ w.r.t. the random selections of $i_k$ and $j_k$, we get
\begin{equation}
\bE[y_{i_k,k}+y_{j_k,k}\mid \mathcal{F}_{k+1}]=\frac{1}{n}\sum_{i=1}^n\left( y_{i,k}+\sum_{j=1}^n\pi_{ij}y_{j,k}\right)=2\oy_k,
\end{equation}
and
\begin{multline}
\bE[\|y_{i_k,k}+y_{j_k,k}\|^2\mid\mathcal{F}_{k+1}]=\frac{1}{n}\sum_{i=1}^n\sum_{j=1}^n\pi_{ij}\|y_{i,k}+y_{j,k}\|^2=\frac{2}{n}\sum_{i=1}^n\|y_{i,k}\|^2+\frac{2}{n}\sum_{i=1}^n\sum_{j=1}^n\pi_{ij}\langle y_{i,k},y_{j,k}\rangle\\
=\frac{2}{n}\|\my_k\|^2+\frac{2}{n}\my_k^{\T}\mPi\my_k=\frac{2}{n}\|\my_k\|^2+\frac{2}{n}(\my_k-\mathbf{1}\oy_k)^{\T}\mPi(\my_k-\mathbf{1}\oy_k)+2\|\oy_k\|^2
\le \frac{2}{n}(1+\rho_{\pi})\|\my_k-\mathbf{1}\oy_k\|^2+4\|\oy_k\|^2\\
	\le \frac{4}{n}\|\my_k-\mathbf{1}\oy_k\|^2+4\|\oy_k\|^2,
\end{multline}
where $\rho_{\pi}<1$ denotes the spectral norm of $\mPi-\frac{1}{n}\mathbf{1}\mathbf{1}^{\T}$.
It follows that
\begin{equation}
\label{gossip: first main inequality_pre}
\bE[\|\ox_{k+1}-x^*\|^2\mid \mathcal{F}_{k+1}]\le \|\ox_k-x^*\|^2-\frac{4\alpha}{n}\langle \ox_k-x^*, \oy_k\rangle+\frac{4\alpha^2}{n^2}\|\oy_k\|^2
+\frac{4\alpha^2}{n^3}\|\my_k-\mathbf{1}\oy_k\|^2.
\end{equation}
Noticing that	$\bE[\oy_k]=\bE[h(\mx_k)]$, from Lemma \ref{lem: oy_k-h_k} we have
\begin{equation}
\label{E[oy_k^2]}
\bE[\|\oy_k\|^2]=\bE[\|\oy_k-h(\mx_k)\|^2+2\langle \oy_k-h(\mx_k),h(\mx_k)\rangle+\|h(\mx_k)\|^2\mid\mathcal{F}_k]\le \frac{\sigma^2}{n}+\bE[\|h(\mx_k)\|^2].
\end{equation}
Then from (\ref{gossip: first main inequality_pre}) we obtain
\begin{align*}
& \bE[\|\ox_{k+1}-x^*\|^2]\\
\le & \bE[\|\ox_k-x^*\|^2]-\frac{4\alpha}{n}\bE[\langle \ox_k-x^*, h(\mx_k)\rangle]+\frac{4\alpha^2}{n^2}\bE[\|h(\mx_k)\|^2]+\frac{4\alpha^2\sigma^2}{n^3}+\frac{4\alpha^2}{n^3}\bE[\|\my_k-\mathbf{1}\oy_k\|^2]\\
= &\bE[\|\ox_k-x^*\|^2]-\frac{4\alpha}{n}\bE[\langle \ox_k-x^*, \nabla f(\ox_k)\rangle]-\frac{4\alpha}{n}\bE[\langle \ox_k-x^*, h(\mx_k)-\nabla f(\ox_k)\rangle]\\
& +\frac{4\alpha^2}{n^2}\bE[\|h(\mx_k)-\nabla f(\ox_k)\|^2]+\frac{4\alpha^2}{n^2}\bE[\|\nabla f(\ox_k)\|^2]+\frac{8\alpha^2}{n^2}\bE[\langle h(\mx_k)-\nabla f(\ox_k), \nabla f(\ox_k)\rangle]\\
& +\frac{4\alpha^2}{n^3}\bE[\|\my_k-\mathbf{1}\oy_k\|^2]+\frac{4\alpha^2\sigma^2}{n^3}\\
= &\bE\left[\left\|\ox_k-\frac{2\alpha}{n}\nabla f(\ox_k)-x^*\right\|^2\right]-\frac{4\alpha}{n}\bE[\langle \ox_k-\frac{2\alpha}{n}\nabla f(\ox_k)-x^*,h(\mx_k)-\nabla f(\ox_k)\rangle]\\
& +\frac{4\alpha^2}{n^2}\bE[\|h(\mx_k)-\nabla f(\ox_k)\|^2]+\frac{4\alpha^2}{n^3}\bE[\|\my_k-\mathbf{1}\oy_k\|^2\mid\mathcal{F}_{k}^{-}]+\frac{4\alpha^2\sigma^2}{n^3}.
\end{align*}
Since $\alpha<n/(\mu+L)$, we know from Lemma \ref{lem: strong_convexity} that
\begin{equation*}
\left\|\ox_k-\frac{2\alpha}{n}\nabla f(\ox_k)-x^*\right\|^2\le \left(1-\frac{2\alpha\mu}{n}\right)^2\|\ox_k-x^*\|^2,
\end{equation*}
and
\begin{equation*}
\|h(\mx_k)-\nabla f(\ox_k)\|\le \frac{L}{\sqrt{n}}\|\mx_k-\mathbf{1}\ox_k\|.
\end{equation*}
It follows that
\begin{align*}
& \bE[\|\ox_{k+1}-x^*\|^2]\\
\le & \left(1-\frac{2\alpha\mu}{n}\right)^2\bE[\|\ox_k-x^*\|^2]+\frac{4\alpha }{n}\left(1-\frac{2\alpha\mu}{n}\right)\frac{L}{\sqrt{n}}\bE[\|\ox_k-x^*\|\|\mx_k-\mathbf{1}\ox_k\|]+\frac{4\alpha^2L^2}{n^3}\bE[\|\mx_k-\mathbf{1}\ox_k\|^2]\\
&+\frac{4\alpha^2}{n^3}\bE[\|\my_k-\mathbf{1}\oy_k\|^2]+\frac{4\alpha^2\sigma^2}{n^3}\\
\le &
\left(1-\frac{2\alpha\mu}{n}\right)^2\bE[\|\ox_k-x^*\|^2]+\frac{2\alpha}{n}\left( \left(1-\frac{2\alpha\mu}{n}\right)^2\mu\bE[\|\ox_k-x^*\|^2]+\frac{L^2}{\mu n}\bE[\|\mx_k-\mathbf{1}\ox_k\|^2]\right)\\
& +\frac{4\alpha^2L^2}{n^3}\bE[\|\mx_k-\mathbf{1}\ox_k\|^2]+\frac{4\alpha^2}{n^3}\bE[\|\my_k-\mathbf{1}\oy_k\|^2]+\frac{4\alpha^2\sigma^2}{n^3}\\
\le &
\left(1-\frac{2\alpha\mu}{n}\right)\bE[\|\ox_k-x^*\|^2]+\frac{2\alpha L^2}{\mu n^2}\left(1+\frac{2\alpha\mu}{n}\right)\bE[\|\mx_k-\mathbf{1}\ox_k\|^2]
+\frac{4\alpha^2}{n^3}\bE[\|\my_k-\mathbf{1}\oy_k\|^2]+\frac{4\alpha^2\sigma^2}{n^3}.
\end{align*}

To bound the consensus error $\bE[\|\mx_{k+1}-\mathbf{1}\ox_{k+1}\|^2]$, note that from (\ref{eq:x-update gossip compact}) and (\ref{eq: x_k+1 gossip}) we have
\begin{multline}
\label{consensus_error_pre_gossip}
\|\mx_{k+1}-\mathbf{1}\ox_{k+1}\|^2=\left\|\mathbf{W}_k \mx_k-\alpha \mathbf{D}_k \my_k-\mathbf{1}\ox_k+\frac{\alpha}{n}\mathbf{1}(y_{i_k,k}+y_{j_k,k})\right\|^2\\
=\|\mathbf{W}_k \mx_k-\mathbf{1}\ox_k\|^2-2\alpha\langle \mathbf{W}_k \mx_k-\mathbf{1}\ox_k, \mathbf{D}_k \my_k-\frac{1}{n}\mathbf{1}(y_{i_k,k}+y_{j_k,k})\rangle+\alpha^2\left\|\mathbf{D}_k \my_k-\frac{1}{n}\mathbf{1}(y_{i_k,k}+y_{j_k,k})\right\|^2\\
=\trace\left[(\mx_k-\mathbf{1}\ox_k)^{\T} \mathbf{W}_k^{\T} \mathbf{W}_k(\mx_k-\mathbf{1}\ox_k)\right]-2\alpha\langle \mathbf{W}_k \mx_k-\mathbf{1}\ox_k, \mathbf{D}_k \my_k\rangle+\alpha^2\left\|\mathbf{D}_k \my_k-\frac{1}{n}\mathbf{1}(y_{i_k,k}+y_{j_k,k})\right\|^2.
\end{multline}
By Lemma \ref{lem: spectral norm_gossip}, the conditional expectation of $\trace\left[(\mx_k-\mathbf{1}\ox_k)^{\T}\mathbf{W}_k^{\T} \mathbf{W}_k(\mx_k-\mathbf{1}\ox_k)\right]$ can be bounded below,
\begin{equation}
\label{consensus_error_pre_gossip first inequality}
\bE[\trace[(\mx_k-\mathbf{1}\ox_k)^{\T} \mathbf{W}_k^{\T} \mathbf{W}_k(\mx_k-\mathbf{1}\ox_k)]\mid\mathcal{F}_{k+1}]=\trace\left[(\mx_k-\mathbf{1}\ox_k)^{\T}\bar{\mW}(\mx_k-\mathbf{1}\ox_k)\right]\le \rho_{\bar{w}}\|\mx_k-\mathbf{1}\ox_k\|^2.
\end{equation}
In view of the structure of $\mathbf{D}_k$, we rewrite $2\langle \mathbf{W}_k \mx_k-\mathbf{1}\ox_k, \mathbf{D}_k \my_k\rangle$ as follows:
\begin{equation*}
2\langle \mathbf{W}_k \mx_k-\mathbf{1}\ox_k, \mathbf{D}_k \my_k\rangle=2\langle \frac{1}{2}(x_{i_k,k}+x_{j_k,k})-\ox_k,y_{i_k,k}+y_{j_k,k}\rangle=\langle x_{i_k,k}+x_{j_k,k}-2\ox_k,y_{i_k,k}+y_{j_k,k}\rangle.
\end{equation*}
Note that
\begin{multline*}
\bE[\langle x_{i_k,k}-\ox_k, y_{i_k,k}\rangle\mid\mathcal{F}_{k+1}]=\frac{1}{n}\sum_{i=1}^n\langle x_{i,k}-\ox_k, y_{i,k}\rangle=\frac{1}{n}\sum_{i=1}^n\langle x_{i,k}-\ox_k, y_{i,k}-\oy_k\rangle
=\frac{1}{n}\langle \mx_k-\mathbf{1}\ox_k, \my_k-\mathbf{1}\oy_k\rangle,
\end{multline*}
\begin{multline*}
\bE[\langle x_{j_k,k}-\ox_k, y_{j_k,k}\rangle\mid\mathcal{F}_{k+1}]=\frac{1}{n}\sum_{i=1}^n\sum_{j=1}^n\pi_{ij}\langle x_{j,k}-\ox_k, y_{j,k}\rangle=\frac{1}{n}\sum_{j=1}^n\langle x_{j,k}-\ox_k, y_{j,k}\rangle
=\frac{1}{n}\langle \mx_k-\mathbf{1}\ox_k, \my_k-\mathbf{1}\oy_k\rangle,
\end{multline*}
\begin{multline*}
\bE[\langle x_{i_k,k}-\ox_k, y_{j_k,k}\rangle\mid\mathcal{F}_{k+1}]=\frac{1}{n}\sum_{i=1}^n\langle x_{i,k}-\ox_k,\sum_{j=1}^n \pi_{ij}y_{j,k}\rangle=\frac{1}{n}\sum_{i=1}^n\sum_{j=1}^n\langle x_{i,k}-\ox_k,\pi_{ij}(y_{j,k}-\oy_k)\rangle\\
=\frac{1}{n}\trace\left[(\mx_k-\mathbf{1}\ox_k)^{\T}\mPi (\my_k-\mathbf{1}\oy_k)\right],
\end{multline*}
and similarly,
\begin{equation*}
\bE[\langle x_{j_k,k}-\ox_k, y_{i_k,k}\rangle\mid\mathcal{F}_{k+1}]=\frac{1}{n}\trace\left[(\mx_k-\mathbf{1}\ox_k)^{\T}\mPi (\my_k-\mathbf{1}\oy_k)\right].
\end{equation*}
The following inequality holds:
\begin{multline}
\label{consensus_error_pre_gossip second inequality}
2\bE[\langle \mathbf{W}_k \mx_k-\mathbf{1}\ox_k, \mathbf{D}_k \my_k\rangle\mid\mathcal{F}_{k+1}]=\frac{2}{n}\langle\mx_k-\mathbf{1}\ox_k, \my_k-\mathbf{1}\oy_k\rangle+\frac{2}{n}\trace\left[(\mx_k-\mathbf{1}\ox_k)^{\T}\mPi (\my_k-\mathbf{1}\oy_k)\right]\\
\le \frac{4}{n}\|\mx_k-\mathbf{1}\ox_k\|\|\my_k-\mathbf{1}\oy_k\|.
\end{multline}
The last term in equation (\ref{consensus_error_pre_gossip}) can be bounded in the following way:
\begin{multline}
\label{consensus_error_pre_gossip third inequality}
\left\|\mathbf{D}_k \my_k-\frac{1}{n}\mathbf{1}(y_{i_k,k}+y_{j_k,k})\right\|^2=\|\mathbf{D}_k \my_k\|^2-\frac{2}{n}\langle \mathbf{D}_k \my_k, \mathbf{1}(y_{i_k,k}+y_{j_k,k})\rangle+\frac{1}{n}\|y_{i_k,k}+y_{j_k,k}\|^2\\
\le 2(\|y_{i_k,k}\|^2+\|y_{j_k,k}\|^2)-\frac{2}{n}\|y_{i_k,k}+y_{j_k,k}\|^2+\frac{1}{n}\|y_{i_k,k}+y_{j_k,k}\|^2\le 2(\|y_{i_k,k}\|^2+\|y_{j_k,k}\|^2).
\end{multline}
In view of inequalities (\ref{consensus_error_pre_gossip first inequality})-(\ref{consensus_error_pre_gossip third inequality}), from (\ref{consensus_error_pre_gossip}) we obtain
\begin{multline*}
\bE[\|\mx_{k+1}-\mathbf{1}\ox_{k+1}\|^2\mid \mathcal{F}_{k+1}]
\le \rho_{\bar{w}}\|\mx_k-\mathbf{1}\ox_k\|^2+\frac{4\alpha}{n}\|\mx_k-\mathbf{1}\ox_k\|\|\my_k-\mathbf{1}\oy_k\|+\frac{4\alpha^2}{n}\|\my_k\|^2\\
\le \rho_{\bar{w}}\|\mx_k-\mathbf{1}\ox_k\|^2+\frac{2\alpha}{n}\left(\beta_1\|\mx_k-\mathbf{1}\ox_k\|^2+\frac{1}{\beta_1}\|\my_k-\mathbf{1}\oy_k\|^2\right)+\frac{4\alpha^2}{n}\|\my_k\|^2\\
=\left(\rho_{\bar{w}}+\frac{2\alpha}{n}\beta_1\right)\|\mx_k-\mathbf{1}\ox_k\|^2+\frac{2\alpha}{n}\frac{1}{\beta_1}\|\my_k-\mathbf{1}\oy_k\|^2+\frac{4\alpha^2}{n}\left(\|\my_k-\mathbf{1}\oy_k\|^2+n\|\oy_k\|^2\right)\\
=\left(\rho_{\bar{w}}+\frac{2\alpha}{n}\beta_1\right)\|\mx_k-\mathbf{1}\ox_k\|^2+\frac{2\alpha}{n}\left(\frac{1}{\beta_1}+2\alpha\right)\|\my_k-\mathbf{1}\oy_k\|^2+4\alpha^2\|\oy_k\|^2,
\end{multline*}
where $\beta_1>0$ is arbitrary.
Notice that by relation (\ref{E[oy_k^2]}) and Lemma \ref{lem: strong_convexity},
\begin{multline}
\label{gossip oy_k}
\bE[\|\oy_k\|^2]\le \frac{\sigma^2}{n}+\bE[\|h(\mx_k)-\nabla f(\ox_k)+\nabla f(\ox_k)\|^2]
\le \frac{\sigma^2}{n}+\bE\left[\left(\frac{L}{\sqrt{n}}\|\mx_k-\mathbf{1}\ox_k\|+L\|\ox_k-x^*\|\right)^2\right]\\
\le \frac{\sigma^2}{n}+\frac{2L^2}{n}\bE[\|\mx_k-\mathbf{1}\ox_k\|^2]+2L^2\bE[\|\ox_k-x^*\|^2].
\end{multline}	
We obtain
\begin{multline*}
\bE[\|\mx_{k+1}-\mathbf{1}\ox_{k+1}\|^2]
\le\left(\rho_{\bar{w}}+\frac{2\alpha}{n}\beta_1+\frac{8\alpha^2 L^2}{n}\right)\bE[\|\mx_k-\mathbf{1}\ox_k\|^2]
+\frac{2\alpha}{n}\left(\frac{1}{\beta_1}+\alpha\right)\bE[\|\my_k-\mathbf{1}\oy_k\|^2]\\
+8\alpha^2 L^2\bE[\|\ox_k-x^*\|^2]+\frac{4\alpha^2 \sigma^2}{n},
\end{multline*}
which is exactly inequality (\ref{gossip: second main inequality}).

Finally we prove inequality (\ref{gossip: third main inequality}). From the update rule (\ref{eq:y-update gossip compact}),
\begin{multline}
\label{gossip: third main inequality_pre1}
\|\my_{k+1}-\mathbf{1}\oy_{k+1}\|^2=\|\mathbf{W}_k \my_k+\tilde{\mathbf{D}}_k G_{k+1}-\tilde{\mathbf{D}}_k G_k-\mathbf{1}\oy_k+\mathbf{1}\oy_k-\mathbf{1}\oy_{k+1}\|^2\\
=\|\mathbf{W}_k\my_k-\mathbf{1}\oy_k\|^2+\|\tilde{\mathbf{D}}_k G_{k+1}-\tilde{\mathbf{D}}_k G_k\|^2+n\|\oy_k-\oy_{k+1}\|^2+2\langle \mathbf{W}_k\my_k-\mathbf{1}\oy_k,\tilde{\mathbf{D}}_k G_{k+1}-\tilde{\mathbf{D}}_k G_k\rangle\\
+2\langle \mathbf{W}_k\my_k-\mathbf{1}\oy_k,\mathbf{1}\oy_k-\mathbf{1}\oy_{k+1}\rangle+2\langle \tilde{\mathbf{D}}_k G_{k+1}-\tilde{\mathbf{D}}_k G_k,\mathbf{1}\oy_k-\mathbf{1}\oy_{k+1}\rangle\\
=\trace\left[(\my_k-\mathbf{1}\oy_k)^{\T}\mathbf{W}_k^{\T}\mathbf{W}_k(\my_k-\mathbf{1}\oy_k)\right]+\|\tilde{\mathbf{D}}_k G_{k+1}-\tilde{\mathbf{D}}_k G_k\|^2+n\|\oy_k-\oy_{k+1}\|^2\\
+2\langle \mathbf{W}_k\my_k-\mathbf{1}\oy_k,\tilde{\mathbf{D}}_k G_{k+1}-\tilde{\mathbf{D}}_k G_k\rangle+2\langle \tilde{\mathbf{D}}_k G_{k+1}-\tilde{\mathbf{D}}_k G_k,\mathbf{1}\oy_k-\mathbf{1}\oy_{k+1}\rangle.
\end{multline}
Denote $g_{i,k}:=g_i(x_{i,k},\xi_{i,k})$ for short. Notice that
\begin{equation*}
2\langle \tilde{\mathbf{D}}_k G_{k+1}-\tilde{\mathbf{D}}_k G_k,\mathbf{1}\oy_k-\mathbf{1}\oy_{k+1}\rangle=2\langle g_{i_k,k+1}-g_{i_k,k}+\ik(g_{j_k,k+1}-g_{j_k,k}) ,\oy_k-\oy_{k+1}\rangle,
\end{equation*}
and
\begin{equation*}
\oy_k-\oy_{k+1}=\frac{1}{n}\left(y_{i_k,k}-y_{i_k,k+1}+\ik(y_{j_k,k}-y_{j_k,k+1})\right)=-\frac{1}{n}\left(g_{i_k,k+1}-g_{i_k,k}+\ik(g_{j_k,k+1}-g_{j_k,k})\right).
\end{equation*}
We have
\begin{equation*}
2\langle \tilde{\mathbf{D}}_k G_{k+1}-\tilde{\mathbf{D}}_k G_k,\mathbf{1}\oy_k-\mathbf{1}\oy_{k+1}\rangle=-2n\|\oy_k-\oy_{k+1}\|^2.
\end{equation*}
Hence relation (\ref{gossip: third main inequality_pre1}) leads to
\begin{multline}
\label{y_k+1-oy_k+1 pre}
\|\my_{k+1}-\mathbf{1}\oy_{k+1}\|^2
\le \trace\left[(\my_k-\mathbf{1}\oy_k)^{\T}\mathbf{W}_k^{\T}\mathbf{W}_k(\my_k-\mathbf{1}\oy_k)\right]+\|\tilde{\mathbf{D}}_k G_{k+1}-\tilde{\mathbf{D}}_k G_k\|^2\\
+2\langle \mathbf{W}_k\my_k-\mathbf{1}\oy_k,\tilde{\mathbf{D}}_k G_{k+1}-\tilde{\mathbf{D}}_k G_k\rangle.
\end{multline}
Now we analyze the three terms on the right-hand side of (\ref{y_k+1-oy_k+1 pre}), respectively.
First,
\begin{equation}
\label{my_k-1 oy_k WW bound}
\bE[\trace[(\my_k-\mathbf{1}\oy_k)^{\T}\mathbf{W}_k^{\T}\mathbf{W}_k(\my_k-\mathbf{1}\oy_k)]\mid \mathcal{F}_{k+1}]\le \rho_{\bar{w}}\|\my_k-\mathbf{1}\oy_k\|^2.
\end{equation}
Second,
\begin{equation}
\label{D_kG_k+1-D_kG_k_pre}
\|\tilde{\mathbf{D}}_k G_{k+1}-\tilde{\mathbf{D}}_k G_k\|^2=\|g_{i_k,k+1}-g_{i_k,k}\|^2+\ik\|g_{j_k,k+1}-g_{j_k,k}\|^2.
\end{equation}	
In light of Assumption \ref{asp: gradient samples} and Assumption \ref{asp: strconvexity}, we can bound $\bE[\|g_{i_k,k+1}-g_{i_k,k}\|^2]$ and $\bE[\ik\|g_{j_k,k+1}-g_{j_k,k}\|^2]$:
\begin{align}
\label{bound: g_{i_k,k+1}-g_{i_k,k}}
\bE[\|g_{i_k,k+1}-g_{i_k,k}\|^2] =&\bE[\|\nabla f_{i_k}(x_{i_k,k+1})-\nabla f_{i_k}(x_{i_k,k})\|^2]\notag\\
&+2\bE[\langle \nabla f_{i_k}(x_{i_k,k+1})-\nabla f_{i_k}(x_{i_k,k}),g_{i_k,k+1}-\nabla f_{i_k}(x_{i_k,k+1})-g_{i_k,k}+\nabla f_{i_k}(x_{i_k,k})\rangle]\notag\\
&+\bE[\|g_{i_k,k+1}-\nabla f_{i_k}(x_{i_k,k+1})-g_{i_k,k}+\nabla f_{i_k}(x_{i_k,k})\|^2]\notag\\
=&\bE[\|\nabla f_{i_k}(x_{i_k,k+1})-\nabla f_{i_k}(x_{i_k,k})\|^2]+2\bE[\langle \nabla f_{i_k}(x_{i_k,k+1}), -g_{i_k,k}+\nabla f_{i_k}(x_{i_k,k})\rangle ]\notag\\
&+\bE[\|g_{i_k,k+1}-\nabla f_{i_k}(x_{i_k,k+1})\|^2]+\bE[\|g_{i_k,k}-\nabla f_{i_k}(x_{i_k,k})\|^2]\notag\\
\le & L^2\bE[\|x_{i_k,k+1}-x_{i_k,k}\|^2]+2\bE[\langle \nabla f_{i_k}(x_{i_k,k+1}), -g_{i_k,k}+\nabla f_{i_k}(x_{i_k,k})\rangle ]+2\sigma^2.
\end{align}
Similarly,
\begin{multline}
\label{bound: g_{j_k,k+1}-g_{j_k,k}}
\bE[\ik\|g_{j_k,k+1}-g_{j_k,k}\|^2]
\le L^2\bE[\ik\|x_{j_k,k+1}-x_{j_k,k}\|^2]+2\bE[\ik\langle \nabla f_{j_k}(x_{j_k,k+1}), -g_{j_k,k}+\nabla f_{j_k}(x_{j_k,k})\rangle ]+2\sigma^2.
\end{multline}
To further bound $\bE[\|g_{i_k,k+1}-g_{i_k,k}\|^2+\ik\|g_{j_k,k+1}-g_{j_k,k}\|^2]$, we introduce the following lemma.
\begin{lemma}
	\label{lem: gossip 3 bablas}
	\begin{equation*}
	\bE[\langle \nabla f_{i_k}(x_{i_k,k+1}), -g_{i_k,k}+\nabla f_{i_k}(x_{i_k,k})\rangle]+\bE[\ik\langle \nabla f_{j_k}(x_{j_k,k+1}), -g_{j_k,k}+\nabla f_{j_k}(x_{j_k,k})\rangle]
	\le 2\alpha L\sigma^2.
	\end{equation*}
\end{lemma}
The proof is similar to that of Lemma \ref{lem: 3 nablas} and is omitted here.
Equation (\ref{D_kG_k+1-D_kG_k_pre}) together with (\ref{bound: g_{i_k,k+1}-g_{i_k,k}}), (\ref{bound: g_{j_k,k+1}-g_{j_k,k}}) and Lemma \ref{lem: gossip 3 bablas} leads to the following inequality:
\begin{align}
\label{D_kG_k+1-D_kG_k_pre2}
& \bE[\|\tilde{\mathbf{D}}_k G_{k+1}-\tilde{\mathbf{D}}_k G_k\|^2]\notag\\
\le & L^2\bE[\|x_{i_k,k+1}-x_{i_k,k}\|^2]+L^2\bE[\ik\|x_{j_k,k+1}-x_{j_k,k}\|^2]+4(\alpha L+1)\sigma^2\notag\\
= & L^2\bE[\ik(\|x_{i_k,k+1}-x_{i_k,k}\|^2+\|x_{j_k,k+1}-x_{j_k,k}\|^2)]+L^2\bE[(1-\ik)\|x_{i_k,k+1}-x_{i_k,k}\|^2]+4(\alpha L+1)\sigma^2.
\end{align}
Note that from (\ref{eq:x-update gossip}) and (\ref{eq:x-update gossip2}), we have
\begin{align*}
& \ik(\|x_{i_k,k+1}-x_{i_k,k}\|^2+\|x_{j_k,k+1}-x_{j_k,k}\|^2)\\
= & \ik\left[\left\|\frac{1}{2}(x_{i_k,k}+x_{j_k,k})-\alpha y_{i_k,k}-x_{i_k,k}\right\|^2+\left\|\frac{1}{2}\left(x_{i_k,k}+x_{j_k,k}\right)-\alpha y_{j_k,k}-x_{j_k,k}\right\|^2\right]\\
= &\frac{1}{4}\ik\left(\|x_{j_k,k}-x_{i_k,k}-2\alpha y_{i_k,k}\|^2+\|x_{i_k,k}-x_{j_k,k}-2\alpha y_{j_k,k}\|^2\right)\\
= & \frac{1}{2}\ik\left[\|x_{i_k,k}-x_{j_k,k}\|^2+2\alpha^2\left(\|y_{i_k,k}\|^2+\|y_{j_k,k}\|^2\right)+2\alpha\langle x_{i_k,k}-x_{j_k,k}, y_{i_k,k}-y_{j_k,k} \rangle\right],
\end{align*}
and
\begin{equation*}
(1-\ik)\|x_{i_k,k+1}-x_{i_k,k}\|^2=4\alpha^2(1-\ik)\|y_{i_k,k}\|^2.
\end{equation*}
Then by (\ref{D_kG_k+1-D_kG_k_pre2}),
\begin{align}
\label{D_kG_k+1-D_kG_k_3}
&\bE[\|\tilde{\mathbf{D}}_k G_{k+1}-\tilde{\mathbf{D}}_k G_k\|^2]\notag\\
\le & L^2\bE\left[\ik\left(\frac{1}{2}\|x_{i_k,k}-x_{j_k,k}\|^2+\alpha^2\left(\|y_{i_k,k}\|^2+\|y_{j_k,k}\|^2\right)+\alpha\langle x_{i_k,k}-x_{j_k,k}, y_{i_k,k}-y_{j_k,k} \rangle\right)\right]\notag\\
& +4\alpha^2 L^2\bE[(1-\ik)\|y_{i_k,k}\|^2]+4(\alpha L+1)\sigma^2\notag\\
\le & L^2\bE\left[\ik\left(\|x_{i_k,k}-x_{j_k,k}\|^2+2\alpha^2\left(\|y_{i_k,k}\|^2+\|y_{j_k,k}\|^2\right)\right)\right]+4\alpha^2 L^2\bE[(1-\ik)\|y_{i_k,k}\|^2]+4(\alpha L+1)\sigma^2\notag\\
\le & \frac{L^2}{n}\bE\left[2\|\mI-\mPi\|\|\mx_k-\mathbf{1}\ox_k\|^2+4\alpha^2\|\my_k\|^2\right]+4(\alpha L+1)\sigma^2\notag\\
\le & \frac{4L^2}{n}\bE\left[\|\mx_k-\mathbf{1}\ox_k\|^2+\alpha^2\|\my_k-\mathbf{1}\oy_k\|^2+\alpha^2 n\|\oy_k\|^2\right]+4(\alpha L+1)\sigma^2.
\end{align}
For the last term on the right-hand side of (\ref{y_k+1-oy_k+1 pre}), note that
\begin{equation*}
2\langle \mathbf{W}_k\my_k-\mathbf{1}\oy_k,\tilde{\mathbf{D}}_k G_{k+1}-\tilde{\mathbf{D}}_k G_k\rangle=\langle y_{i_k,k}+y_{j_k,k}-2\oy_k,g_{i_k,k+1}-g_{i_k,k}+\ik(g_{j_k,k+1}-g_{j_k,k}) \rangle.
\end{equation*}
By Assumption \ref{asp: gradient samples}, we have
\begin{multline}
\label{W_k y_k, D_k G_k+1-D_k G_k}
2\bE[\langle \mathbf{W}_k\my_k-\mathbf{1}\oy_k,\tilde{\mathbf{D}}_k G_{k+1}-\tilde{\mathbf{D}}_k G_k\rangle]\\
=\bE[\langle y_{i_k,k}+y_{j_k,k}-2\oy_k, \nabla f_{i_k}(x_{i_k,k+1})-g_{i_k,k}+\ik(\nabla f_{j_k}(x_{j_k,k+1})-g_{j_k,k}) \rangle].
\end{multline}	
The following lemma is useful.
\begin{lemma}
	\label{lem: gossip 4 nablas}
	\begin{equation*}
	\bE[\langle y_{i_k,k}+y_{j_k,k}-2\oy_k, g_{i_k,k}-\nabla f_{i_k}(x_{i_k,k})+\ik(g_{j_k,k}-\nabla f_{j_k}(x_{j_k,k})) \rangle]\ge 0.
	\end{equation*}
\end{lemma}
\begin{proof}
	See Appendix \ref{appendix: proof lem: gossip 4 nablas}.
\end{proof}
In light of (\ref{W_k y_k, D_k G_k+1-D_k G_k}) and Lemma \ref{lem: gossip 4 nablas},
\begin{align}
\label{W_k y_k, D_k G_k+1-D_k G_k 2}
& 2\bE[\langle \mathbf{W}_k\my_k-\mathbf{1}\oy_k,\tilde{\mathbf{D}}_k G_{k+1}-\tilde{\mathbf{D}}_k G_k\rangle]\notag\\
\le & \bE[\langle y_{i_k,k}+y_{j_k,k}-2\oy_k, \nabla f_{i_k}(x_{i_k,k+1})-\nabla f_{i_k}(x_{i_k,k})+\ik[\nabla f_{j_k}(x_{j_k,k+1})-\nabla f_{j_k}(x_{j_k,k})] \rangle]\notag\\
=  & \bE[\ik\langle y_{i_k,k}+y_{j_k,k}-2\oy_k, \nabla f_{i_k}(x_{i_k,k+1})-\nabla f_{i_k}(x_{i_k,k})+\nabla f_{j_k}(x_{j_k,k+1})-\nabla f_{j_k}(x_{j_k,k}) \rangle]\notag\\
& + 2\bE[(1-\ik)\langle y_{i_k,k}-\oy_k, \nabla f_{i_k}(x_{i_k,k+1})-\nabla f_{i_k}(x_{i_k,k})\rangle].
\end{align}
Notice that from (\ref{eq:x-update gossip}), (\ref{eq:x-update gossip2}) and Assumption \ref{asp: strconvexity},
\begin{align*}
& \ik\langle y_{i_k,k}+y_{j_k,k}-2\oy_k, \nabla f_{i_k}(x_{i_k,k+1})-\nabla f_{i_k}(x_{i_k,k})+\nabla f_{j_k}(x_{j_k,k+1})-\nabla f_{j_k}(x_{j_k,k}) \rangle\\
\le & L\ik (\|y_{i_k,k}-\oy_k\|+\|y_{j_k,k}-\oy_k\|)(\|x_{i_k,k+1}-x_{i_k,k}\|+\|x_{j_k,k+1}-x_{j_k,k}\|)\\
= & \frac{L}{2}\ik(\|y_{i_k,k}-\oy_k\|+\|y_{j_k,k}-\oy_k\|)(\|x_{j_k,k}-x_{i_k,k}-2\alpha y_{i_k,k}\|+\|x_{i_k,k}-x_{j_k,k}-2\alpha y_{j_k,k}\|)\\
= & L\ik (\|y_{i_k,k}-\oy_k\|+\|y_{j_k,k}-\oy_k\|)\|x_{i_k,k}-x_{j_k,k}\|+\alpha L\ik(\|y_{i_k,k}-\oy_k\|+\|y_{j_k,k}-\oy_k\|)(\|y_{i_k,k}\|+\|y_{j_k,k}\|),
\end{align*}
and
\begin{multline*}
(1-\ik)\langle y_{i_k,k}-\oy_k, \nabla f_{i_k}(x_{i_k,k+1})-\nabla f_{i_k}(x_{i_k,k})\rangle\le L(1-\ik)\| y_{i_k,k}-\oy_k\| \|x_{i_k,k+1}-x_{i_k,k}\|\\
= 2\alpha L(1-\ik)\| y_{i_k,k}-\oy_k\|\|y_{i_k,k}\|.
\end{multline*}
We have from (\ref{W_k y_k, D_k G_k+1-D_k G_k 2}) that
\begin{align}
\label{W_k y_k, D_k G_k+1-D_k G_k 3}
& 2\bE[\langle \mathbf{W}_k\my_k-\mathbf{1}\oy_k,\tilde{\mathbf{D}}_k G_{k+1}-\tilde{\mathbf{D}}_k G_k\rangle]\notag\\
\le & L\bE[\ik (\|y_{i_k,k}-\oy_k\|+\|y_{j_k,k}-\oy_k\|)\|x_{i_k,k}-x_{j_k,k}\|]\notag\\
& +\alpha L\bE[\ik(\|y_{i_k,k}-\oy_k\|+\|y_{j_k,k}-\oy_k\|)(\|y_{i_k,k}\|+\|y_{j_k,k}\|)]+4\alpha L\bE[(1-\ik)\| y_{i_k,k}-\oy_k\|\|y_{i_k,k}\|)\notag\\
\le & \bE\left[\ik \beta_2(\|y_{i_k,k}-\oy_k\|^2+\|y_{j_k,k}-\oy_k\|^2)+\frac{L^2\ik}{2\beta_2}\|x_{i_k,k}-x_{j_k,k}\|^2\right]\notag\\
& +\alpha L\bE\left[\ik(\|y_{i_k,k}-\oy_k\|^2+\|y_{j_k,k}-\oy_k\|^2+\|y_{i_k,k}\|^2+\|y_{j_k,k}\|^2)\right]+2\alpha L\bE[(1-\ik)(\|y_{i_k,k}-\oy_k\|^2+\|y_{i_k,k}\|^2)]\notag\\
\le & \frac{2}{n}\beta_2\bE[\|\my_k-\mathbf{1}\oy_k\|^2]+\frac{L^2\|\mI-\mPi\|}{n}\frac{1}{\beta_2}\bE[\|\mx_k-\mathbf{1}\ox_k\|^2]+\frac{2\alpha L}{n}\bE[\|\my_k-\mathbf{1}\oy_k\|^2]+\frac{2\alpha L}{n}\bE[\|\my_k\|^2]\notag\\
= & \frac{2}{n}\beta_2\bE[\|\my_k-\mathbf{1}\oy_k\|^2]+\frac{2L^2}{n}\frac{1}{\beta_2}\bE[\|\mx_k-\mathbf{1}\ox_k\|^2]+\frac{4\alpha L}{n}\bE[\|\my_k-\mathbf{1}\oy_k\|^2]+2\alpha L\bE[\|\oy_k\|^2],
\end{align}
for any $\beta_2>0$. In light of (\ref{my_k-1 oy_k WW bound}), (\ref{D_kG_k+1-D_kG_k_3}) and (\ref{W_k y_k, D_k G_k+1-D_k G_k 3}), we obtain by (\ref{y_k+1-oy_k+1 pre}) that
\begin{align*}
& \bE[\|\my_{k+1}-\mathbf{1}\oy_{k+1}\|^2]\\
\le & \left(\rho_{\bar{w}}+\frac{2}{n}\beta_2+\frac{4\alpha L}{n}\right)\bE[\|\my_k-\mathbf{1}\oy_k\|^2]+\frac{2L^2}{n}\frac{1}{\beta_2}\bE[\|\mx_k-\mathbf{1}\ox_k\|^2]+2\alpha L\bE[\|\oy_k\|^2]\\
& +\frac{4L^2}{n}\left[\bE[\|\mx_k-\mathbf{1}\ox_k\|^2]+\alpha^2\bE[\|\my_k-\mathbf{1}\oy_k\|^2]+\alpha^2 n\bE[\|\oy_k\|^2]\right]+4(\alpha L+1)\sigma^2\\
=& \left(\rho_{\bar{w}}+\frac{2}{n}\beta_2+\frac{4\alpha L}{n}+\frac{4\alpha^2L^2}{n}\right)\bE[\|\my_k-\mathbf{1}\oy_k\|^2\mid\mathcal{F}_k]+\frac{L^2}{n}\left(4+\frac{2}{\beta_2}\right)\bE[\|\mx_k-\mathbf{1}\ox_k\|^2]\\
& +(4\alpha^2L^2+2\alpha L)\bE[\|\oy_k\|^2]+4(\alpha L+1)\sigma^2.
\end{align*}
Since by (\ref{gossip oy_k}),
\begin{equation*}
\bE[\|\oy_k\|^2]
\le \frac{\sigma^2}{n}+\frac{2L^2}{n}\bE[\|\mx_k-\mathbf{1}\ox_k\|^2]+2L^2\bE[\|\ox_k-x^*\|^2].
\end{equation*}
We conclude that
\begin{multline*}
\bE[\|\my_{k+1}-\mathbf{1}\oy_{k+1}\|^2]\le \left(\rho_{\bar{w}}+\frac{2}{n}\beta_2+\frac{4\alpha L}{n}+\frac{4\alpha^2L^2}{n}\right)\bE[\|\my_k-\mathbf{1}\oy_k\|^2]\\
+\frac{L^2}{n}\left(4+\frac{2}{\beta_2}+8\alpha^2 L^2+4\alpha L\right)\bE[\|\mx_k-\mathbf{1}\ox_k\|^2]+(8\alpha^2L^4+4\alpha L^3)\bE[\|\ox_k-x^*\|^2]+\frac{(4\alpha^2L^2+2\alpha L)\sigma^2}{n}\\
+4(\alpha L+1)\sigma^2.
\end{multline*}

\subsection{Proof of Lemma \ref{lem: gossip 4 nablas}}
\label{appendix: proof lem: gossip 4 nablas}
The following relation holds:
\begin{multline*}
\bE[\langle y_{i_k,k}+y_{j_k,k}-2\oy_k, g_{i_k,k}-\nabla f_{i_k}(x_{i_k,k})+\ik(g_{j_k,k}-\nabla f_{j_k}(x_{j_k,k})) \rangle]\\
=\bE[\ik\langle y_{i_k,k}+y_{j_k,k}-2\oy_k, g_{i_k,k}-\nabla f_{i_k}(x_{i_k,k})+g_{j_k,k}-\nabla f_{j_k}(x_{j_k,k})\rangle]
+2\bE[(1-\ik)\langle y_{i_k,k}-\oy_k, g_{i_k,k}-\nabla f_{i_k}(x_{i_k,k})\rangle].
\end{multline*}
From the updating rules (\ref{eq:y-update gossip}) and (\ref{eq:y-update gossip2}), we have
\begin{multline*}
\bE[\ik\langle y_{i_k,k}+y_{j_k,k}-2\oy_k, g_{i_k,k}-\nabla f_{i_k}(x_{i_k,k})+g_{j_k,k}-\nabla f_{j_k}(x_{j_k,k})\rangle]\\
=\left(1-\frac{2}{n}\right)\bE[\ik\langle (g_{i_k,k}+g_{j_k,k}), g_{i_k,k}-\nabla f_{i_k}(x_{i_k,k})+g_{j_k,k}-\nabla f_{j_k}(x_{j_k,k})\rangle]\\
=\left(1-\frac{2}{n}\right)\bE[\ik\|g_{i_k,k}-\nabla f_{i_k}(x_{i_k,k})\|^2+\ik\|g_{j_k,k}-\nabla f_{j_k}(x_{j_k,k})\|^2]\ge 0,
\end{multline*}
and
\begin{equation*}
\bE[(1-\ik)\langle y_{i_k,k}-\oy_k, g_{i_k,k}-\nabla f_{i_k}(x_{i_k,k})\rangle]=\left(1-\frac{1}{n}\right)\bE[(1-\ik)\|g_{i_k,k}-\nabla f_{i_k}(x_{i_k,k})\|^2]\ge 0.
\end{equation*}
Hence
\begin{equation*}
\bE[\langle y_{i_k,k}+y_{j_k,k}-2\oy_k, g_{i_k,k}-\nabla f_{i_k}(x_{i_k,k})+\ik(g_{j_k,k}-\nabla f_{j_k}(x_{j_k,k})) \rangle]\ge 0.
\end{equation*}


\bibliographystyle{spmpsci}      
\bibliography{mybib}

\end{document}